\documentclass[oneside,11pt]{article} 
\usepackage{makeidx}
\makeindex
\usepackage[utf8]{inputenc} 
\usepackage{amsthm}
\usepackage{amsmath}
\numberwithin{equation}{section}
\usepackage[colorlinks=true, linkcolor=blue, citecolor=blue, urlcolor=blue]{hyperref}
\usepackage{url}
\usepackage{xcolor}
\usepackage{geometry} 
\usepackage{graphicx} 
\usepackage{booktabs} 
\usepackage{array} 
\usepackage{paralist} 
\usepackage{verbatim}
\usepackage{subfig} 
\usepackage{authblk}
\usepackage{amsmath}
\usepackage{amsfonts}
\usepackage{amssymb}
\usepackage{amsthm}
\usepackage{mathrsfs} 
\usepackage{fancyhdr} 
\numberwithin{equation}{section} 
\usepackage{sectsty}
\allsectionsfont{\sffamily\mdseries\upshape} 
\usepackage[nottoc,notlof,notlot]{tocbibind} 
\usepackage[titles,subfigure]{tocloft}

\theoremstyle{remark}
\newtheorem{remark}{Remark}
\theoremstyle{theorem}
\newtheorem{theorem}{Theorem}
\theoremstyle{proposition}
\newtheorem{prop}{Proposition}
\theoremstyle{lemma}
\newtheorem{lemma}{Lemma}
\theoremstyle{defin}
\newtheorem{defin}{Definition}
\theoremstyle{hyp}
\newtheorem{hyp}{Assumption}
\theoremstyle{coro}

\theoremstyle{example}
\newtheorem{example}{Example}
\title{Entropic bounds for conditionally Gaussian vectors \\ and applications to neural networks}

\newcommand{\R}{\mathbb{R}}

\usepackage{etoolbox}
\author[1]{Lucia Celli}
\author[1]{Giovanni Peccati}
\affil[1]{Department of Mathematics, Luxembourg University}
\date{}
\begin{document}
\maketitle
\begin{abstract}
{ Using entropic inequalities from information theory, we provide new bounds on the total variation and 2-Wasserstein distances between a conditionally Gaussian law and a Gaussian law with invertible covariance matrix. We apply our results  to quantify the speed of convergence to Gaussian of a randomly initialized fully connected neural network and its derivatives --- evaluated in a finite number of inputs --- when the initialization is Gaussian and the sizes of the inner layers diverge to infinity. Our results require mild assumptions on the activation function, and allow one to recover optimal rates of convergence in a variety of distances, thus improving and extending the findings of Basteri and Trevisan (2023), Favaro {\it et al.} (2023), Trevisan (2024) and Apollonio {\it et al.} (2024). One of our main tools are the quantitative cumulant estimates established in Hanin (2024). {As an illustration, we apply our results to bound the total variation distance between the Bayesian posterior law of the neural network and its derivatives, and the posterior law of the corresponding Gaussian limit: this yields quantitative versions of a posterior CLT by Hron {\it et al.} (2022), and extends several estimates by Trevisan (2024) to the total variation metric}.
}
\end{abstract}

\tableofcontents
\section{Introduction and statement of the main results}
\subsection{Overview}\label{ss:overview}
{

{ The aim of this paper is to develop a general methodology to assess the discrepancy between the distribution of a Gaussian vector and that of a {\it conditionally Gaussian} random vector with the same dimension, using tools and concepts from information theory, see e.g. \cite{john, NPS}. Our main abstract estimates, proved by means of an interpolation technique inspired by the work of Trevisan \cite{Trev}, are stated in Theorems \ref{final_noC} and \ref{final} below.

\smallskip

As demonstrated in the sections to follow, our principal goal is to use our abstract bounds to quantitatively assess the fluctuations of \emph{randomly initialized fully connected neural networks} (see, e.g. \cite{Agg, RYH22, surveyNN, Ye}, as well as Definition \ref{def_NN}) by establishing quantitative versions of a seminal central limit theorem (CLT) by R. Neal \cite{Neal96, HanGa, HannonG,MHRTG,LBNSPS}, recalled in Theorem \ref{Hanin} below. As discussed in Sections \ref{ss:mainintro} and \ref{ss:litintro}, our findings allow one to deduce \emph{optimal Berry-Esseen bounds} for Neal's CLT, valid in any dimension and holding for total variation and Wasserstein-type distances \cite[Chapter 6]{villani}. Our bounds ({presented in Theorem \ref{fin_th_n}}) scale as the inverse of the network width, matching known lower bounds from \cite{FHMNP} in many cases --- see Remark \ref{optim_sp}. More broadly, our findings unify, improve, and generalize the collection of quantitative CLTs for fully connected neural networks recently established in \cite{Torr23, Btesi, BT24, FHMNP, Trev}. 

{Following \cite{Trev}, we also apply Theorem \ref{fin_th_n} to Bayesian inference, establishing a quantitative version of a key result in \cite{HBNPS}. Specifically, we bound the total variation distance between the exact posterior distribution of a neural network and that associated with its Gaussian limit, extending existing results to include network gradients. Theorem \ref{bay_TV} below provides an explicit bound on this distance.}

The content of Theorem \ref{fin_th_n} and Theorem \ref{bay_TV} is informally captured by the next statement, that we present for the reader's benefit. Precise definitions are given in Sections \ref{ss:introabstract} and \ref{def_NN_sec}.

\begin{theorem}[\bf Informal version of Theorems \ref{fin_th_n} and \ref{bay_TV}]\label{t:informal} 
Let $z^{(L+1)}$ be a fully connected feed-forward neural network with width $n$, fixed depth $L$, and Gaussian initialization. Then, as $n\to\infty$ and under an appropriate non-degeneracy assumption, the finite-dimensional marginal distributions of $z^{(L+1)}$ and of its gradients converge to a Gaussian limit both in the total variation and 2-Wasserstein distances, with a convergence rate of order $O\left(\frac1n\right)$. In the case of the network's marginals, the rate $\frac1n$ is optimal. 
An analogous quantitative CLT continues to hold for the posterior finite-dimensional distributions of $z^{(L+1)}$ and its gradients, provided the likelihood is bounded and continuous.
\end{theorem}

\smallskip

The rate $O(\frac{1}{n})$ in the $2$-Wasserstein distance for the network's marginals was already deduced in \cite{Trev}, whereas one-dimensional total variation bounds of the same order have been obtained in \cite{FHMNP}. We emphasize that---unlike in the case of bounds involving the convex distance---{\it multi-dimensional} estimates in total variation are typically \textit{not directly accessible} via coupling techniques (as those exploited in ~\cite{BT24, Btesi, Trev}) or via Stein's method, which remains the method of choice in~\cite{Torr23, FHMNP}. This limitation motivates the conceptually distinct, information-theoretic approach developed in the present work. See also the discussion contained in~\cite{Herry, NPS}.

From now on, we assume that every random element is defined on a common probability space $(\Omega, \mathcal{A}, \mathbb{P})$, with $\mathbb{E}$ denoting expectation with respect to $\mathbb{P}$. The following definition is standard and used throughout the paper.

\begin{defin}[\bf Conditionally Gaussian Vectors]\label{def_cond_ga}
Let $X$ be an integrable random vector with values in $\mathbb{R}^d$, $d\geq 1$, and assume that $\mathbb{E}[X]=0$. The vector $X$ is said to be \emph{conditionally Gaussian} with respect to a $\sigma$-field $\mathcal{F} \subseteq \mathcal{A}$ if there exists a positive semi-definite random matrix $A\in\mathbb{R}^{d\times d}$ (called \emph{conditional covariance matrix}) which is $\mathcal{F}$-measurable and such that, a.s.-$\mathbb{P}$,
\begin{equation}\label{cond_Ga_def}
\mathbb{E}\Big[e^{i\langle y,X\rangle}\big|\mathcal{F}\Big]=e^{-\frac{1}{2}\langle y,Ay\rangle},\quad \text{for every $y\in\mathbb{R}^d$}.
\end{equation}
\end{defin}

We will now state the main abstract results of our paper.}

}

\subsection{Main abstract bounds}\label{ss:introabstract}

{ To state our general results we need to introduce some standard probabilistic distances and discrepancies. A detailed discussion of their properties is provided in Section \ref{ss:results on distances}.
\begin{defin}[{Total variation distance}, see e.g. Appendix C in \cite{NP12}]
Given random vectors $X,Y$  with values in $\mathbb{R}^d$, the total variation distance (TV distance) between the laws of $X$ and $Y$ is defined as
\begin{multline}\label{dTV_def}
d_{TV}(X,Y):=\sup_{B\in\mathcal{B}(\mathbb{R}^d)}\Big|\mathbb{P}(X\in B)-\mathbb{P}(Y\in B)\Big|
=\frac{1}{2}\sup_{h\in \mathcal{M}_1}\Big|\mathbb{E}[h(X)]-\mathbb{E}[h(Y)]\Big|,
\end{multline}
where $\mathcal{B}(\mathbb{R}^d)$ is the Borel $\sigma$-field of $\mathbb{R}^d$ and {
\[
\mathcal{M}_1:=\{h:\mathbb{R}^d\to\R\quad\text{Borel measurable with}\quad \|h\|_{\infty}\le 1\}.
\]
}
\end{defin}

\begin{defin}[Convex distance,see \cite{Torr23, FHMNP, KG22}]
Given random vectors $X,Y$ with values in $\mathbb{R}^d$, the convex distance between the distributions of $X$ and $Y$ is defined as
\begin{equation}\label{dC_def}
d_{C}(X,Y):=\sup_{B\in\mathcal{C}(\mathbb{R}^d)}\Big|\mathbb{P}(X\in B)-\mathbb{P}(Y\in B)\Big|,
\end{equation}
where $\mathcal{C}(\mathbb{R}^d)$ is the class of all convex subsets of $\mathbb{R}^d$ (observe in particular that $d_C(X,Y)\le d_{TV}(X,Y)$).
\end{defin}

\begin{defin}[$p$-Wasserstein distance \cite{villani}]
Given $p\ge1$ and $X,Y$ random vectors with values in $\mathbb{R}^d$ and such that $\mathbb{E}[\|X\|^p],\mathbb{E}[\|Y\|^p]<\infty$, the $p$-Wasserstein distance between the laws of $X$ and $Y$ is defined as 
\begin{equation}\label{dW_def}
W_p(X,Y):=\inf \mathbb{E}[\|Z-W\|^p]^{1/p},
\end{equation}
 where the infimum is over all pairs $(Z,W)$ such that $Z\sim X$ and $W\sim Y$.
\end{defin}

For any random vectors $X,Y$ taking values in $\mathbb{R}^d$, $d\ge 1$, one can define the \emph{relative entropy} (or \emph{Kullback-Leibler divergence}) of $Y$ with respect to $X$, whenever the law of $Y$ is absolutely continuous with respect to the law of $X$. 
\begin{defin}[Relative entropy \cite{john}]\label{def_RE}
For $X,Y$ as above let $\nu_X$, $\nu_Y$ denote, respectively, the laws of $X$ and $Y$. Writing $\frac{d\nu_Y}{d\nu_X}$ to indicate the density of the law of $Y$ with respect to the law of $X$, we define the \emph{relative entropy} of the law of $Y$ with respect to the law of $X$ to be the quantity
\[{
D(Y||X):=\int_{\mathbb{R}^d}\log\Bigg(\frac{d\nu_Y}{d\nu_X}(z)\Bigg)\nu_Y(dz)=\mathbb{E}\Bigg[\frac{d\nu_Y}{d\nu_X}(X)\log\Bigg(\frac{d\nu_Y}{d\nu_X}(X)\Bigg)\Bigg],}
\]
with the convention that $0\log0=0$.
\end{defin}

We will see below that the relative entropy allows one to control the TV and 2-Wasserstein distances between two vectors, through the well-known {\it Pinsker-Csiszar-Kullback} and {\it Talagrand's} inequalities (see, respectively, Theorem \ref{PCK} and Theorem \ref{tala}). 

\smallskip 

Our first result is a general bound (Theorem \ref{final_noC}) on the relative entropy between a conditionally Gaussian law and a Gaussian law, under some conditions that ensure absolute continuity. 
\begin{hyp}\label{hip_g_cond}
\emph{ Fix $d\in\mathbb{N}$, and consider the following situation:
\begin{itemize}
\item $G$ is a random variable such that $G\sim\mathcal{N}_d(0,K)$ with values in $\mathbb{R}^d$ and with $K\in\mathbb{R}^{d\times d}$ invertible, 
\item $F$ is a random variable with values in $\mathbb{R}^d$ and $\mathcal{F}$ is a $\sigma$-field such that $F$ is conditionally Gaussian with respect to $\mathcal{F}$, with conditional covariance matrix $A\in\mathbb{R}^{d\times d}$ (see Definition \ref{def_cond_ga}).
\end{itemize}}
\end{hyp}
In what follows, we will denote by $\|A\|_{HS}$ the Hilbert-Schmidt norm of a matrix $A$. See Section \ref{ss:notation} for a detailed presentation of our notational conventions.

\begin{theorem}\label{final_noC}
 Fix $d\in\mathbb{N}$ and let Assumption \ref{hip_g_cond} prevail.
If moreover $\mathbb{E}[\|A\|_{HS}^8]<\infty$,  $\mathbb{P}(\det A>0)=1$ and $\mathbb{E}[\|A^{-1}\|_{HS}^2]<\infty$, then
\[
D(F||G)\le
C_1\|\mathbb{E}[A]-K\|_{HS}^2
+C_2\mathbb{E}\Big[\|A-K\|_{HS}^8\Big]^{1/2},
\]
where $C_1$ and $C_2$ are two explicit constants that depend on $d,K$ and $A$ (see Theorem \ref{final} for analytic expressions).
\end{theorem}
The requirements in Theorem \ref{final_noC} may be too restrictive for applications. As a consequence,  we will also derive bounds (stated in Theorem \ref{th_gen_tv_2w}) on the total variation and the 2-Wasserstein distances that hold under less stringent assumptions. The proof is based on the already recalled {Pinsker-Csiszar-Kullback} and {Talagrand's} inequalities. 
\begin{theorem}\label{th_gen_tv_2w}
Fix $d\in\mathbb{N}$, and let Assumption \ref{hip_g_cond} prevail, with
$\mathbb{E}[\|A\|_{HS}^8] <\infty$.
 Then,
\begin{equation}\label{e:zz}
\max\Big\{d_{TV}(F,G),W_2(F,G)\Big\}\le C_3\|\mathbb{E}[A]-K\|_{HS}+C_4\mathbb{E}[\|A-K\|_{HS}^8]^{1/4},
\end{equation}
where $C_3>0$ and $C_4>0$ are two explicit constants that depend on $d$ and $K$ (see Theorem \ref{th_dis_nuo} for precise expressions).
\end{theorem}
{
\begin{remark}\label{remark_cov_no_inv}
    Inspecting the proof of Proposition 5.9 in \cite{FHMNP} and noting that --- as the convex distance --- the total variation distance is invariant under orthogonal transformations, one can see that the assumption of $K$ being invertible can be removed when $K=\mathbb{E}[A]$. In this case, one can deduce a bound analogous to the right-hand side of \eqref{e:zz}, with $\|\mathbb{E}[A]-K\|_{HS}=0$, and a constant $C_4$ continuously depending on the rank and on the minimum nonzero and maximum eigenvalues of $K = \mathbb{E}[A]$. 
\end{remark}
}
\begin{remark}
    The proofs of Theorem \ref{final_noC} and Theorem \ref{th_gen_tv_2w} do not rely on the well-known {\it De Bruijn's identity} \cite[Theorem C.1]{john}, in contrast to reference \cite{NPS}, where the authors study the relative entropy between a Gaussian law and the law of a random vector with components in a Wiener chaos. We will see in the forthcoming sections that --- differently from \cite{NPS} --- our approach leads to bounds without logarithmic corrections, and that these bounds will be shown to be optimal in many instances.
    Another crucial methodological aspect is that the proof of Theorem \ref{th_gen_tv_2w} allows one to apply entropic bounds (via Theorem \ref{PCK} and Theorem \ref{tala}) without using conditioning techniques. In this way, one is able to deduce estimates featuring the term $\|\mathbb{E}[A]-K\|_{HS}$ instead of $\mathbb{E}[\|A-K\|_{HS}]$, which would have yielded less efficient bounds in our applications to neural networks --- see e.g. Theorem \ref{fin_th_n}.
\end{remark}

\begin{remark}
An alternative to using the Pinsker–Csiszár–Kullback inequality (Theorem \ref{PCK}) is the bound given in inequality (1.3) of \cite{bob25}, which involves the Rényi $\alpha$-divergence $D_\alpha$ for $0<\alpha<1$:
\[
\frac{\alpha}{2}d_{TV}(X,Y)^2\le D_\alpha(X||Y),
\]
where $X$ and $Y$ are random variables with distributions absolutely continuous with respect to a $\sigma$-finite measure $\mu$, and respective densities $f_X$ and $f_Y$. The $\alpha$-divergence is then defined as
\[
D_\alpha(X,Y):=\frac{1}{\alpha-1}\log E\Big[\Big(\frac{f_X(Y)}{f_Y(Y)}\Big)^\alpha\Big].
\]
Using the concavity of the function $x \mapsto x^\alpha$ and Jensen's inequality, one should be able to obtain bounds in the spirit of those established in the proofs of Theorems  \ref{final_noC} and \ref{th_gen_tv_2w}. A full treatment of this approach is beyond the scope of the present paper and will be pursued in future work.
\end{remark}

We will now introduce the collection of conditionally Gaussian objects that constitute the main motivation of our work, and to which Theorem \ref{th_gen_tv_2w} will be applied. }
\subsection{Neural networks as conditionally Gaussian objects}\label{def_NN_sec}

{ \emph{Deep neural networks} \cite{Agg, RYH22, Ye} are parametrized families of functions, at the heart of several recent advances in areas as diverse as structural biology \cite{Jumper-alphafold}, computer vision \cite{macchinechesiguidano} or language processing \cite{Brown-languagemodels}.  One of their typical uses is that of approximating an unknown function $f : \R^n \to \R^m$ (with $n$ and $m$ equal, respectively, to the input and output dimensions) starting from a so-called \emph{training data set} 
\begin{equation}\label{e:data}
\{ ( x^{(i)} , f(x^{(i)})) : i=1,...,p\},
\end{equation}
consisting of the values of $f$ at $p$ distinct points. Given the set \eqref{e:data}, one first selects a neural network architecture, which induces a parametric collection of mappings, and then searches within this collection for an approximation to $f$. In this article, we focus on the simple architecture of (feed-forward) \emph{fully connected networks}, whose formal definition is given below. See e.g. \cite[Chapter 6]{Ye} and \cite[Chapter 2]{RYH22} for a general introduction to these objects.
}
\smallskip

\smallskip 

{
\begin{defin}[\bf Fully Connected Neural Networks]\label{def_NN}
Fix integers $ L,n_0,n_{L+1}\ge 1$. A \emph{fully connected neural network (FCNN)} with \emph{depth} $L$, \emph{input dimension} $n_0$, \emph{output dimension} $n_{L+1}$, \emph{hidden layers widths} $n_1,...n_L\geq 1$ and \emph{non-linearity} (or \emph{activation function}) $\sigma:\mathbb{R}\to\mathbb{R}$ is a mapping of the form
$$
z^{(L+1)} : x = (x_1,...,x_{n_0})\mapsto z^{(L+1)}(x) = (z_1^{(L+1)}(x),...., z_{n_{L+1}}^{(L+1)}(x))  : \mathbb{R}^{n_0}\to \mathbb{R}^{n_{L+1}},
$$
defined recursively as follows
\begin{equation}\label{NN_expr}
\begin{cases}
z_{j}^{(1)}(x)=b_j^{(1)}+\sum_{k=1}^{n_{0}}\tilde{W}_{j,k}^{(1)}x_{k} &\text{for $j=1,\dots,n_1$, if $\ell=1$,}\\
z_{j}^{(\ell)}(x)=b_j^{(\ell)}+\sum_{k=1}^{n_{\ell-1}}\tilde{W}_{j,k}^{(\ell)}\sigma\big(z_{k}^{(\ell-1)}(x)\big) & \text{ for $j=1,\dots,n_\ell$, if $\ell=2,\dots,L+1$,}
\end{cases}
\end{equation}
where the \emph{trainable parameters} $b:=\{b_j^{(\ell)}\}_{j=1,\dots,n_\ell}^{\ell=1,\dots,L+1} $ and $\tilde{W}:=\{\tilde{W}_{j,k}^{(\ell)}\}_{j=1,\dots,n_\ell;k=1,\dots,n_{\ell-1}}^{\ell=1,\dots,L+1} $ are called, respectively, the \emph{biases} and the \emph{weights} of the neural network. For $\ell = 1,...,n_{L+1}$, we also use the following notation:
\begin{equation}\label{e:bnot}
b^{(\ell)} := (b_1^{(\ell)},..., b_{n_\ell}^{(\ell)} )\in \mathbb{R}^{n_\ell}, 
\end{equation}
and 
\begin{equation}\label{e:matrixnot}
\tilde{W}^{(\ell)}:=\left\{\tilde{W}_{j,k}^{(\ell)} : j=1,\dots,n_\ell, \,\, k=1,\dots,n_{\ell-1}\right\}\in \mathbb{R}^{n_\ell\times n_{\ell-1}} .
\end{equation}
\end{defin}
}
\smallskip

{When it is well-defined, the \emph{neural tangent kernel} (NTK) \cite{arora_2020, CC23, Jacot, RYH22} of $z^{(L+1)}$ is given by the mapping \begin{equation}\label{e:ntk}(x,y)\mapsto T_{L+1}(x,y) := \nabla z^{(L+1)}(x)\cdot \nabla z^{(L+1)}(y), \quad x,y\in \R^{n_0},\end{equation} where the gradient is considered with respect to the parameters $\Theta:= \{b,  \tilde{W}\}$, and the `dot' indicates an inner product in $\R^{|\Theta|}$. As explained e.g. in \cite{Agg, RYH22, Ye}, feed-forward FCNNs are among the basic building blocks of many network architectures used in practice --- their explaining power being a consequence of \emph{universal approximation theorems} \cite{Cybenko}. In general, given the training dataset \eqref{e:data} and an architecture such as \eqref{NN_expr}, the goal is to determine a configuration of the parameters $\Theta$ such that not only one has $ z^{(L+1)}(x) \approx f(x)$
for $x$ in the training set \eqref{e:data}, but also for inputs that do not belong to the training data. This optimization usually consists of two steps: (i) \emph{randomly initialize} the network trainable parameters (that is, sample $\Theta$ according to some multivariate probability distribution), and (ii) optimize the parameters by using some adequate variant of \emph{gradient descent} on an empirical loss such as the squared error \begin{equation}\label{e:sqloss} \sum_{i=1}^p\|z^{(L+1)}(x^{(i)})-f(x^{(i)})\|_{\R^{n_{L+1}}}^2 = \sum_{i=1}^p\|z^{(L+1)}(x^{(i)};\Theta)-f(x^{(i)})\|_{\R^{n_{L+1}}}^2,
\end{equation}
where on the right-hand side we have emphasized the dependency of the network on the trainable parameters $\Theta$ (with respect to which the optimization is realized). This yields optimization dynamics that can be directly expressed in terms of the {NTK} \eqref{e:ntk}, see \cite[Ch. 11]{Ye}. One should observe that the optimization problem described in \eqref{e:sqloss} is, in general, {\it highly non-convex}: as discussed e.g. in \cite{belkinhessian, belkinloss}, the fact that a global minimum is attained with overwhelming probability (when the parameter space dimension is sufficiently large), is explained by the specific geometry of the associated {\it loss landscapes}, which in turn emerges from the subsistence of some variation of the so-called {\it Polyak-\L{}ojasiewicz condition} \cite[Section 11.3]{Ye}. }

{In this work, we adopt one of the most popular forms of random initialization (sometimes called \emph{Le Cun initialization} and formally described in Assumption \ref{bW} below) consisting in sampling the trainable parameters $\Theta$ according to a multivariate centered Gaussian distribution with weight variances that are inversely proportional to the width of the network. In general, the rationale for randomly initializing neural biases and weights is to break the initial symmetry within the network: this ensures that each layer can learn unique features during training, as the optimization process will update the layers in distinct ways --- see e.g. \cite{weight}. 

From now on, for every $d\geq 1$ we will write $\mathcal{N}_d(m,\Sigma)$ to indicate a Gaussian law with expectation $m \in \mathbb{R}^d$ and covariance matrix $\Sigma\in\mathbb{R}^{d\times d}$ (note that, when $d=1$, one has that $m$ and $\Sigma$ are scalar). We also use the notation $X\sim Y$ to indicate that two random elements $X,Y$ have the same distribution; similarly, $X\sim \mu$ indicates that $X$ has law $\mu$. 
\begin{hyp}[\bf Random Gaussian Initialization]\label{bW}
\emph{Consider the FCNN defined in \eqref{NN_expr}. The parameters $\left\{b_j^{(\ell)}, \tilde{W}_{j,k}^{(\ell)}\right\} $ are mutually stochastically independent random variables such that, for every $\ell=1,\dots,L+1$, every $i=1,\dots,n_\ell$ and every $j=1,\dots,n_{\ell-1}$, one has that
\[
b_i^{(\ell)}\sim\mathcal{N}_1(0,C_b),
\]
\[
\tilde{W}_{i,j}^{(\ell)}\sim\mathcal{N}_1\Big(0,\frac{C_W}{n_{\ell-1}}\Big),
\]
with $C_b\ge0$ and $C_W> 0$. This implies in particular that $W_{i,j}^{(\ell)} :=\tilde{W}_{i,j}^{(\ell)}\times \frac{\sqrt{n_{\ell-1}}}{\sqrt{C_W}}\sim\mathcal{N}_1(0,1)$.
}
\end{hyp}
}
{We will also require some regularity properties on the non-linearity function $\sigma$ appearing in \eqref{NN_expr}. The following assumption, already used in \cite{FHMNP, HanGa}, is satisfied by most activations used in the literature, such as e.g., the Logistic Sigmoid, Tanh, ReLU, Swish and Mish (see e.g. \cite{activations}):

\begin{hyp}\label{hip_s}
\emph{There exists an integer $r\ge 1$ such that $\sigma$ is {either} $r$ times continuously differentiable, or it is $r-1$ times continuously differentiable and the $(r-1)$-derivative is a piece-wise linear function with a finite number of points of discontinuity for its derivative. Moreover there exists $k\ge 1$ s.t. 
\[
\sup_{x\in\mathbb{R}}\Big|(1+|x|)^{-k}\frac{d^r}{dx^r}\sigma(x)\Big|<\infty.
\]}
\end{hyp}

The following elementary statement shows that the neural network introduced in \eqref{NN_expr} defines a conditional Gaussian object, in the sense of Definition \ref{def_cond_ga}.

\begin{lemma}[\cite{HanGa}, Lemma 7.1]\label{gaus_stut_NN}
{Adopt the notation introduced in Definition \ref{def_NN}}, and let Assumption \ref{bW} prevail. Fix an integer $d\ge 1$, as well as inputs $\mathcal{X}:=\{x^{(1)},\dots,x^{(d)}\}\subseteq\mathbb{R}^{n_0}$, and define $\mathcal{F}_L$ to be the $\sigma$-field generated by $\{b^{(\ell)}, \tilde{W}^{(\ell)} : \ell=1,...,L\}$. For $i=1,...,n_{L+1}$, set 
\begin{equation}\label{vec_z}
z_i^{(L+1)}(\mathcal{X}):=(z_i^{(L+1)}(x^{(1)}),\dots, z_i^{(L+1)}(x^{(d)})).
\end{equation}
Then, one has that: {\rm (i)} conditionally on $\mathcal{F}_L$, the random vectors $z_i^{(L+1)}(\mathcal{X})$, $i=1,..., n_{L+1}$, are stochastically independent, and {\rm (ii)} each $z_i^{(L+1)}(\mathcal{X})$ is Gaussian conditionally on $\mathcal{F}_L$, in the sense of Definition \ref{def_cond_ga} and with a conditional covariance matrix  $A=A^{(L+1)}$ defined as follows: for $i,j=1,\dots, d$,
\begin{equation}\label{cov_z}
A^{(L+1)}_{i,j}:=A^{(L+1)}(x^{(i)},x^{(j)}):=
\begin{cases}
C_b+\frac{C_W}{n_{L}}\sum_{k=1}^{n_{L}}\sigma(z_k^{(L)}(x^{(i)}))\sigma(z_k^{(L)}(x^{(j)})), & \text{if $L\ge 1$}\\
C_b+\frac{C_W}{n_0}\sum_{k=1}^{n_0}x^{(i)}_{k}x^{(j)}_{k}, & \text{if $L=0$}.
\end{cases}
\end{equation}
\end{lemma}

{We observe that the case $L=0$ in \eqref{cov_z} corresponds to the covariance of the (Gaussian) field $z^{(1)}$ defined in \eqref{NN_expr}.} As argued in Remark \ref{gaus_cond_der_z}, the content of Lemma \ref{gaus_stut_NN} can be suitably extended to include the derivatives of $z^{(L+1)}$ with respect to the inputs. We will now explain how the content of Theorem \ref{th_gen_tv_2w} can be used to assess the fluctuations of large neural networks initialized as in Assumption \ref{bW}.

}
\subsection{Main results: tight bounds in large-width CLTs}\label{ss:mainintro}

{
In what follows, we will focus on the so-called {\it large-width analysis} of the network $z^{(L+1)}$ defined in \eqref{NN_expr}, obtained by fixing $L, n_0, n_{L+1}$ (depth and input/output dimensions) and letting $n_1,...,n_L\to \infty$. By doing so, the following two fundamental (and strictly related) phenomena emerge whenever the trainable parameters are initialized as in Assumption \ref{bW}: 
\begin{enumerate}
\item[\bf (A1)] The random field $z^{(L+1)}$ converges weakly to a $n_{L+1}$-dimensional Gaussian field with independent coordinates and a layer-wise recursively defined covariance structure \cite{Neal96, LBNSPS, HanGa, HannonG};
\item[\bf (A2)] The {neural tangent kernel} $T_{L+1}$ defined in \eqref{e:ntk}  converges (say, in probability) towards a deterministic mapping \cite{arora_2020, Jacot}. 
\end{enumerate}
\label{K}As discussed e.g. in \cite{arora_2020, Jacot, Ye}, the phenomenon described at Point {\bf (A2)} yields that, as the width diverges to infinity, with overwhelming probability the training of $z^{(L+1)}$ becomes indistinguishable from the optimization of a {\it linear model} (a situation sometimes referred to as ``lazy regime''). As a consequence, the central limit theorem (CLT) at Point {\bf (A1)} allows one to explicitly approximate the neural network after training as a deterministic affine transformation of the network at initialization, by applying classical formulae of {kernel regression} \cite[Chapter 2]{GPML}. 

\smallskip

\,

The CLT at Point {\bf (A1)} above --- first established in Neal's seminal paper \cite{Neal96} and then refined over more than two decades by several authors --- is the content of the next statement.
\begin{theorem}[\bf Large width CLT \cite{HanGa, HannonG, Neal96, MHRTG, LBNSPS}]\label{Hanin}
Fix $n_0,n_{L+1}$ and a smooth compact set $T\subseteq \mathbb{R}^{n_0}$. Let Assumption \ref{bW} and \ref{hip_s} prevail. As $n_1,\dots,n_{L} \to \infty $, the stochastic processes
\[
 T \ni x:=(x_{1},\dots,x_{n_0})\mapsto z^{(L+1)}(x)\in\mathbb{R}^{n_{L+1}}
\]
converge weakly in $C^{r-1}(T,\mathbb{R}^{n_{L+1}})$ to a centered Gaussian process $G^{(L+1)}$ taking values in $\mathbb{R}^{n_{L+1}}$ with independent and identically distributed coordinates. The coordinate-wise covariance function of $G^{(L+1)}$, defined for every $x^{(1)},x^{(2)}\in T$ as
\begin{equation}\label{lim_cov}
K_{1,2}^{(L+1)}:=K^{(L+1)}(x^{(1)},x^{(2)}):=\lim_{n_1,\dots,n_{L}\to\infty}Cov(z_{i}^{(L+1)}(x^{(1)}),z_{i}^{(L+1)}(x^{(2)}))
\end{equation}
satisfies the layer-wise recursion
\[
K_{1,2}^{(\ell)} :=K^{(\ell)}(x^{(1)},x^{(2)}) =C_b+C_W\mathbb{E}\Big[\sigma(G_1^{(\ell-1)}(x^{(1)}))\sigma(G_1^{(\ell-1)}(x^{(2)}))\Big],
\]
where (with obvious notation)
\[
\Big(G_1^{(\ell-1)}(x^{(1)}), G_1^{(\ell-1)}(x^{(2)})\Big)\sim\mathcal{N}_2\Bigg(0,
\begin{pmatrix}
K_{1,1}^{(\ell-1)} & K_{1,2}^{(\ell-1)}\\
K_{1,2}^{(\ell-1)} & K_{2,2}^{(\ell-1)}
\end{pmatrix}
\Bigg),
\]
for $\ell\ge 2$, with initial condition
\[
K_{1,2}^{(1)}:= K^{(1)}(x^{(1)},x^{(2)})=C_b+\frac{C_W}{n_0}\sum_{j=1}^{n_0}x^{(1)}_{j}x^{(2)}_{j}.
\]
\end{theorem}

In recent years, several authors have established quantitative versions of the CLT stated in Theorem \ref{Hanin}, both at the finite-dimensional and functional level --- see e.g. \cite{Torr23, Btesi, BT24, FHMNP, Trev}, as well as the forthcoming discussion. In what follows, we use Theorem \ref{th_gen_tv_2w} to deduce {\it tight bounds} in the TV and 2-Wasserstein distances for the finite-dimensional CLTs implied by Theorem \ref{Hanin}. As explained in Remark \ref{optim_sp}, our bounds are tight because they provide rates of convergence that scale as the inverse of the width of the network, yielding optimal rates of convergence in many situations. Our main findings, informally stated in Theorem \ref{t:informal} and collected in the forthcoming Theorem \ref{fin_th_n}, are preceded by a sequence of preliminary remarks and definitions of a (necessarily) technical nature.

\begin{remark}[See Remarks 2.3 and 2.6 in \cite{FHMNP}]\label{oss_FHMNP}
Under Assumptions \ref{bW} and \ref{hip_s}, for all $\ell=1,\dots,L+1$ the following properties hold true for the neural network $z^{(\ell)}$ defined in \eqref{NN_expr} and for its Gaussian limit $G^{(\ell)}$ introduced in Theorem \ref{Hanin}:
\begin{enumerate}[(i)]
\item $G^{(\ell)},z^{(\ell)}\in C^{r-1}(\mathbb{R}^{n_0};\mathbb{R}^{n_\ell})$ and $A^{(\ell)}\in C^{r-1,r-1}(\mathbb{R}^{n_0}\times \mathbb{R}^{n_0};\mathbb{R})$ with probability one, where $A^{(\ell)}$ is defined in \eqref{cov_z};
\item $z^{(\ell)}$, $G^{(\ell)}$ and $A^{(\ell)}$ are $r$-times differentiable almost everywhere with probability one. Moreover, for every multi-index $I:=(i_1,\dots,i_{n_0})\in\mathbb{N}_0^{n_0}$ with

 $|I|:=i_1+\dots+i_{n_0}=r$, the mixed derivatives 
\[
D^{I}_xz^{(\ell)}(x)\quad\text{and}\quad D^{I}_xG^{(\ell)}(x)
\]
are well defined and finite with probability one for every $x\ne 0$, where
\begin{equation}\label{der_sem}
D^{I}_x:=\frac{\partial^{i_1}}{\partial x_1^{i_1}}\dots\frac{\partial^{i_{n_0}}}{\partial x_{n_0}^{i_{n_0}}};
\end{equation}
\item For all $x^{(i)},x^{(j)}\in\mathbb{R}^{n_0}$ and for all $I,J\in\mathbb{N}_0^{n_0}$ such that $|I|,|J|\le r-1$ one has that
\begin{equation}\label{prod_der_G}
\mathbb{E}[D_{x^{(i)}}^IG^{(\ell)}(x^{(i)}) \cdot D_{x^{(j)}}^JG^{(\ell)}(x^{(j)})]=D_{x^{(i)}}^ID_{x^{(j)}}^JK^{(\ell)}_{i,j},
\end{equation}
with 
\begin{equation}\label{e:matrix}
K_{i,j}^{(\ell)}:= K^{(\ell)}(x^{(i)}, x^{(j)})
\end{equation}
(similarly to \eqref{lim_cov}), and where we have used the convention that when $x^{(i)}=x^{(j)}$, for every enough regular function $f:\mathbb{R}^{n_0}\times \mathbb{R}^{n_0}\to\mathbb{R}$, we have 
\begin{equation}\label{conv_der_ugu}
D_{x^{(i)}}^ID_{x^{(j)}}^Jf(x^{(i)},x^{(j)})=D_{x}^ID_{y}^Jf(x,y)_{|_{x=y=x^{(i)}}}.
\end{equation}
Identity \eqref{prod_der_G} holds also when $|I|=r$ or $|J|=r$ under the hypothesis that $x^{(i)},x^{(j)}\in\mathbb{R}^{n_0}\setminus\{0\}$;
\item For all $x^{(i)},x^{(j)}\in\mathbb{R}^{n_0}$ and for all $I,J\in\mathbb{N}_0^{n_0}$ such that $|I|,|J|\le r$ one has that
\[
\mathbb{E}[D_{x^{(i)}}^ID_{x^{(j)}}^JA_{i,j}^{(\ell)}]=D_{x^{(i)}}^ID_{x^{(j)}}^J\mathbb{E}[A^{(\ell)}_{i,j}]
\]
where we have adopted a notational convention similar to \eqref{e:matrix}, provided we assume that the mixed derivatives $D_{x^{(i)}}^ID_{x^{(j)}}^JA_{i,j}^{(\ell)}$ are well defined and finite with probability one when $|I|=|J|=r$.
\end{enumerate}
\end{remark}

\begin{remark}{\rm As in \cite{FHMNP}, for an integer $p\ge1$, we will denote a generic set of $p$ directional derivative operators in $\mathbb{R}^{n_0}$ as
\begin{equation}\label{def_V_set}
V=\{V_1,\dots,V_p\}
\end{equation}
where, for every $j=1,\dots,p$, we implicitly assume that there exists a vector $v_j=(v_{j,1},\dots,v_{j,n_0})\in\mathbb{R}^{n_0}$ such that
\begin{equation}\label{def_v_sum}
V_j=\sum_{i=1}^{n_0}v_{j,i}\frac{\partial}{\partial x_i}.
\end{equation}
Given $x\in\mathbb{R}^{n_0}$ and a multi-index $J:=(j_1,\dots,j_p)\in\mathbb{N}_0^p$ we define
\begin{equation}\label{power_V}
V^J_y:={V_1^{j_1}\dots V_p^{j_p}}_{|_{x= y}},
\end{equation}
meaning that the derivatives are computed at $x$, ({with $V^0_i =$ identity, by convention}). Finally, for integers $q\ge 0$ and $p\geq 1$, define
\begin{equation}\label{set_der_q}
{\mathcal{M}^{(p)}_q:=\{J:=(j_1,\dots,j_p)\in\mathbb{N}_0^p  : |J|\le q\},}
\end{equation}
where
\begin{equation}\label{mod_ind}
|J|:=j_1+\dots+j_p
\end{equation}
is the size of the multi-index $J$. {Note that 
$\mathcal{M}^{(p)}_0 = \{ {\bf 0}\}$, where ${\bf 0}$ indicates the element of $\mathbb{N}_0^p$ with identical zero entries.}
}
\end{remark}

\begin{defin}[Definition 2.4 in \cite{FHMNP}]\label{non_deg}
Fix $\mathcal{X}:=\{x^{(1)},\dots,x^{(d)}\}\subseteq\mathbb{R}^{n_0}\setminus\{0\}$ and consider the infinite-width $d\times d$ covariance matrices $\{K^{(\ell)}\}_{\ell=1,\dots,L+1}$ defined in Theorem \ref{Hanin} through the convention \eqref{e:matrix} (considering Assumption \ref{hip_s}), as well as a finite set of $p$ directional derivative operators $V$ as in \eqref{def_V_set}. Then, $\{K^{(\ell)}\}_{\ell=1,\dots,L+1}$ is said to be \emph{non-degenerate} on $\mathcal{X}$ to the order $q\le r$ with respect to $V$ if for every $\ell=1,\dots,L+1$ the matrix
\[
\Big(V_{x^{(i)}}^{J^{(i)}}V_{x^{(j)}}^{J^{(j)}}K_{i,j}^{(\ell)}\Big)_{(x^{(i)},J^{(i)}),(x^{(j)},J^{(j)})\in \mathcal{X}\times\mathcal{M}^{(p)}_q}
\]
is invertible, where we have used \eqref{power_V} and \eqref{set_der_q} together with a convention analogous to \eqref{conv_der_ugu}.
\end{defin}

\begin{remark}\label{q_zero}
If $q=0$ then $\{K^{(\ell)}\}_{\ell=1,\dots,L+1}$ is non-degenerate to the order $0$ if $K^{(\ell)}$ is invertible for every $\ell=1,\dots,L+1$.
\end{remark}

\begin{remark}
    In \cite[Remark (a), Subsection 3.2]{FHMNP}, it is proved that, when the non-linearity is $\sigma(x):=ReLU(x):=\max\{0,x\}, C_b=0, C_W=2$ and $x\ne 0$ then the limiting covariance matrix $\{K^{(\ell)}\}_{\ell=1,\dots,L+1}$ is non-degenerate on $x$ both to order $0$
 and to order $1$ with respect to $V=\Big\{\frac{\partial}{\partial x_i}\Big\}$ for $i\in\{1,\dots,n_0\}$.
 If moreover $\|x\|=1$, in \cite[Remark (d), Subsection 3.3]{FHMNP}, the authors also prove that one can find a set of directional derivatives $V$ (non necessarily canonical) such that $\{K^{(\ell)}\}_{\ell=1,\dots,L+1}$ is again non-degenerate on $x$ to the order $1$ with respect to $V$.
 \end{remark}

\begin{hyp}\label{deriv_mix_A}
For every $x^{(i)},x^{(j)}\in\mathbb{R}^{n_0}$ and for every $I,J\in \mathbb{N}_0^{n_0}$ with $|I|=r$ or $|J|=r$ the mixed derivatives $D_{x^{(i)}}^ID_{x^{(j)}}^JA_{i,j}^{(\ell)}$ are well defined and finite with probability one for all $\ell=1,\dots, L+1$, where we have adopted notations \eqref{cov_z}, \eqref{der_sem} and \eqref{mod_ind} respectively for the definitions of $A^{(\ell)}$, $D^J_{x}$ and $|I|$.
\end{hyp}
 
The following statement is one of the main achievements of the present work. We will see that the proof combines Theorem \ref{th_gen_tv_2w} with the content of Remark \ref{oss_FHMNP} as well as the forthcoming Remarks \ref{gaus_cond_der_z}, \ref{norm_8}, \ref{trA_fin}, and Proposition \ref{EAmK}. 
}
{
\begin{theorem}\label{fin_th_n}
Let Assumptions \ref{bW}, \ref{hip_s} and \ref{deriv_mix_A} prevail, and {fix $q\in\{0,1,..., r\}$}. Fix $\mathcal{X}:=\{x^{(1)},\dots,x^{(d)}\}\subseteq \mathbb{R}^{n_0}\setminus \{0\}$, a set of $p$ directional derivative operators $V:=\{V_1,\dots,V_p\}$ as in notation \eqref{def_V_set}, and a set of multi-indices $\{J^{(j)}\}_{j=1,\dots,d}$ {with $J^{(j)}\in\mathcal{M}_q^{(p)}$,}
for every $j=1,\dots,d$.
Assume that the matrix defined in \eqref{lim_cov}, $\{K^{(\ell)}\}_{\ell=1,\dots,L+1}$, is non-degenerate on $\mathcal{X}$ to the order $q$ with respect to $V$ as in Definition \ref{non_deg}. Then, if there exists $n\in\mathbb{N}$ such that 
{
\begin{equation}\label{bounds_n}
cn\le n_1, \dots n_{L}\le Cn
\end{equation}}
for some $c,C>0$ constants, and recalling the definitions in \eqref{dTV_def} and in \eqref{dW_def}, one has that
\begin{equation}\label{e:tvmain}
d_{TV}\Big(\big(V_{x^{(j)}}^{J^{(j)}}z_1^{(L+1)}(x^{(j)})\big)_{j=1,\dots,d},\big(V_{x^{(j)}}^{J^{(j)}}G_1^{(L+1)}(x^{(j)})\big)_{j=1,\dots,d}\Big)\le \frac{D_1}{n}  
\end{equation}
and
\begin{equation}\label{e:w2main}
{W}_2\Big(\big(V_{x^{(j)}}^{J^{(j)}}z_1^{(L+1)}(x^{(j)})\big)_{j=1,\dots,d},\big(V_{x^{(j)}}^{J^{(j)}}G_1^{(L+1)}(x^{(j)})\big)_{j=1,\dots,d}\Big)\le \frac{D_2}{n}, 
\end{equation}
where $D_1$ and $D_2$ are positive constants that do not depend on $n, n_1,\dots, n_{L}$.
\end{theorem}

}

We will now discuss the content of Theorem \ref{fin_th_n}.

\subsection{Remarks on Theorem \ref{fin_th_n}}

\begin{remark}\label{one_dim_eq}
    To keep the notational complexity within bounds, Theorem \ref{fin_th_n} only considers the distances between the first component of the neural network and its Gaussian limit. However, our results can be easily generalized to the case in which one considers the whole output. To see this, observe that the proof of Theorem \ref{fin_th_n} is based on the conditional Gaussianity of the neural network and on Theorem \ref{e:zz}, yielding bounds on the distances depending on the dimension of the random vectors, the minimum eigenvalue of the limiting covariance matrix and the norm of the difference between the covariance matrices and their expectation. Since the components of the output of the neural network (resp. of its limit) are conditionally independent and identically distributed (resp. independent and identically distributed), it follows that its conditional covariance matrix (resp. covariance matrix) has a block diagonal structure where every block on the main diagonal is given by $n_{L+1}$ copies of the conditional covariance matrix of $\big(V_{x^{(j)}}^{J^{(j)}}z_1^{(L+1)}(x^{(j)})\big)_{j=1,\dots,d}$ (resp. copies of $\big(V_{x^{(j)}}^{J^{(j)}}G_1^{(L+1)}(x^{(j)})\big)_{j=1,\dots,d}$). Starting from this observation, it is easy to suitably modify the proof of Theorem \ref{fin_th_n} and derive bounds analogous to \eqref{e:tvmain}--\eqref{e:w2main}, where the constants $D_1$ and $D_2$ now depend on $n_{L+1}$.
\end{remark}

{

\begin{remark}\label{r:quick comp} As discussed in the forthcoming Section \ref{ss:litintro}, the content of Theorem \ref{fin_th_n} substantially complements and extends the existing literature in the following sense: 
\begin{itemize}
\item[--] Bound \eqref{e:tvmain} generalizes and improves all available finite-dimensional bounds in the convex distance (see, e.g., \cite[Section 3.3]{FHMNP}, \cite{Torr23} { and Subsection \ref{subsec_dc} of the present paper for more details}) both by lifting them to the total variation setting, and by yielding convergence rates proportional to $\frac 1 n$, rather than to $\frac{1}{\sqrt{ n}}$.

\item[--] Bound \eqref{e:w2main} allows one to recover the optimal rates of convergence in the 2-Wasserstein distance established in \cite{Trev}  {(see subsection \ref{subsec_trev}) for a wider class of activation functions, including Lipschitz continuous functions}, and yields commensurate rates also for the (iterated) gradients of the network. 
\end{itemize}
    
\end{remark}

\begin{remark}

Thanks to Remark \ref{q_zero}, in order to apply Theorem \ref{fin_th_n} in the case $q=0$ (without derivatives) it is sufficient to assume that the limiting covariance matrices $K^{(\ell)}$ are invertible for every $\ell = 1, \dots, L+1$. This is not a restrictive assumption. In particular, Theorems 6 and 7 in \cite{CC24} provide conditions on the inputs of the neural network that ensure $K^{(\ell)}$ is strictly positive definite for all $\ell = 1, \dots, L+1$, under the assumption that the activation function $\sigma$ is continuous and non-polynomial:
\begin{itemize}
    \item[--] When $C_b \neq 0$, it is sufficient to assume that the inputs are all distinct.
    \item[--] When $C_b = 0$, it is sufficient to assume that the inputs are pairwise non-proportional.
\end{itemize}
\end{remark}

{
\begin{remark}
    Under the assumptions and notations of the previous theorem, we consider the case where the non-degeneracy condition on the sequence $\{K^{(\ell)}\}_{\ell=1,\dots,L+1}$ is not imposed. Instead, we assume that $\sigma$ is a smooth mapping, and that the matrix 
\begin{equation}\label{e:nonnull}
    B := \Big({V_{x^{(i)}}^{J^{(i)}}V_{x^{(j)}}^{J^{(j)}}K_{i,j}^{(L+1)}}\Big)_{i,j=1,\dots,d}
\end{equation}
is not the null matrix. Using Remark \ref{remark_cov_no_inv} and defining
    \[
    \tilde{G} \sim \mathcal{N}_d \left(0, \Big(\mathbb{E}\big[{V_{x^{(i)}}^{J^{(i)}}V_{x^{(j)}}^{J^{(j)}}A_{i,j}^{(L+1)}}\big]\Big)_{i,j=1,\dots,d} \right),
    \]
    we obtain the following bound:
    \begin{multline*}
        d_{TV} \Big(\big(V_{x^{(j)}}^{J^{(j)}}z_1^{(L+1)}(x^{(j)})\big)_{j=1,\dots,d},\tilde{G}\Big) \\
        \leq \tilde{C}_4 
        \mathbb{E}\Bigg[\Big\|\Big({V_{x^{(i)}}^{J^{(i)}}V_{x^{(j)}}^{J^{(j)}}A_{i,j}^{(L+1)}}\Big)_{i,j=1,\dots,d} - \Big(\mathbb{E}[{V_{x^{(i)}}^{J^{(i)}}V_{x^{(j)}}^{J^{(j)}}A_{i,j}^{(L+1)}}]\Big)_{i,j=1,\dots,d} \Big\|^8\Bigg]^{1/4},
    \end{multline*}
    where $\tilde{C}_4 > 0$ is a constant that continuously depends on the rank and the maximum and minimum positive eigenvalues of the matrix
    \[
    \tilde{A} := \Big(\mathbb{E}[{V_{x^{(i)}}^{J^{(i)}}V_{x^{(j)}}^{J^{(j)}}A_{i,j}^{(L+1)}}]\Big)_{i,j=1,\dots,d}.
    \]
Now denote by $\lambda_{+}(\tilde{A})$ and $\lambda_{+}(B)$ the minimum positive eigenvalues of $\tilde{A}$ and $B$, respectively. Using Theorem 4.5.3 (Weyl’s inequality) from \cite{HDP}, we obtain
    \[
    |\lambda_{+}(\tilde{A}) - \lambda_{+}(B)| \leq \|\tilde{A} - B\|_{op} \leq \|\tilde{A} - B\|_{HS} \leq \frac{D_1}{n},
    \]
    where the last inequality follows from Theorem \ref{coll_obs}, and $D_1$ is a constant independent of $n,n_1,\dots,n_L$. As a consequence, for every \( n \geq \frac{2D_1}{\lambda_{+}(B)} \), we have
    \[
    \lambda_{+}(\tilde{A}) \geq \lambda_{+}(B) - \frac{D_1}{n} \geq \frac{\lambda_{+}(B)}{2}.
    \]
Since we are assuming that \( \sigma \in C^{\infty}(\mathbb{R}) \), Theorem \ref{coll_obs} below along with the previous estimates yields that
    \[
    d_{TV} \Big(\big(V_{x^{(j)}}^{J^{(j)}}z_1^{(L+1)}(x^{(j)})\big)_{j=1,\dots,d},\tilde{G}\Big) \leq \frac{D_2}{n},
    \]
    where \( D_2 > 0 \) is a constant independent of \( n,n_1,\dots,n_L \).
\end{remark}
}

\begin{remark}[See Lemma 7.1 in \cite{HanGa}]\label{gaus_cond_der_z}
Exactly as in Lemma \ref{gaus_stut_NN}, under the assumptions of Theorem \ref{fin_th_n}, one can easily prove that conditionally on the $\sigma$-field $\mathcal{F}_{L}$ the vector of gradients $\left(V^{J^{(j)}}_{x^{(j)}} z_1^{(L+1)}(x^{(j)})\right)_{j=1,\dots,d}$ has a Gaussian law with covariance
\[
\mathbb{E}\Big[V^{J^{(i)}}_{x^{(i)}} z_1^{(L+1)}(x^{(i)})\cdot V^{J^{(j)}}_{x^{(j)}} z_1^{(L+1)}(x^{(j)})\big|\mathcal{F}_{L}\Big]=V_{x^{(i)}}^{J^{(i)}}V_{x^{(j)}}^{J^{(j)}}A_{i,j}^{(L+1)},
\]
where we have adopted a convention analogous to \eqref{conv_der_ugu}.
\end{remark}

\begin{remark}\label{optim_sp}
In \cite[Theorem 3.3]{FHMNP}, it is proved that, for $x\in\mathbb{R}^{n_0}$, under Assumptions \ref{bW} and \ref{hip_s}, supposing $K^{(\ell)}(x,x)\ne 0$ for every $\ell=1,\dots, L+1$ and $$\tilde{z}^{(L+1)}(x)\sim\mathcal{N}_1(0,\mathbb{E}[A^{(L+1)}(x,x)]),$$ then one has that
\begin{equation}\label{th_1_PM}
\min\Big\{W_1\big(z_1^{(L+1)}(x), \tilde{z}^{(L+1)}(x)\big),d_{TV}\big(z_1^{(L+1)}(x),\tilde{z}^{(L+1)}(x)\big)\Big\}\ge \frac{C_0}{n}
\end{equation}
where $C_0>0$ is a constant that does not depend on $n,n_1,\dots, n_{L}$. Now consider $\mathcal{X}:=\{x^{(1)},\dots, x^{(d)}\}\subseteq\mathbb{R}^{n_0}$, let the notations and assumptions of Theorem \ref{fin_th_n} prevail in the case $q=0$, and define
\[
\tilde{z}^{(L+1)}(\mathcal{X})\sim\mathcal{N}_d(0,\mathbb{E}[A^{(L+1)}]).
\]
Then, our findings imply that there exists a constant $C_5>0$ independent of $n,n_1,\dots,n_L$ such that
\begin{multline}\label{optim_bound}
\frac{C_5}{n}\ge \max\Big\{d_{TV}(z_1^{(L+1)}(\mathcal{X}),\tilde{z}^{(L+1)}(\mathcal{X})),d_{W}(z_1^{(L+1)}(\mathcal{X}),\tilde{z}^{(L+1)}(\mathcal{X}))\Big\}\\
\ge \min\Big\{d_{TV}(z_1^{(L+1)}(x^{(1)}),\tilde{z}^{(L+1)}(x^{(1)})),W_1 (z_1^{(L+1)}(x^{(1)}),\tilde{z}^{(L+1)}(x^{(1)}))\Big\}\ge \frac{C_0}{n},
\end{multline}
where: (i) the first bound follows from Theorem \ref{th_gen_tv_2w}, Lemma \ref{gaus_stut_NN} and  the forthcoming {Proposition \ref{EAmK}}, (ii) the second inequality is an elementary consequence of the definitions of $d_{TV}$ and $W_1$, and (iii) the third estimate follows from \eqref{th_1_PM}. The relation \eqref{optim_bound} shows in particular that, in the case $q=0$, the dependence on $n$ on the upper bounds established in Theorem \ref{fin_th_n} is optimal.
\end{remark}

\medskip

In the next section we show how to apply Theorem \ref{fin_th_n} to typical problems in Bayesian inference.

}

\subsection{Application to Bayesian deep neural networks}\label{sec_Bay}

{Consider a training dataset 
\[
\mathcal{D} = \{(x^{(i)}, y^{(i)})\}_{i=1,\dots,d} \subseteq \mathbb{R}^{n_0} \times \mathbb{R}^{n_{L+1}}
\]
where the labels satisfy
\begin{equation} \label{y_come_der}
    y^{(i)} = V_{x^{(i)}}^{J^{(i)}} f(x^{(i)}), \quad i=1,\dots,d,
\end{equation}
for some suitably regular function \( f: \mathbb{R}^{n_0} \to \mathbb{R}^{n_{L+1}} \). Here, $r\geq 1$,  $ 0 \leq  q \le r-1$, and  $ p \geq 1 $ are integers, the multi-indices \( \{J^{(i)}\}_{i=1,\dots,d} \) are elements of \( \mathcal{M}^{(p)}_q \) (as defined in \eqref{set_der_q}), and the operators $ \{V_{x^{(i)}}^{J^{(i)}}\}_{i=1,\dots,d} $ are defined as in \eqref{power_V}. The indices $r,p,q$ are fixed for the rest of the section.

We consider a family of neural networks as in Definition \ref{def_NN}, parameterized by the hyperparameters \( \Theta := \{b, \tilde{W}\} \) and with a non-linearity $\sigma$ obeying Assumption \ref{hip_s} for some $r\geq 1$. An alternative strategy to best approximate the labels \( \{y^{(i)}\}_{i=1,\dots,d} \) consists of adopting a Bayesian perspective, rather than the approach described in Section \ref{def_NN_sec}. This methodology, outlined e.g. in \cite{GPML,Trev,Fort22,HBNPS}, involves selecting a likelihood function \( \mathcal{L} \), which depends on \( \Theta \) and the training dataset \( \mathcal{D} \), and imposing a prior distribution on \( \Theta \), which in turn induces a prior distribution $\mu$ (that is, a prior law for \( z^{(L+1)}(\cdot;\Theta) \) and its derivatives) on the functional space associated with the network. Given the regularity assumptions on $\sigma$, without loss of generality we may regard the prior $\mu$ as a probability measure on the space $C^{r-1}(\mathbb{R}^{n_0}, \mathbb{R}^{n_{L+1}})$ (endowed with its Borel $\sigma$-field).

Once the prior distribution is fixed, the needed likelihood function can be introduced under the following assumption.

{
\begin{hyp} \label{hyp_L}
    We assume the following: {\rm (a)} conditionally on the network $z^{(L+1)}$ and its derivatives, the law of the vector
    \begin{equation}\label{zel}
{\bf u} := (y^{(1)}, \dots, y^{(d)})
\end{equation}
is absolutely continuous with respect to a fixed positive measure \( \nu_d \) on \( \mathbb{R}^{d \times n_{L+1}} \), and {\rm (b)} the distribution of the vector ${\bf u}$ conditionally on $\Theta$ coincides with the distribution of ${\bf u}$ conditionally on $z^{(L+1)}(\cdot; \Theta)$. 
\end{hyp}
}
\medskip 

Under Assumption \ref{hyp_L}, the {\it likelihood function associated with} $$ (x^{(1)},...,x^{(d)}, V^{J^{(1)}},...,V^{J^{(d)}})$$  is simply the density of the vector $(y^{(1)},...,y^{(d)})$ (with respect to \( \nu_d \) and evaluated in $(y^{(1)},...,y^{(d)})$), conditionally on $z^{(L+1)} =z\in C^{r-1}(\mathbb{R}^{n_0}, \mathbb{R}^{n_{L+1}})$. From now on, such a likelihood is written
\begin{equation} \label{e:likely}
\mathcal{L}(z; \mathcal{D}) = \mathcal{L}(z; \{(x^{(i)}, y^{(i)})\}_{i=1,\dots,d}), \quad z\in C^{r-1}(\mathbb{R}^{n_0}, \mathbb{R}^{n_{L+1}}).
\end{equation}

\begin{example}
    Consider the case $r=1$ (so that $q=0$), and assume that, conditionally on $\Theta$, the labels follow the noisy model
    \[
    y^{(i)} = z^{(L+1)}(x^{(i)}) + \varepsilon_i, \quad i = 1, \dots, d,
    \]
    where \( \{\varepsilon_i\}_{i=1,\dots,d} \) are i.i.d. standard Gaussian vectors in \( \mathbb{R}^{n_{L+1}} \). Let \( \nu_d \) be the Lebesgue measure on \( \mathbb{R}^{d \times n_{L+1}} \). In this case, the likelihood function is the density (with respect to \( \nu_d \)) of \( (y^{(1)}, \dots, y^{(d)}) \) conditionally on \( (z^{(L+1)}(x^{(1)}), \dots, z^{(L+1)}(x^{(d)})) \), which corresponds to a product of Gaussian densities with means \( z^{(L+1)}(x^{(i)}) \) and unit variances.
\end{example}

In accordance with Bayes’ theorem, the {\it posterior distribution} on the functional space $C^{r-1}(\mathbb{R}^{n_0}, \mathbb{R}^{n_{L+1}})$, written \( \mu_{|\mathcal{D}} \), is given by
\begin{eqnarray}\notag
d\mu_{|\mathcal{D}}(z) = \frac{\mathcal{L}(z, \{(x^{(i)}, y^{(i)})\}_{i=1,\dots,d})}{\int_{C^{r-1}(\mathbb{R}^{n_0}, \mathbb{R}^{n_{L+1}})} \mathcal{L}(z, \{(x^{(i)}, y^{(i)})\}_{i=1,\dots,d}) d\mu(z)} d\mu(z) \\ := \frac{\mathcal{L}(z, \{(x^{(i)}, y^{(i)})\}_{i=1,\dots,d})}{T} d\mu(z),\label{e:Z}
\end{eqnarray}
where the factor $\frac 1T$ (assuming $T>0$) ensures that \( \mu_{|\mathcal{D}} \) is a probability measure. This posterior distribution is then used to make predictions about the values of the gradients of the unknown function \( f \) at a new set of inputs 
\[
\mathcal{X}_{*} := (x^{(1)}_{*}, \dots, x^{(s)}_{*}) \in \mathbb{R}^{s \times n_0}, \quad s\geq 1.
\]
In this respect, an important role is played by the measure describing the law of the network at unseen inputs and unseen directional derivatives under the posterior distribution $\mu_{|\mathcal{D}}$, given by the mapping
\begin{multline} \label{pred}
    B\mapsto 
     \frac1T \mathbb{E}\Bigg[ 1_B \Big( V^{J_*^{(1)}}_{x^{(1)}_{*}} z^{(L+1)}(x^{(1)}_{*}), \dots, V^{J_*^{(s)}}_{x^{(s)}_{*}} z^{(L+1)}(x^{(s)}_{*}) \Big) 
    {\mathcal{L} \Big(  z^{(L+1)}; \mathcal{D}  \Big)}\Bigg]\\
    := \mathbb{P} \Big( (V^{J_*^{(1)}}_{x^{(1)}_{*}} z^{(L+1)}(x^{(1)}_{*}), \dots, V^{J_*^{(s)}}_{x^{(s)}_{*}} z^{(L+1)}(x^{(s)}_{*})) \in B \Big| \mathcal{X}_{*}, \mathcal{D} \Big),
\end{multline}
where $T$ is implicitly defined in \eqref{e:Z}, and $B$ is a Borel subset of $\mathbb{R}^{s\times n_{L+1}}$. See e.g. \cite{GPML, Trev}.
}

\medskip 

{ The convergence in law of a fully connected neural network to a Gaussian process $G^{(L+1)}$ as the inner widths tend to infinity (Theorem \ref{Hanin}) naturally raises the question of the limiting behavior of the posterior distribution. In \cite{HBNPS}, the authors addressed this problem under the assumption that the labels \( y^{(1)}, \dots, y^{(d)} \) depend on the parameters \( \Theta \) of the neural network and the inputs \( \mathcal{X} = \{x^{(1)}, \dots, x^{(d)}\} \) only through \( z^{(L+1)}(\mathcal{X}) \). Moreover, they assumed that Assumption \ref{hyp_L} holds with \( r=1 \), and that the likelihood function is given by a non-negative, bounded, and continuous mapping $\mathcal{L}(\bullet)$ computed on the vector \( z^{(L+1)}(\mathcal{X})\), where the definition of $\mathcal{L}$ does not depend on the inner widths. Under these conditions, it was shown in \cite{HBNPS} that if \( \mathbb{E}[\mathcal{L}(G^{(L+1)}(\mathcal{X}))] > 0 \), then the posterior distribution of the neural network induced by the dataset \( \mathcal{D}_0 := \{(x^{(i)}, f(x^{(i)}))\}_{i=1,\dots,d} \), denoted by \( {z^{(L+1)}}_{|\mathcal{D}_0} \), converges in law to the posterior distribution of its Gaussian process limit, denoted by \( {G^{(L+1)}}_{|\mathcal{D}_0} \), as the inner widths tend to infinity.

The first quantitative result on the convergence in law of posteriors was established in \cite{Trev}. There, the author assumed that the activation function \( \sigma \) is Lipschitz, that the limiting covariance matrix is non-degenerate of order \( q=0 \) on \( \mathcal{X} \) (as in Definition \ref{non_deg}), and that the likelihood function is Lipschitz continuous and satisfies the same assumptions as in \cite{HBNPS}. Under these conditions, the results proved in \cite[Section 5]{Trev} yield that, if \( n := \min\{n_1, \dots, n_L\} \) is sufficiently large, then
\[
W_1({z^{(L+1)}}_{|\mathcal{D}_0},{G^{(L+1)}}_{|\mathcal{D}_0})\le \frac{C}{n},
\]
where \( C>0 \) is a constant independent of the inner widths. If the non-degeneracy condition does not hold, the bound is of order \( \frac{1}{\sqrt{n}} \).
}

{
The following result, proved in the Appendix, does not require the likelihood function to be Lipschitz and extends the results presented in \cite{HBNPS,Trev} by including the derivatives of the neural network.} Note that, as in Theorem \ref{fin_th_n} and in order not to overcharge the notation, we only present bounds that involve the first coordinate of the vector $z^{(L+1)} = (z_1^{(L+1)},...,z_{n_{L+1}}^{(L+1)})$; reasoning as in Remark \ref{one_dim_eq} (now applied to the content of the forthcoming Theorem \ref{bay_TV}) one can see that our bounds can be immediately generalized to include the full network's output. {
\begin{theorem}\label{bay_TV}
Let Assumption \ref{hyp_L} prevail, and assume that the likelihood \eqref{e:likely} admits a version such that
$$
\mathcal{L}(z_1^{(L+1)} ; \mathcal{D}) = \mathcal{L}\Big(\Big(V^{J^{(j)}}_{x^{(j)}}z_1^{(L+1)}(x^{(j)})\Big)_{j=1,\dots,d}\Big),\quad z_1^{(L+1)}\in C^{r-1}(\R^{n_0}, \R),
$$
where $\mathcal{L} : \mathbb{R}^{d} \to \R$ is non-negative, bounded and continuous.
We assume that under the prior measure the assumptions of Theorem \ref{fin_th_n} are satisfied and define \[ Z:=\Big(V^{J^{(j)}}_{x^{(j)}}z_1^{(L+1)}(x^{(j)})\Big)_{j=1,\dots,d}\quad\text{and}\quad G:=\Big(V^{J^{(j)}}_{x^{(j)}}G_1^{(L)}(x^{(j)})\Big)_{j=1,\dots,d}. \] 
If $T= \mathbb{E}[\mathcal{L}(Z)]>0$ and $\mathbb{E}[\mathcal{L}(G)]>0$, then there exists a constant $D$ independent of $n,n_1,\dots,n_L$ such that
    \[
    d_{TV}({Z}_{|\mathcal{D}},{G}_{|\mathcal{D}})\le \frac{D\|\mathcal{L}\|_{\infty}}{\mathbb{E}[\mathcal{L}(G)]}\Bigg(1+\frac{\|\mathcal{L}\|_{\infty}}{\mathbb{E}[\mathcal{L}(Z)]}\Bigg)\frac{1}{n},
    \]
    where: {\rm (i)} ${Z}_{|\mathcal{D}}$ indicates a vector distributed according to the posterior law of $Z$, that is, the law of $Z$ under the measure $\mu_{|\mathcal{D}}$ defined in \eqref{e:Z}, and {\rm (ii)} ${G}_{|\mathcal{D}}$ is a vector distributed according to the posterior law of $G$, that is, according to the probability measure given by
$$
B\mapsto \frac{1}{\mathbb{E}[\mathcal{L}(G)]}\mathbb{E}[1_{B}(G)\mathcal{L}(G)], \quad B\in \mathcal{B}(\R^d).
$$
    
\end{theorem}
}
\smallskip

{
\begin{remark}
    Since $\mathcal{L}$ is continuous and bounded, Theorem \ref{Hanin} yields that
    \[
    \mathbb{E}[\mathcal{L}(Z)]\to \mathbb{E}[\mathcal{L}(G)]\quad\text{as $n\to\infty$}.
    \]
    It follows that, if $\mathbb{E}[\mathcal{L}(G)]>0$, there exists an $N\in\mathbb{N}$ such that $\mathbb{E}[\mathcal{L}(Z)]>0$ for every $n\ge N$. As a consequence, one can remove the assumption on the positivity of $\mathbb{E}[\mathcal{L}(Z)]$ from the previous statement, provided $n$ is sufficiently large.
\end{remark}}

{
\begin{remark} 
Fix new inputs $\mathcal{X}_{*}:=\{x^{(1)}_{*},\dots,x^{(s)}_{*}\}\subset \mathbb{R}^{n_0}$, with $s\ge 1$. For every Borel set $B\in\mathcal{B}(\mathbb{R}^{s})$, under the assumptions and notation of Theorem \ref{bay_TV} and of \eqref{pred}, one has that
    \begin{multline}\label{diff_prop_post}
   \Bigg| \mathbb{E}\Bigg[1_{B}\Big((V^{J_*^{(i)}}_{x^{(i)}_{*}}{z}_1^{(L+1)}(x^{(i)}_{*}))_{i=1,\dots,s}\Big)
\mathcal{L}\Big((V^{J^{(i)}}_{x^{(i)}}{z_1^{(L+1)}}(x^{(i)}))_{i=1,\dots,d}\Big)\Bigg]\frac{1}{\mathbb{E}[\mathcal{L}(Z)]} \\
- \mathbb{E}\Bigg[1_{B}\Big((V^{J_*^{(i)}}_{x^{(i)}_{*}}{G}_1^{(L+1)}(x^{(i)}_{*}))_{i=1,\dots,s}\Big)
\mathcal{L}\Big((V^{J^{(i)}}_{x^{(i)}}{G_1^{(L+1)}}(x^{(i)}))_{i=1,\dots,d}\Big)\Bigg]\frac{1}{\mathbb{E}[\mathcal{L}(G)]}\Bigg|\\
\le d_{TV}(\tilde{Z}_{|\mathcal{D}},\tilde{G}_{|\mathcal{D}}),
    \end{multline}
    where $\tilde{Z}_{|\mathcal{D}}$ is a $(d+s)$-dimensional vector with law proportional to
    \[
    \mathcal{L}\Big(z_1,\dots,z_d\Big)d\mu_1(z_1,\dots,z_d,z_{d+1},\dots,z_{d+s})
    \]
    and $\tilde{G}_{|\mathcal{D}}$ is a $(d+s)$-dimensional vector with law proportional to
    \[
    \mathcal{L}\Big(g_1,\dots,g_d\Big)d\mu_2(g_1,\dots,g_d,g_{d+1},\dots,g_{d+s}),
    \]
    where $\mu_1$ and $\mu_2$ are, respectively, the laws of
    \[
    \tilde{Z}:=\Big((V^{J^{(i)}}_{x^{(i)}}{z}_1^{(L+1)}(x^{(i)}))_{i=1,\dots,d},(V^{J_*^{(i)}}_{x_{*}^{(i)}}{z_1^{(L+1)}}(x_{*}^{(i)}))_{i=1,\dots,d+s})\Big)
    \]
    and of
    \[
    \tilde{G}:=\Big((V^{J^{(i)}}_{x^{(i)}}{G}_1^{(L+1)}(x^{(i)}))_{i=1,\dots,d},(V^{J_*^{(i)}}_{x_{*}^{(i)}}{G_1^{(L+1)}}(x_{*}^{(i)}))_{i=1,\dots,d+s})\Big).
    \]
    Assuming, as in the setting of Theorem \ref{fin_th_n}, that the limiting covariance matrices $$\{K^{(\ell)}\}_{\ell=1,\dots,L+1}$$ are non-degenerate on $\mathcal{X}_{*}\cup\mathcal{X}$ (with $\mathcal{X}:=\{x^{(1)},\dots,x^{(d)}\}$), to the order $q\in\{0,\dots,r\}$ with respect to $V$, one can rehearse the proof of Theorem \ref{bay_TV} to infer that 
    \[
    d_{TV}(\tilde{Z}_{|\mathcal{D}},\tilde{G}_{|\mathcal{D}})\le \frac{D\|\mathcal{L}\|^2_{\infty}}{\mathbb{E}[\mathcal{L}(G)]}\Bigg(1+\frac{\|\mathcal{L}\|_{\infty}}{\mathbb{E}[\mathcal{L}(Z)]}\Bigg)\frac{1}{n}.
    \]
    It follows from identity \eqref{pred} and inequality \eqref{diff_prop_post} that computing 
    \[
    \mathbb{P}\Big((V^{J_*^{(1)}}_{x^{(1)}_{*}}z^{(L+1)}(x^{(1)}_{*}),\dots,V^{J_*^{(s)}}_{x^{(s)}_{*}}z^{(L+1)}(x^{(s)}_{*}))\in B\Big|\mathcal{X}_{*},\mathcal{D}\Big)
    \]
    using the posterior of the Gaussian limit instead of that of the neural network yields an error scaling as $O(1/n)$ as $n$ diverges to infinity.
\end{remark}
}

{
\subsection{Literature review}\label{ss:litintro}

As already discussed, several recent papers have addressed the problem of bounding the distance between the law of a Gaussian FCNN (under various assumptions on the network architecture) and its Gaussian limit, as the inner widths tend to infinity. The techniques used are mainly probabilistic (Stein’s method and/or coupling) or based on optimal transport. In contrast, this paper addresses the problem using information-theoretic methods. {

We now compare existing results with the new contributions of this paper, focusing on the type of distance considered.

\subsubsection{Known bounds on the TV, Convex and in 1-Wasserstein distances}\label{subsec_dc}

In \cite{BFF24} the authors work under Assumption \ref{bW} (invertibility of the limiting covariance matrix $K^{(L+1)}$), assuming $L=1$ (shallow network) and $\sigma$ polynomially bounded along with its first two derivatives. As an application of the second-order Poincaré inequalities from \cite{Vid}, it is shown that the TV distance (for one input) and the Wasserstein distance (for multiple inputs) between the law of the network and its Gaussian limit are bounded by $\frac{1}{\sqrt{n_1}}$, where $n_1$ is the length of the inner layer of the network. For $L=2$, they obtained a slower convergence rate.

{ In \cite{Torr23}, the authors assumed Gaussian weights and biases as in Assumption \ref{bW} and general conditions on the activation function $\sigma$ that holds for example when $\sigma$ is a Lipschitz continuous function (see Proposition 5.2 in \cite{Torr23}):
\begin{itemize}
    \item $\forall a_1,a_2\ge 0$, and $C_b,C_W>0$, there exists a polynomial $P$, with non-negative coefficients depending only on $\sigma, C_b, C_W$ and with degree independent of $\sigma, a_1, a_2, C_b, C_W$, such that
    \[
    |\sigma(x\sqrt{C_b+C_Wa_1})^2-\sigma(x\sqrt{C_b+C_Wa_2})^2|\le P(|x|)|a_2-a_1|, \quad\text{for all $x\in\mathbb{R}$};
    \]
    \item For every $k\in\mathbb{R}$, $\mathbb{E}[\sigma^4(kZ)]<\infty$, where $Z\sim\mathcal{N}_1(0,1)$.
\end{itemize}
In \cite[Theorem 6.1 and Theorem 6.2]{Torr23} it is proved that for $L \ge 1$ and $x\in\mathbb{R}^{n_0}$ 
\begin{equation}\label{e:apo}
\max\Big\{W_1(z^{(L+1)}(x),G^{(L+1)}(x)), d_C(z^{(L+1)}(x),G^{(L+1)}(x))\Big\} \le \sum_{\ell=1}^{L} C_\ell \frac{1}{\sqrt{n_\ell}},
\end{equation}
where $n_1,...,n_L$ are the inner widths of the network and $C_\ell > 0$ is an explicit constant. The approach of \cite{Torr23} relies on the conditional Gaussianity of the network and on the use of Stein's method. Commensurate rates are established for the TV, Kolmogorov, and 1-Wasserstein distances, when $L=1$ and one considers a single input.}

In \cite{FHMNP}, as an application of Stein's method and of the estimates from \cite{HanGa}, the authors improved these bounds in several ways. Under Assumptions \ref{bW} and \ref{hip_s}, they showed that for a single input, both the 1-Wasserstein and TV distances are bounded above and below (in the case where no derivatives are involved --- see Remark \ref{optim_sp}) by quantities scaling $\frac{1}{n}$, yielding optimal convergence rates. The results proved in \cite{FHMNP} are valid in the general framework of Theorem \ref{Hanin}) and consider in particular the derivatives of the neural network and the corresponding Gaussian limit with respect to the input.

Extending the Stein's method approach from \cite{FHMNP} to multiple inputs, while maintaining the optimality of the rates, is challenging. The authors of \cite{FHMNP} analyzed the convex distance between the derivatives of the network and the Gaussian limit under a non-degeneracy condition on the limiting covariance (analogous to Definition \ref{non_deg}), achieving a sub-optimal $\frac{1}{\sqrt{n}}$ bound, when compared to the estimates in Theorem \ref{fin_th_n} above. When no derivatives are involved, the multi-dimensional bounds proved in \cite{FHMNP} roughly match the bound \eqref{e:apo} from \cite{Torr23}.

\subsubsection{The case of $W_p$ distances, for $p\ge 2$}\label{subsec_trev}

In \cite{BT24}, the authors studied a FCNN with a finite number of inputs, Gaussian weights (Assumption \ref{bW}), general depth $L \ge 1$, and a Lipschitz activation function. Using an inductive argument and properties of the 2-Wasserstein distance, they proved an explicit bound of order $\sum_{\ell=1}^{L}\frac{1}{\sqrt{n_\ell}}$ for the distance between the network and its Gaussian limit, without assuming the invertibility of the limiting covariance matrix.

In \cite{Trev}, the author considered a more general network architecture and improved the results of \cite{BT24}, obtaining an upper bound of optimal order $\sum_{\ell=1}^{L}\frac{1}{n_\ell}$ for the $p$-Wasserstein distance between the law of the network and of its Gaussian limit (for all $p \ge 1$), with a finite number of inputs. This result uses Assumption \ref{bW}, a Lipschitz non-linearity, and the invertibility of $K^{(\ell)}$ for all $\ell = 1, \dots, L+1$. 

{ The last two hypothesis on the non-linearity and on the limiting covariance matrix are respectively particular situations of Assumption \ref{hip_s} with $r=1$ and of the non-degeneracy condition in the case of no derivatives as in Definition \ref{non_deg}. This means that the speed of convergence of the order $\frac{1}{n}$ is recovered by Theorem \ref{fin_th_n} in the 2-Wasserstein distance, if $cn\le n_1,\dots, n_{L}\le Cn$ with $c,C>0$ constants. }

As already pointed out, the paper \cite{Trev} served as a key reference for establishing Theorems \ref{final_noC} and \ref{th_gen_tv_2w} in our work. In particular, we will see that our strategy of proof exploits the idea of partitioning the probability space into regions where the network has a density (allowing for a Taylor expansion of such a density) and regions where it does not, but that are easier to handle.

One crucial difference between \cite{Trev} and our work is that in \cite{Trev} such an approach is adopted to control Gaussian fluctuations of an empirical kernel around its expectation, rather than entropy bounds. The author of \cite{Trev} also applied recent results on the $p$-Wasserstein distance, including \cite{Bon20} and \cite{Led19}. We also point out that our use of results from \cite{HanGa} allows us to directly deal with derivatives of the network with respect to the input.

\subsubsection{Functional results}

Papers such as \cite{FHMNP, CMSV24, BGRS23, EMS21, Klu22} study the infinite-dimensional problem, where the neural network is treated as a random continuous function. In \cite{BGRS23, EMS21, Klu22}, the case $L=1$ is considered, focusing on the 2-Wasserstein distance, the $\infty$-Wasserstein distance (defined via the sup-norm), and standard distances for random variables in a Hilbert space. While bound of order $\frac{1}{n_1}$ is not achieved, the order $\frac{1}{\sqrt{n_1}}$ is obtained for polynomial activation functions.

The case $L \ge 1$ is studied in \cite{FHMNP, BGRS23}. In \cite{BGRS23}, extending techniques from \cite{Bar21}, the authors established a smoothing result for the 1-Wasserstein distance between the laws of random fields taking values in the space of continuous functions on the sphere. Using this and Stein's method, they derived a bound of order
\[
\sum_{\ell=1}^{L} \sqrt{n_{\ell+1}} \left(\frac{n_{\ell+1}^4}{n_\ell}\right)^c \log\left(\frac{n_\ell}{n_{\ell+1}^4}\right)
\]
for some constant $c > 0$. The assumptions in \cite{BGRS23} include Lipschitz non-linearities, spherical inputs, Gaussian biases, and i.i.d. weights (not necessarily Gaussian) satisfying adequate moment conditions. The convergence rate improves when considering Gaussian weights and smoother non-linearities but still does not capture convergence in law when all inner widths diverge at the same speed.

In \cite{FHMNP}, the authors worked under Assumptions \ref{bW} and \ref{hip_s}, and assumed further the non-degeneracy of $\{K^{(\ell)}\}_{\ell=1,\dots,L+1}$ up to order $q \le r-1$ with respect to $\{\partial/\partial x_1, \dots, \partial/\partial x_{n_0}\}$ (or $C^\infty$ non-linearity), and further technical conditions on the input space $\mathbb{U}$ and the eigenvalues of the trace-class operator associated with $K^{(L+1)}$:
\begin{multline*}
h \mapsto Kh := \Big\{(Kh)_j(x^{(i)}) 
= \sum_{J \in \mathcal{M}_q^{(n_0)}} \int_{\mathbb{U}} D_{x^{(k)}}^J h_j(x^{(k)}) D^J_{x^{(k)}} K_{k,i}^{(L+1)} dx^{(k)}, \\
j=1,\dots,n_{L+1}, \, x^{(i)} \in \mathbb{U} \Big\},
\end{multline*}
where $\mathcal{M}_q^{(n_0)}$ and $D_x^J$ are defined in \eqref{set_der_q} and \eqref{der_sem}, respectively, and $h \in \mathbb{W}^{q,2}(\mathbb{U}; \mathbb{R}^{n_{L+1}})$, defined as the Sobolev space of functions with square-integrable weak derivatives up to order $q$. Under these conditions, the authors of \cite{FHMNP} proved that
\[
W_{2;q}\big(z^{(L+1)}(\mathbb{U}), G^{(L+1)}(\mathbb{U})\big) \le C n^{-\frac{1}{8}},
\]
with a constant $C > 0$ independent of $n,n_1,\dots,n_L$ (which is characterized as in \eqref{bounds_n}), and $W_{2;q}$ denoting the 2-Wasserstein distance associated with the $\mathbb{W}^{2;q}$ norm. Assuming further that Assumption \ref{hip_s} holds for all $r \ge 1$, and using the Sobolev embedding theorem (see e.g. \cite{DG12}), in \cite{FHMNP} it is proved further that, for fixed $k \ge 1$:
\[
W_{\infty;k}\big(z^{(L+1)}(\mathbb{U}), G^{(L+1)}(\mathbb{U})\big) \le C n^{-\frac{1}{8}},
\]
where $C > 0$ is constant and $W_{\infty;k}$ is the $\infty$-Wasserstein distance defined with the $C^k(\bar{\mathbb{U}})$ norm. Bounds of the order $n^{-1/2}$ are also established for $C^2$-type distances associated with Hilbert-type Sobolev spaces.

\subsection{Structure of the paper}

Section \ref{pre} introduces definitions and known results used throughout the paper, including properties of Gaussian FCNNs and Hermite polynomials. Section \ref{gen_res} outlines the proof scheme for Theorem \ref{final_noC}, while Section \ref{NN_app} contains the proof of Theorem \ref{th_gen_tv_2w}. The Appendix provides the proof of Theorem \ref{bay_TV}, the proofs of all lemmas used in Section \ref{gen_res}, and the proofs of ancillary or technical results.

}

\section{Preliminaries}\label{pre}

\subsection{Results on entropy, distances and conditionally Gaussian variables.}\label{ss:results on distances}

\subsubsection{Dual representation of Wasserstein-type distances}
{
The following result is a particular case of \cite[Theorem 5.10]{villani}. It provides an alternate representation of the $p$-Wasserstein distance with $p\ge 1$ integer (as defined in \eqref{dW_def}).
\begin{theorem}\label{rappW2}
If $X,Y$ are square integrable random variables in $\mathbb{R}^d$ and $p\ge 1$ is an integer, then
\begin{equation}\label{W_2form_altra}
W_p(X,Y)^p = \sup_{h\in L^1(\mu_Y)} \Big(\mathbb{E}[h(Y)]-\mathbb{E}[h^*(X)]\Big)
\end{equation}
where $\mu_Y$ is the law of $Y$ and $h^*$ is defined as
\begin{equation}\label{hstar}
h^*(x):=\sup_{y\in\mathbb{R}^d}\big(h(y)-\|x-y\|^p\big).
\end{equation}

\end{theorem}

\begin{remark}

{ By definition, the supremum on the right-hand side of equation \eqref{W_2form_altra} is taken over all $h\in L^1(\mu_Y)$ and over all versions of $h$ (that is, over all functions belonging to the equivalence class of $h$). We notice that, if a given version of $h\in L^1(\mu_Y)$ takes the value $+\infty$ on a set of $\mu_Y$-measure zero, then $h^*(x)=+\infty$ for all $x\in\mathbb{R}^d$ and therefore
   \[
   \mathbb{E}[h(Y)]-\mathbb{E}[h^*(X)]=-\infty.
   \]
   As a consequence, without loss of generality one can remove versions with this property from the supremum on the right-hand side of \eqref{W_2form_altra}. A similar argument allows one to remove from the supremum any version of $h\in L^1(\mu_Y)$ such that $h(y) = -\infty$ for some $y$ belonging to a set of $\mu_Y$-measure zero.
   }

\end{remark}

\begin{remark}
    If $h(y)=0$ for every $y\in\mathbb{R}^d$ then also $h^*(x)=0$ for every $x\in\mathbb{R}^d$. Hence
    \[
    {\sup_{g\in L^1(\mu_Y)} \Big(\mathbb{E}[g(Y)]-\mathbb{E}[g^*(X)]\Big)\ge \mathbb{E}[h(Y)]-\mathbb{E}[h^*(X)]=0}
    \]
    and therefore the expression on the right-hand side of \eqref{W_2form_altra} {is non-negative and one can avoid the use of absolute values.}
\end{remark}
\begin{remark}\label{boundhhstar}
Observe that, thanks to the definition \eqref{hstar} of $h^*$, 
\[
h(y)-h^*(x)=h(y)-\sup_{z\in\mathbb{R}^d}\big(h(z)-\|x-z\|^p\big)\le \|x-y\|^p.
\]
\end{remark}
\begin{remark}
In the case $p=1$, Theorem \ref{rappW2} implies the following dual representation of the 1-Wasserstein distance (see \cite[Remark 6.5]{villani}):
\begin{equation}\label{lip_dist}
W_1(X,Y):=\sup_{f\in \mathcal{L}}\Big|\mathbb{E}[f(X)]-\mathbb{E}[f(Y)]\Big|,
\end{equation}
where $\mathcal{L}:=\Big\{f:\mathbb{R}^d\to\mathbb{R}$ s.t. $\sup_{z,w\in\mathbb{R}^n,z\ne w}\frac{|f(z)-f(w)|}{\|z-w\|}\le 1\Big\}$. 
\end{remark}
\subsubsection{Bounds using relative entropy}
As mentioned in Section \ref{ss:introabstract} and demonstrated by the following two statements, the relative entropy introduced in Definition \ref{def_RE} can be used to bound from above the Total Variation and 2-Wasserstein distances. 

\begin{theorem}[Pinsker-Csizsar-Kullback inequality \cite{BGL14}]\label{PCK}
If $X$, $Y$ are two random vectors in $\mathbb{R}^d$, such that the law of $X$ has a density with respect to the law of $Y$, then
\[
d_{TV}(X,Y)\le \sqrt{\frac{1}{2} D(X||Y)},
\]
where $d_{TV}(X,Y)$ is defined in \eqref{dTV_def}.
\end{theorem}

\begin{theorem}[Talagrand's inequality \cite{Tala}]\label{tala}
Let $Y\sim\mathcal{N}_d(0,I_d)$ r.v. in $\mathbb{R}^d$, where $I_d$ is the identity matrix of dimension $d\times d$, and let $X$ be a random vector with values in $\mathbb{R}^d$ such that $\mathbb{E}[\|X\|^2]<\infty$. Then
\[
W_2(X,Y)\le\sqrt{2D(X||Y)},
\]
where  $W_2(X,Y)$ is defined in \eqref{dW_def}.
\end{theorem}
{
\begin{remark}
From the previous two statements, one infers that --- if $\{X_n,Y\}$ meet appropriate conditions and if $D(X_n||Y):= \varphi(n)\to 0$ --- then $d_{TV}(X_n, Y)$ and $W_2(X_n,Y)$ also converge to zero {at a rate of the order} $O\left(\sqrt{\varphi(n)}\right)$. This implies, in particular, that $X_n$ converges to $Y$ in distribution (Proposition C.3.1 in \cite{NP12}) .
\end{remark}
}
\begin{remark}
Theorem \ref{PCK} implies a bound also on the convex distance defined in \eqref{dC_def} and Theorem \ref{tala} gives a bound also on the $1$-Wasserstein distance defined in \eqref{lip_dist}).

\end{remark}
\subsubsection{Conditionally Gaussian random variables}
We will now present (as remarks) some elementary properties of conditionally Gaussian random variables --- see Definition \ref{def_cond_ga}.  

\begin{remark}\label{rad_AN}
If a random vector $X$ in $\mathbb{R}^d$ with $\mathbb{E}[X]=0$ is conditionally Gaussian with respect to a $\sigma$-field $\mathcal{F}$ and with conditional covariance $M$, then, if $N\sim\mathcal{N}_d(0,I_d)$ is a standard Gaussian vector in $\mathbb{R}^d$ independent of the matrix $M$, one has that 
\[
X\sim \sqrt{M}N.
\]
Indeed, for all $y\in \R^d$ one has that
\[
\mathbb{E}\Big[e^{i\langle y,\sqrt{M}Ny\rangle}\Big]=\mathbb{E}\Big[\mathbb{E}\Big[e^{i\langle y,\sqrt{M}Ny\rangle}\big|\mathcal{F}\Big]\Big]=\mathbb{E}\Big[e^{-\frac{1}{2}\langle y,My\rangle}\Big]=\mathbb{E}\Big[e^{i\langle y,X\rangle}\Big],
\]
where we have used elementary properties of the conditional expectation (see e.g. \cite{PMW}) together with the identity \eqref{cond_Ga_def}.
\end{remark}

\begin{remark}\label{dens_F}
Let $X$ be a random vector in $\mathbb{R}^d$, such that $\mathbb{E}[X]=0$ and $X$ is conditionally Gaussian with respect to a $\sigma$-field $\mathcal{F}$, and with conditional covariance matrix $M$. Remark \ref{rad_AN} implies that, if $\mathbb{P}({\rm det}\, M>0)=1$, then for every $f:\mathbb{R}^d\to \mathbb{R}$ measurable and bounded, one has the identity
\begin{equation}\label{e:ef}
\mathbb{E}[f(X)]=\mathbb{E}\Big[\mathbb{E}[f(X)|\mathcal{F}]\Big]=\mathbb{E}\Big[\int_{\mathbb{R}^d}f(x)\phi_M(x)dx\Big],
\end{equation}
where $\phi_M$ is the density of the Gaussian law $\mathcal{N}_d(0,M)$. As a consequence, in this case the density of $X$ is given by $x\mapsto \mathbb{E}[\phi_M(x)]$, which is finite for almost every $x\in\mathbb{R}^d$ {(as one can see by choosing $f\equiv 1$ in \eqref{e:ef})}. We recall that, for a positive definite matrix $M$, the Gaussian density in $\mathbb{R}^d$ with zero mean and covariance $M$ is given by
\begin{equation}\label{phi_M}
\phi_M(x) := \frac{1}{(2\pi)^{d/2} \sqrt{\det M}} \, e^{-\frac{1}{2} \langle x, M^{-1} x \rangle}, \quad x\in \R^d.
\end{equation}
\end{remark}

\subsection{Further notation}\label{ss:notation}
Throughout the paper, we write $\mathbb{N}_0 := \mathbb{N} \cup \{0\}$. For any matrix $M \in \mathbb{R}^{d \times d}$, the operator norm is defined as
\[
\|M\|_{op} := \sup_{\substack{x \in \mathbb{R}^d \\ \|x\| = 1}} \|Mx\|,
\]
and the Hilbert-Schmidt norm is given by
\[
\|M\|_{HS} := \sqrt{\sum_{i,j=1}^d M_{i,j}^2} = \sqrt{\sum_{i=1}^d \lambda_i(M)^2},
\]
where $\{\lambda_i(M)\}_{i=1,\dots,d}$ are the eigenvalues of $M$. We use $\lambda(M)$ to denote the smallest eigenvalue of $M$. If $M$ is positive semi-definite, we define $\sqrt{M} \in \mathbb{R}^{d \times d}$ as the unique positive semi-definite matrix satisfying $\sqrt{M} \, \sqrt{M} = M$. The identity matrix in $\mathbb{R}^{d \times d}$ is denoted by $I_d$, and for any vector $x \in \mathbb{R}^d$, $x^T$ denotes its transpose. For any set $B$, the indicator function $1_B$ is defined as
\[
1_B(x) := 
\begin{cases}
1 & \text{if } x \in B, \\
0 & \text{if } x \notin B.
\end{cases}
\]
Given a $\sigma$-field $\mathcal{F}$, the conditional expectation with respect to $\mathcal{F}$ is denoted by $\mathbb{E}[\cdot \mid \mathcal{F}]$. A proposition $P$ is said to hold almost surely (a.s.) on an event $E$ if there exists a measurable subset $E_0 \subseteq E$ such that $P$ holds on $E_0$ and $\mathbb{P}(E \setminus E_0) = 0$.

\subsection{Bounds on observables}

The application of Theorem \ref{th_gen_tv_2w} to the analysis of randomly initialized neural networks motivates the study of the discrepancy between the derivatives of its conditional covariance matrix and those of the limiting covariance matrix. The following theorem provides an upper bound for the $L_p$-norm of this difference, directly derived from Theorem 7.3, Proposition 7.4, and Lemma 7.5 in \cite{HanGa}.

\begin{theorem}\label{coll_obs}
Under the assumptions and notations of Theorem \ref{fin_th_n}, consider any $m$-tuple $F=(f_1,\dots,f_m)$ consisting of measurable functions
\[
f_i:\mathbb{R}^d\to\mathbb{R},\quad i=1,\dots,m.
\]
{ For $\ell = 1,..., L$}, define the collective observable as the following variable
\[
\mathcal{O}_{f_i}^{(\ell)}:=\frac{1}{n_\ell}\sum_{j=1}^{n_\ell}f_i\Big(V_{x}^{J}z_j^{(\ell)}(x), x\in\mathcal{X}, { J\in\mathcal{M}_q^{(p)}}\Big).
\]
Suppose that for every $i=1,\dots,m$, $f_i$ is polynomially bounded and such that
\[
\mathbb{E}\Big[\mathcal{O}_{f_i}^{(\ell)}\Big]=0.
\]
Then, denoting by $\lceil s\rceil$ the integer part of $s+1$, one has that
\begin{equation}\label{bound_obs}
\sup_{n\ge 1}\sup_{1\le i\le m}\Big|n^{\lceil{\frac{s}{2}\rceil}}\mathbb{E}\Big[\big(\mathcal{O}_{f_i}^{(\ell)}\big)^s\Big]\Big|<\infty \quad\text{for all $s\in\mathbb{N}$.}
\end{equation}
{
The bound \eqref{bound_obs} continues to hold if  $\sigma$ and $f_i$ are of class $ C^{\infty}$ for every $i=1,\dots,m$, without assuming that the matrix defined in \eqref{lim_cov}, $\{K^{(\ell)}\}_{\ell=1,\dots,L+1}$, is non-degenerate on $\mathcal{X}$ to the order $q\le r$ with respect to $V$ (see Definition \ref{non_deg}).
}
\end{theorem}

{ An immediate consequence of the previous statement is the following Proposition.}

\begin{prop}\label{AmE}
Under the assumptions of Theorem \ref{fin_th_n}, one has that
\[
\mathbb{E}\Bigg[\Bigg(\sum_{i,j=1}^d\Big(V_{x^{(i)}}^{J^{(i)}}V_{x^{(j)}}^{J^{(j)}}A_{i,j}^{(L+1)}-\mathbb{E}[V_{x^{(i)}}^{J^{(i)}}V_{x^{(j)}}^{J^{(j)}}A_{i,j}^{(L+1)}]\Big)^2\Bigg)^4\Bigg]\le \frac{C_6}{n^4},
\]
where $C_6>0$ is a constant that does not depend on $n,n_1,\dots,n_{L}$.
\end{prop}
Using Proposition 10.3 in  \cite{HanGa}, the following result also easily follows.
\begin{prop}\label{EAmK}
Under the assumptions of Theorem \ref{fin_th_n} one has that for every $p\ge 1$ integer 
\[
\Bigg(\sum_{i,j=1}^d\Big(\mathbb{E}[V_{x^{(i)}}^{J^{(i)}}V_{x^{(j)}}^{J^{(j)}}A_{i,j}^{(L+1)}]-V_{x^{(i)}}^{J^{(i)}}V_{x^{(j)}}^{J^{(j)}}K_{i,j}^{(L+1)}\Big)^2\Bigg)^{1/2}\le \frac{C_7}{n},
\]
where $C_7>0$ is a constant that does not depend on $n,n_1,\dots,n_{L}$.
\end{prop}
\begin{remark}\label{norm_8}
Under the assumptions of Theorem \ref{fin_th_n}, and using the triangle inequality, one has that
\begin{multline*}
\mathbb{E}\Bigg[\Bigg(\sum_{i,j=1}^d\Big(V_{x^{(i)}}^{J^{(i)}}V_{x^{(j)}}^{J^{(j)}}A_{i,j}^{(L+1)}-V_{x^{(i)}}^{J^{(i)}}V_{x^{(j)}}^{J^{(j)}}K_{i,j}^{(L+1)}\Big)^2\Bigg)^4\Bigg]\\
\le\Bigg\{\mathbb{E}\Bigg[\Bigg(\sum_{i,j=1}^d\Big(V_{x^{(i)}}^{J^{(i)}}V_{x^{(j)}}^{J^{(j)}}A_{i,j}^{(L+1)}-\mathbb{E}[V_{x^{(i)}}^{J^{(i)}}V_{x^{(j)}}^{J^{(j)}}A_{i,j}^{(L+1)}]\Big)^2\Bigg)^4\Bigg]^{1/8}\\
+\Bigg(\sum_{i,j=1}^d\Big(\mathbb{E}[V_{x^{(i)}}^{J^{(i)}}V_{x^{(j)}}^{J^{(j)}}A_{i,j}^{(L+1)}]-V_{x^{(i)}}^{J^{(i)}}V_{x^{(j)}}^{J^{(j)}}K_{i,j}^{(L+1)}\Big)^2\Bigg)^{1/2}\Bigg\}^8\leq \frac{C_8}{n^4},
\end{multline*}
thanks to Proposition \ref{AmE} and Proposition \ref{EAmK}, with $C_8>0$ a constant independent of $n, n_1,\dots,n_{L}$.
\end{remark}
\begin{remark}\label{trA_fin}
Under the assumptions of Theorem \ref{fin_th_n}, for every integer $p\ge 1$, 
\begin{multline*}
\sup_n\mathbb{E}\Bigg[\Bigg(\sum_{i,j=1}^d(V_{x^{(i)}}^{J^{(i)}}V_{x^{(j)}}^{J^{(j)}}A_{i,j}^{(L+1)})^2\Bigg)^{p/2}\Bigg]\\
\le 2^{p-1}\sup_n\mathbb{E}\Bigg[\Bigg(\sum_{i,j=1}^d\Big(V_{x^{(i)}}^{J^{(i)}}V_{x^{(j)}}^{J^{(j)}}A_{i,j}^{(L+1)}-V_{x^{(i)}}^{J^{(i)}}V_{x^{(j)}}^{J^{(j)}}K_{i,j}^{(L+1)}\Big)^2\Bigg)^{p/2}\Bigg]\\
+2^{p-1}\Big(\sum_{i,j}(V_{x^{(i)}}^{J^{(i)}}V_{x^{(j)}}^{J^{(j)}}K_{i,j}^{(L+1)})^2\Big)^{p/2}\le C_9,
\end{multline*}
where the dependence on $n$ in the definition of $A^{(L+1)}$ is implicit, and $C_9>0$ is a constant independent of $n,n_1,\dots,n_{L}$. Such an estimate follows by an argument similar to the one used in Remark \ref{norm_8}.
\end{remark}

\subsection{Hermite polynomials}

We refer to \cite[Chapter 1]{NP12} for a basic introduction to real Hermite polynomials. The following definition is standard.

\begin{defin}\label{def_HP}
For every integer $k \ge 0$, the $k$-th Hermite polynomial $H_k: \mathbb{R} \to \mathbb{R}$ is defined as
\[
H_k(x) := (-1)^k \frac{1}{\phi_1(x)} \frac{d^k}{dx^k} \phi_1(x), \quad x \in \mathbb{R},
\]
where $\phi_1(x) := \frac{1}{\sqrt{2\pi}} e^{-\frac{x^2}{2}}$.
\end{defin}

For example, for $x \in \mathbb{R}$:
\[
H_0(x) = 1, \quad H_1(x) = x, \quad H_2(x) = x^2 - 1, \quad H_3(x) = x^3 - 3x, \quad \text{and so on.}
\]

For $k_1\neq k_2 \ge 0$ the polynomials $H_{k_1}, H_{k_2}$ are orthogonal under the standard Gaussian measure, and their second moments have a simple form.

\begin{prop}[Proposition 1.4.2 in \cite{NP12}]\label{var_her}
For any integers $k, j \ge 0$,
\[
\int_{\mathbb{R}} H_k(x) H_j(x) \phi_1(x) \, dx =
\begin{cases}
k! & \text{if } k = j, \\
0  & \text{otherwise.}
\end{cases}
\]
\end{prop}

Definition \ref{def_HP} can be extended to the multivariate case as in \cite{RHerPol}. In order to do that, for fixed $k\geq 0$ and $d\geq 1$, we will use the generic notation $J^{(k)}$ to indicate a multi-index of the type $J^{(k)}:=(j^{(k)}_1,\dots, j^{(k)}_d)\in\mathbb{N}_0^d$, satisfying moreover the condition $j^{(k)}_1+\dots+j^{(k)}_d=k$.

\begin{defin}[Multivariate Hermite Polynomials]
Fix $d\geq 1$ and $k\ge 0$. For every multi-index $J^{(k)}$ as above, the Hermite polynomial of multi-index $J^{(k)}$ and degree $k= |J^{(k)}|$ is the mapping $H_{J^{(k)}}:\mathbb{R}^d\to \mathbb{R}$ defined as
\[
H_{J^{(k)}}(x):=(-1)^{|J^{(k)}|}\frac{1}{\phi_{I_d}(x)}\frac{\partial^{|J^{(k)}|}}{\partial x_1^{j^{(k)}_1}\dots\partial x_d^{j^{(k)}_d}}\phi_{I_d}(x), \quad x:=(x_1,\dots,x_d)\in\mathbb{R}^d,
\]
where $\phi_{I_d}(x):=\frac{1}{(2\pi)^{d/2}}e^{-\frac{1}{2}(x_1^2+\dots+x_d^2)}$.
\end{defin}
\begin{remark}
Since 
\[
\phi_{I_d}(x) = \phi_1(x_1) \cdots \phi_1(x_d),
\]
it follows that, for any $k \ge 0$ and multi-index $J^{(k)}$,
\begin{multline}\label{her_mult_uni}
H_{J^{(k)}}(x) = (-1)^{j^{(k)}_1} \cdots (-1)^{j^{(k)}_d} \frac{1}{\phi_1(x_1)} \cdots \frac{1}{\phi_1(x_d)} \\
\times \frac{\partial^{j^{(k)}_1} \phi_1}{\partial x_1^{j^{(k)}_1}}(x_1) \cdots \frac{\partial^{j^{(k)}_d} \phi_1}{\partial x_d^{j^{(k)}_d}}(x_d) \\
= H_{j^{(k)}_1}(x_1) \cdots H_{j^{(k)}_d}(x_d).
\end{multline}
See e.g. \cite[Chapter 1]{NualartBook} for more details.
\end{remark}

}

\section{Entropy between a Gaussian law and a conditionally Gaussian law }\label{gen_res}
We will now provide a bound on the relative entropy between a Gaussian law and a conditionally Gaussian law as in Definition \ref{def_cond_ga}. 

\begin{theorem}\label{final}
Let Assumption \ref{hip_g_cond} prevail, and assume in addition that  
 $\mathbb{E}[\|A\|_{HS}^8]<\infty$,  $\mathbb{P}(\det A >0)=1$ and $\mathbb{E}[\|A^{-1}\|_{HS}^2]<\infty$. Then,
\begin{multline*}
D(F||G)\le
\frac{\sqrt{3}}{12\sqrt{2}}\big(2\sqrt{6}+3\sqrt{2}+2+\sqrt{d}\big)\|K^{-1}\|_{HS}^2\|\mathbb{E}[A]-K\|_{HS}^2\\
+\frac{1}{\lambda(K)^4}\Bigg\{8\|K\|_{HS}\mathbb{E}[\|A^{-1}\|_{HS}^2]^{1/2}
+8\|K^{-1}\|_{HS}\mathbb{E}[\|A\|_{HS}^2]^{1/2}\\
+{\frac{\sqrt{2}}{\sqrt{3}}\big(2\sqrt{6}+3\sqrt{2}+2+\sqrt{d}\big)\|K^{-1}\|_{HS}^2(\lambda(K)^2+4\|K\|_{HS}^2)}\\
+ {\sqrt{70}d^2}\Bigg(\frac{1}{2}\max{\Big\{\Big|\log\frac{2^d \det K}{(2\|K\|_{op}+\lambda(K))^d}\Big|,\log\frac{2^d \det K}{\lambda(K)^d}\Big\}}
+\frac{(\sqrt{2}+1){d}}{4}
\Bigg)\\
+\Big(3+\frac{\sqrt{3}}{8}+ 5\sqrt{10}+\frac{4\sqrt{30}}{3}+\frac{\sqrt{15}}{3}+\frac{10\sqrt{5}}{3}\Big)d^2\Bigg\}\mathbb{E}\Big[\|A-K\|_{HS}^8\Big]^{1/2}.
\end{multline*}
where $\lambda(K)$ is the minimum eigenvalue of the matrix $K$.
\end{theorem}
The proof of Theorem \ref{inE} (given at the end of the present section) hinges on the forthcoming technical Lemmas \ref{inEC} and \ref{inE}, whose proofs are detailed in Section \ref{Sec_pr1}. Our overall strategy, inspired by the ideas developed by D. Trevisan in \cite[Sections 3.1 and 3.2]{Trev}, consists in partitioning the probability space in the event  
\begin{equation}\label{def_ev_E}
E=\Big\{\|A-K\|_{op}\le\frac{\lambda(K)}{2}\Big\}
\end{equation}
 and its complement. Then, using the notation \eqref{phi_M}, Definition \ref{def_RE} and Remark \ref{dens_F} one has that
\[
D(F||G)=\int_{\mathbb{R}^d}{\mathbb{E}[\phi_A(x)]}\log \Bigg(\frac{\mathbb{E}[\phi_A(x)]}{\phi_K(x)}\Bigg)dx.
\]

Assuming that $\mathbb{P}(E)\ne 0,1$, one obtains that
\begin{multline*}
D(F||G)=\int_{\mathbb{R}^d}\Bigg(\frac{\mathbb{P}(E)\mathbb{E}[\phi_A(x)1_E]}{\mathbb{P}(E)}+\frac{\mathbb{P}(E^C)\mathbb{E}[\phi_A(x)1_{E^C}]}{\mathbb{P}(E^C)}\Bigg)\cdot\\
\cdot\log \Bigg(\frac{\mathbb{P}(E)\mathbb{E}[\phi_A(x)1_E]}{\mathbb{P}(E)\phi_K(x)}+\frac{\mathbb{P}(E^C)\mathbb{E}[\phi_A(x)1_{E^C}]}{\mathbb{P}(E^C)\phi_K(x)}\Bigg)dx
\end{multline*}
\begin{multline}\label{spiezzo in due}
\le \int_{\mathbb{R}^d}\mathbb{E}[\phi_A(x)1_E]\log \Bigg(\frac{\mathbb{E}[\phi_A(x)1_E]}{\mathbb{P}(E)\phi_K(x)}\Bigg)dx
+\int_{\mathbb{R}^d}\mathbb{E}[\phi_A(x)1_{E^C}]\log \Bigg(\frac{\mathbb{E}[\phi_A(x)1_{E^C}]}{\mathbb{P}(E^C)\phi_K(x)}\Bigg)dx
\end{multline}
where we have used the fact { that the mapping $x\mapsto x\log (x/c)$ is convex for every $c>0$}, and that $\mathbb{P}(E)+\mathbb{P}(E^C)=1$. As a consequence, to prove Theorem \ref{final}, it suffices to establish bounds on the two terms appearing in \eqref{spiezzo in due}.

 The second term in \eqref{spiezzo in due} is controlled by the forthcoming Lemma \ref{inEC}, whose proof (presented in Section \ref{ss:plinec}) uses the convexity of the function $x\to x\log x$ and Jensen inequality, as applied to the probability measure with density $\frac{1_{E^C}}{\mathbb{P}(E^C)}$ with respect to $\mathbb{P}$.
\begin{lemma}\label{inEC}

Let Assumption \ref{hip_g_cond} prevail, and assume that $\mathbb{P}(\det A>0)=1$ and that $\mathbb{E}[\|A^{-1}\|_{HS}^2]<\infty$. Then if $\mathbb{P}(E)\ne 1$, where $E$ is defined in \eqref{def_ev_E}, one has that
\begin{multline*}
\int_{\mathbb{R}^d}{\mathbb{E}[\phi_A(x)1_{E^C}]}\log \Bigg(\frac{\mathbb{E}[\phi_A(x)1_{E^C}]}{\mathbb{P}(E^C)\phi_K(x)}\Bigg)dx\\
\le
\frac{8}{\lambda(K)^4}\Big(\|K\|_{HS}\mathbb{E}[\|A^{-1}\|_{HS}^2]^{1/2}+\|K^{-1}\|_{HS}\mathbb{E}[\|A\|_{HS}^2]^{1/2}\Big)\mathbb{E}[\|A-K\|_{HS}^8]^{1/2}.
\end{multline*}
\end{lemma}
The first term in \eqref{spiezzo in due} is bounded by using an interpolation scheme --- studied in detail in Section \ref{ss:ppline} --- yielding a collection of random variables $\{F_t : t\in [0,1]\}$, smoothly depending on the parameter $t$ and such that $F_0$ and $F_1$  have, respectively the same law as $G$ and $F$. As demonstrated e.g. in Proposition \ref{var_hk}, our techniques involve a fine control of the Taylor expansion of the relative entropy between $F_t$ and $G$, as a function of $t\in (0,1)$. The resulting global bound is given in the next statement, whose proof is detailed in Section \ref{ss:pline}.

\begin{lemma}\label{inE}
Assume Assumption \ref{hip_g_cond} and that $\mathbb{E}[\|A\|_{HS}^8]<\infty$. Then if $\mathbb{P}(E)\ne0$, where $E$ is defined in \eqref{def_ev_E}, it holds that
\begin{multline*}
\int_{\mathbb{R}^d}{\mathbb{E}[\phi_A(x)1_E]}\log \Bigg(\frac{\mathbb{E}[\phi_A(x)1_E]}{\mathbb{P}(E)\phi_K(x)}\Bigg)dx\\
\le
\frac{\sqrt{3}}{12\sqrt{2}}\big(2\sqrt{6}+3\sqrt{2}+2+\sqrt{d}\big)\|K^{-1}\|_{HS}^2\|\mathbb{E}[A]-K\|_{HS}^2\\
+\Bigg\{\frac{\sqrt{2}}{\lambda(K)^4\sqrt{3}}\big(2\sqrt{6}+3\sqrt{2}+2+\sqrt{d}\big)\|K^{-1}\|_{HS}^2(\lambda(K)^2+4\|K\|_{HS}^2)\\
+ \frac{\sqrt{70}d^2}{\lambda(K)^4}\Bigg(\frac{1}{2}\max{\Big\{\Big|\log\frac{2^d \det K}{(2\|K\|_{op}+\lambda(K))^d}\Big|,\log\frac{2^d \det K}{\lambda(K)^d}\Big\}}
+\frac{(\sqrt{2}+1){d}}{4}
\Bigg)\\
+\Big(3+\frac{\sqrt{3}}{8}+ 5\sqrt{10}+\frac{4\sqrt{30}}{3}+\frac{\sqrt{15}}{3}+\frac{10\sqrt{5}}{3}\Big)\frac{d^2}{\lambda(K)^4}\Bigg\}\mathbb{E}\Big[\|A-K\|_{HS}^8\Big]^{1/2}.
\end{multline*}
\end{lemma}

\begin{proof}[Proof of Theorem \ref{final}]
If $\mathbb{P}(E)\ne 0\ne \mathbb{P}(E^C)$, applying Lemma \ref{inEC} and Lemma \ref{inE} to inequality \eqref{spiezzo in due} one infers that
\begin{multline*}
D(F||G)\le 
\frac{8}{\lambda(K)^4}\Big(\|K\|_{HS}\mathbb{E}[\|A^{-1}\|_{HS}^2]^{1/2}+\|K^{-1}\|_{HS}\mathbb{E}[\|A\|_{HS}^2]^{1/2}\Big)\mathbb{E}[\|A-K\|_{HS}^8]^{1/2}\\
+\frac{\sqrt{3}}{12\sqrt{2}}\big(2\sqrt{6}+3\sqrt{2}+2+\sqrt{d}\big)\|K^{-1}\|_{HS}^2\|\mathbb{E}[A]-K\|_{HS}^2\\
+\Bigg\{\frac{\sqrt{2}}{\lambda(K)^4\sqrt{3}}\big(2\sqrt{6}+3\sqrt{2}+2+\sqrt{d}\big)\|K^{-1}\|_{HS}^2(\lambda(K)^2+4\|K\|_{HS}^2)\\
+ \frac{\sqrt{70}d^2}{\lambda(K)^4}\Bigg(\frac{1}{2}\max{\Big\{\Big|\log\frac{2^d \det K}{(2\|K\|_{op}+\lambda(K))^d}\Big|,\log\frac{2^d \det K}{\lambda(K)^d}\Big\}}
+\frac{(\sqrt{2}+1){d}}{4}
\Bigg)\\
+\Big(3+\frac{\sqrt{3}}{8}+ 5\sqrt{10}+\frac{4\sqrt{30}}{3}+\frac{\sqrt{15}}{3}+\frac{10\sqrt{5}}{3}\Big)\frac{d^2}{\lambda(K)^4}\Bigg\}\mathbb{E}\Big[\|A-K\|_{HS}^8\Big]^{1/2}
\end{multline*}
\begin{multline*}
= 
\frac{\sqrt{3}}{12\sqrt{2}}\big(2\sqrt{6}+3\sqrt{2}+2+\sqrt{d}\big)\|K^{-1}\|_{HS}^2\|\mathbb{E}[A]-K\|_{HS}^2\\
+\frac{1}{\lambda(K)^4}\Bigg\{8\|K\|_{HS}\mathbb{E}[\|A^{-1}\|_{HS}^2]^{1/2}
+8\|K^{-1}\|_{HS}\mathbb{E}[\|A\|_{HS}^2]^{1/2}\\
+{\frac{\sqrt{2}}{\sqrt{3}}\big(2\sqrt{6}+3\sqrt{2}+2+\sqrt{d}\big)\|K^{-1}\|_{HS}^2(\lambda(K)^2+4\|K\|_{HS}^2)}\\
+ {\sqrt{70}d^2}\Bigg(\frac{1}{2}\max{\Big\{\Big|\log\frac{2^d \det K}{(2\|K\|_{op}+\lambda(K))^d}\Big|,\log\frac{2^d \det K}{\lambda(K)^d}\Big\}}
+\frac{(\sqrt{2}+1){d}}{4}
\Bigg)\\
+\Big(3+\frac{\sqrt{3}}{8}+ 5\sqrt{10}+\frac{4\sqrt{30}}{3}+\frac{\sqrt{15}}{3}+\frac{10\sqrt{5}}{3}\Big){d^2}\Bigg\}\mathbb{E}\Big[\|A-K\|_{HS}^8\Big]^{1/2}.
\end{multline*}

The bound is of course true also if $\mathbb{P}(E)=0$ or $\mathbb{P}(E^C)=0$. In fact if for example $\mathbb{P}(E)=0$, then 
\[
D(F||G)=\int_{\mathbb{R}^d}{\mathbb{E}[\phi_A(x)1_{E^C}]}\log \Bigg(\frac{\mathbb{E}[\phi_A(x)1_{E^C}]}{\mathbb{P}(E^C)\phi_K(x)}\Bigg)\Bigg]
\]
and a direct application of Lemma \ref{inEC} yields the desired estimate. If $\mathbb{P}(E^C)=0$ the procedure is analogous. 
\end{proof}
\section{Bounds on the distance between a Gaussian  and a conditionally Gaussian law }\label{NN_app}

{ In this section, we derive bounds on the total variation and 2-Wasserstein distances, defined respectively in \eqref{dTV_def} and \eqref{dW_def}, between the law of a conditionally Gaussian random variable and a Gaussian distribution with invertible covariance matrix. Indeed, a limitation of Theorem \ref{final} is that its assumptions on the finite moments of the inverse conditional covariance matrix $A^{-1}$ are rarely met in applications, preventing its direct use—along with Theorems \ref{PCK} and \ref{tala}—to obtain distance bounds. This issue is addressed via Lemma \ref{inE}. Specifically, using the notation of the lemma, we observe that on the event $E$ defined in \eqref{def_ev_E}, the inequality $\|A-K\|{op} \le \lambda(K)/2$ holds. Consequently, for any $x \in \mathbb{R}^d$ with $\|x\| = 1$, we have \begin{equation}\label{aut_A} x^T A x = x^T (A - K) x + x^T K x \ge \lambda(K) - \|A - K\|{op} \ge \frac{\lambda(K)}{2} > 0, \end{equation} implying that $A$ is invertible on $E$. It follows that the conditionally Gaussian random vector admits a density on this event without requiring additional assumptions.}

\begin{theorem}\label{th_dis_nuo}
Fix $d\ge 1$, let Assumption \ref{hip_g_cond} prevail, and assume that $\mathbb{E}[\|A\|_{HS}^8]<\infty$.
Then,
\begin{multline}\label{dtv}
d_{TV}(F,G)\le
\Big(\frac{\sqrt{3}}{24\sqrt{2}}\big(2\sqrt{6}+3\sqrt{2}+2+\sqrt{d}\big)\Big)^{1/2}\|K^{-1}\|_{HS} \|\mathbb{E}[A]-K\|_{HS}\\
+\Bigg\{\Bigg[\frac{\sqrt{2}}{2\lambda(K)^4\sqrt{3}}\big(2\sqrt{6}+3\sqrt{2}+2+\sqrt{d}\big)\|K^{-1}\|_{HS}^2(\lambda(K)^2+4\|K\|_{HS}^2)\\
+ \frac{\sqrt{70}d^2}{2\lambda(K)^4}\Bigg(\frac{1}{2}\max{\Big\{\Big|\log\frac{2^d \det K}{(2\|K\|_{op}+\lambda(K))^d}\Big|,\log\frac{2^d \det K}{\lambda(K)^d}\Big\}}
+\frac{(\sqrt{2}+1){d}}{4}
\Bigg)\\
+\Big(3+\frac{\sqrt{3}}{8}+ 5\sqrt{10}+\frac{4\sqrt{30}}{3}+\frac{\sqrt{15}}{3}+\frac{10\sqrt{5}}{3}\Big)\frac{d^2}{2\lambda(K)^4}\Bigg]^{1/2}+\frac{8}{\lambda(K)^2}\Bigg\}\mathbb{E}\Big[\|A-K\|_{HS}^8\Big]^{1/4},
\end{multline}
and
\begin{multline}\label{W2nn}
W_2(F,G)\le
\|K\|_{op}^{1/2}
\Big(\frac{\sqrt{3}}{6\sqrt{2}}(2\sqrt{6}+3\sqrt{2}+2+\sqrt{d}\big)\Big)^{1/2}\|K^{-1}\|_{HS} \|\mathbb{E}[A]-K\|_{HS}\\
+\Bigg\{\|K\|_{op}^{1/2}\Bigg[\frac{2\sqrt{2}}{\lambda(K)^4\sqrt{3}}(2\sqrt{6}+3\sqrt{2}+2+\sqrt{d}\big)\|K^{-1}\|_{HS}^2(\lambda(K)^2+4\|K\|_{HS}^2)\\
+ \frac{2\sqrt{70}d^2}{\lambda(K)^4}\Bigg(\frac{1}{2}\max{\Big\{\Big|\log\frac{2^d \det K}{(2\|K\|_{op}+\lambda(K))^d}\Big|,\log\frac{2^d \det K}{\lambda(K)^d}\Big\}}
+\frac{(\sqrt{2}+1){d}}{4}
\Bigg)\\
+\Big(3+\frac{\sqrt{3}}{8}+ 5\sqrt{10}+\frac{4\sqrt{30}}{3}+\frac{\sqrt{15}}{3}+\frac{10\sqrt{5}}{3}\Big)\frac{2d^2}{\lambda(K)^4}\Bigg]^{1/2}+\frac{2}{\lambda(K)^2}\Bigg\} \mathbb{E}\Big[\|A-K\|_{HS}^8\Big]^{1/4}.
\end{multline}
\end{theorem}
\begin{proof}

{ Recall \eqref{def_ev_E} for the definition of the event $E$ and suppose for now  that $\mathbb{P}(E)\ne 0$.} 
{ Without loss of generality, we can assume that $F$, $A$ and $G$ are defined on the same probability space and that $F$ and $A$ are independent of $G$}. Then for every $S\in\mathcal{B}(\mathbb{R}^d)$,
\begin{multline}\label{d_TV_S_E}
\Big|\mathbb{E}[1_S(F)]-\mathbb{E}[1_S(G)]\Big|
\le \Big|\mathbb{E}\Big[\Big(1_S(F)-1_S(G)\Big)1_{\Big\{\|A-K\|_{op}\le \frac{\lambda(K)}{2}\Big\}}\Big]\Big|\\
+\Big|\mathbb{E}\Big[\Big(1_S(F)-1_S(G)\Big)1_{\Big\{\|A-K\|_{op}> \frac{\lambda(K)}{2}\Big\}}\Big]\Big|
\end{multline}
\begin{equation}\label{d_TV_Mark_Entr}
\le  \sqrt{\frac{1}{2}\int_{\mathbb{R}^d}{\mathbb{E}[\phi_{A}(x)1_{E}]}\log \Bigg(\frac{\mathbb{E}[\phi_{A}(x)\frac{1_{E}}{\mathbb{P}(E)}]}{\phi_{K}(x)}\Bigg)dx}
+2\mathbb{P}\Big(\|A-K\|_{op}> \frac{\lambda(K)}{2}\Big),
\end{equation}
 where we have used Theorem \ref{PCK} under the probability $d\mathbb{Q}:=\frac{1_{E}}{\mathbb{P}(E)}d\mathbb{P}$, as well as the facts that: { (i) the density of the law of $F$ under $\mathbb{Q}$ is $x\mapsto \mathbb{E}\Big[\phi_{A}(x)\frac{1_{E}}{\mathbb{P}(E)}\Big]$, and (ii) since $G$ and $E$ are independent by assumption under $\mathbb{P}$, the density of $G$ under $\mathbb{Q}$ is given by $\phi_K$.} Hence, using Lemma \ref{inE} (for the first term) and the Markov inequality (for the second term), one deduces that
\begin{multline*}
\Big|\mathbb{E}[1_S(F)]-\mathbb{E}[1_S(G)]\Big|\le \frac{8}{\lambda(K)^2}\mathbb{E}\Big[\|A-K\|_{HS}^2\Big]\\
+ \Big(\frac{\sqrt{3}}{24\sqrt{2}}\big(2\sqrt{6}+3\sqrt{2}+2+\sqrt{d}\big)\Big)^{1/2}\|K^{-1}\|_{HS} \|\mathbb{E}[A]-K\|_{HS}\\
+\Bigg\{\frac{\sqrt{2}}{2\lambda(K)^4\sqrt{3}}\big(2\sqrt{6}+3\sqrt{2}+2+\sqrt{d}\big)\|K^{-1}\|_{HS}^2(\lambda(K)^2+4\|K\|_{HS}^2)\\
+ \frac{\sqrt{70}d^2}{2\lambda(K)^4}\Bigg(\frac{1}{2}\max{\Big\{\Big|\log\frac{2^d det K}{(2\|K\|_{op}+\lambda(K))^d}\Big|,\log\frac{2^d detK}{\lambda(K)^d}\Big\}}
+\frac{(\sqrt{2}+1){d}}{4}
\Bigg)\\
+\Big(3+\frac{\sqrt{3}}{8}+ 5\sqrt{10}+\frac{4\sqrt{30}}{3}+\frac{\sqrt{15}}{3}+\frac{10\sqrt{5}}{3}\Big)\frac{d^2}{2\lambda(K)^4}\Bigg\}^{1/2}\mathbb{E}\Big[\|A-K\|_{HS}^8\Big]^{1/4}
\end{multline*}
and taking the sup over $S$ one obtains the desired estimate \eqref{dtv}. {If $\mathbb{P}(E)=0$ then inequality \eqref{d_TV_S_E} reads as
\[
\Big|\mathbb{E}[1_S(F)]-\mathbb{E}[1_S(G)]\Big|
\le \Big|\mathbb{E}\Big[\Big(1_S(F)-1_S(G)\Big)1_{\Big\{\|A-K\|_{op}> \frac{\lambda(K)}{2}\Big\}}\Big]\Big|
\]
and the conclusion directly follows from Markov inequality.
}

We now proceed to the proof of the bound \eqref{W2nn}, assuming that $\mathbb{P}(E)\ne 0$ {(the case $\mathbb{P}(E)=0$ can be studied exactly as in the Total Variation distance)}.
Using Theorem \ref{rappW2} it follows that
\[
W_2(F,G)=\sqrt{\sup_{h\in L^1(\mu_{G})} \Big(\mathbb{E}\big[h(G)\big]-\mathbb{E}\big[h^*(F)\big]\Big)},
\]
where $h^{*}$ is defined in \eqref{hstar} and $\mu_{G}$ denotes the law of $G$. { Since the joint distribution of the pair $(F,G)$ is immaterial for bounding $W_2(F,G)$,} without loss of generality we can suppose $G\sim\sqrt{K}N_1$ and $F\sim\sqrt{A}N_1$, with $N_1\sim\mathcal{N}_d(0,I_d)$ independent of $A$. Considering the event $E$ defined in \eqref{def_ev_E} one has that
\begin{multline}\label{div_E_W2}
W_2(F,G)
\le \sqrt{\mathbb{P}(E)}\sqrt{\sup_{h\in L^1(\mu_{G})} \Big(\mathbb{E}\Big[h(\sqrt{K}N_1)\frac{1_{E}}{\mathbb{P}(E)}\Big]-\mathbb{E}\Big[h^*(\sqrt{A}N_1)\frac{1_{E}}{\mathbb{P}(E)}\Big]\Big)}\\
+\sqrt{\sup_{h\in L^1(\mu_{G})} \Big(\mathbb{E}\big[h(\sqrt{K}N_1)1_{E^C}\big]-\mathbb{E}\big[h^*(\sqrt{A}N_1)1_{E^C}\big]\Big)}
\end{multline}
{ (note that the two suprema are nonnegative, as one can see by considering the case $h=0$)}. The second summand can be easily bounded by conditioning on $A$ and applying the content of Remarks \ref{boundhhstar} and \ref{rad_AN}, yielding the estimate
\begin{multline}\label{term1W2}
\sup_{h\in L^{1}\big(\mu_{G}\big)}\Big(\mathbb{E}\big[h(\sqrt{K}N_1)1_{E^C}\big]-\mathbb{E}\big[h^*(\sqrt{A}N_1)1_{E^C}\big]\Big)\\
\le \mathbb{E}\Big[\|\sqrt{K}N_1-\sqrt{A}N_1\|^21_{E^C}\Big]=\mathbb{E}\Big[\|\sqrt{K}-\sqrt{A}\|_{HS}^21_{E^C}\Big]\\
\le \frac{4}{\lambda(K)^4} \mathbb{E}\Big[\|A-K\|_{HS}^4\Big],
\end{multline}
where we have used the definition of the event $E$, and the fact that $\|\sqrt{A}-\sqrt{K}\|_{HS}\le \frac{1}{\lambda(K)}\|A-K\|_{HS}$ thanks to Proposition 3.2 from \cite{trace}. We now study the first summand on the right-hand side of \eqref{div_E_W2}. Applying Theorem \ref{rappW2} one has that 
\[
{\mathbb{P}(E)}{\sup_{h\in L^1(\mu_{G})} \Big(\mathbb{E}\Big[h(\sqrt{K}N_1)\Big]-\mathbb{E}\Big[h^*(\sqrt{A}N_1)\frac{1_{E}}{\mathbb{P}(E)}\Big]\Big)}
={\mathbb{P}(E)}W_2(Z,G)^2,
\]
where $Z$ is defined as a r.v. in $\mathbb{R}^d$ with density with respect to the Lebesgue measure given by
\[
x\mapsto \mathbb{E}\Big[\phi_{A}(x)\frac{1_{E}}{\mathbb{P}(E)}\Big], \quad\text{ $x\in\mathbb{R}^d$}.
\]
Moreover, writing $N_2\sim\mathcal{N}_d(0,I_d)$,
\begin{multline*}
{\mathbb{P}(E)}W_2(Z,G)^2={\mathbb{P}(E)}\inf_{ (Y,W), Y\sim Z, W\sim G}\mathbb{E}[\|Y-W\|^2] \\
={\mathbb{P}(E)}\inf_{ (Y,X), Y\sim Z, X\sim N_2}\mathbb{E}[\|\sqrt{K}X-Y\|^2]\\
\le{\mathbb{P}(E)}\|K\|_{op}\inf_{ (Y,X), Y\sim Z, X\sim N_2}\mathbb{E}[\|X-(\sqrt{K})^{-1}Y\|^2]
={\mathbb{P}(E)} \|K\|_{op}W_2\Big(\big(\sqrt{K}\big)^{-1}Z,N_2\Big)^2\\
\le2 \|K\|_{op}
{\int_{\mathbb{R}^d}{\mathbb{E}\big[\phi_{(\sqrt{K})^{-1}A(\sqrt{K})^{-1}}(x)1_{E}\big]}\log \Bigg(\frac{\mathbb{E}\big[\phi_{(\sqrt{K})^{-1}A(\sqrt{K})^{-1}}(x)\frac{1_{E}}{\mathbb{P}(E)}\big]}{\phi_{I_d}(x)}\Bigg)dx}\\
= 2\|K\|_{op}
\int_{\mathbb{R}^d}{\mathbb{E}[\phi_{A}(x)1_{E}]}\log \Bigg(\frac{\mathbb{E}[\phi_{A}(x)\frac{1_{E}}{\mathbb{P}(E)}]}{\phi_{K}(x)}\Bigg)dx,
\end{multline*}
{ where the last equality follows from a standard change of variables}, and we have used the fact that the density of $(\sqrt{K})^{-1}Z$ is given by
\[
x\mapsto \mathbb{E}\Big[\phi_{(\sqrt{K})^{-1}A(\sqrt{K})^{-1}}(x)\frac{1_{E}}{\mathbb{P}(E)}\Big],\quad x\in\mathbb{R}^d,
\]
in conjunction with Theorem \ref{tala}.
Finally, thanks to Lemma \ref{inE} one infers that
\begin{multline}\label{term2W2}
{\mathbb{P}(E)}{\sup_{h\in L^1(\mu_{G})} \Big(\mathbb{E}\Big[h(G)\frac{1_{E}}{\mathbb{P}(E)}\Big]-\mathbb{E}\Big[h^*(F)\frac{1_{E}}{\mathbb{P}(E)}\Big]\Big)}\\
\le \|K\|_{op}
\Big(\frac{\sqrt{3}}{6\sqrt{2}}(2\sqrt{6}+3\sqrt{2}+2+\sqrt{d}\big)\Big)\|K^{-1}\|^2_{HS} \|\mathbb{E}[A]-K\|_{HS}^2\\
+\|K\|_{op}\Bigg\{\frac{2\sqrt{2}}{\lambda(K)^4\sqrt{3}}(2\sqrt{6}+3\sqrt{2}+2+\sqrt{d}\big)\|K^{-1}\|_{HS}^2(\lambda(K)^2+4\|K\|_{HS}^2)\\
+ \frac{2\sqrt{70}d^2}{\lambda(K)^4}\Bigg(\frac{1}{2}\max{\Big\{\Big|\log\frac{2^d \det K}{(2\|K\|_{op}+\lambda(K))^d}\Big|,\log\frac{2^d \det K}{\lambda(K)^d}\Big\}}
+\frac{(\sqrt{2}+1){d}}{4}
\Bigg)\\
+\Big(3+\frac{\sqrt{3}}{8}+ 5\sqrt{10}+\frac{4\sqrt{30}}{3}+\frac{\sqrt{15}}{3}+\frac{10\sqrt{5}}{3}\Big)\frac{2d^2}{\lambda(K)^4}\Bigg\}\mathbb{E}\Big[\|A-K\|_{HS}^8\Big]^{1/2}
\end{multline}
Inequality \eqref{W2nn} immediately follows from the bounds \eqref{div_E_W2}, \eqref{term1W2} and \eqref{term2W2}.

\end{proof}

\section*{Acknowledgments}
The first author was supported by the Luxembourg National Research Fund via the grant
PRIDE/21/16747448/MATHCODA. Both authors are grateful to Boris Hanin, Domenico Marinucci, Ivan Nourdin and Dario Trevisan for several useful discussions.

\section{Appendix}\label{pr_inv}
\subsection{Proofs of Theorem \ref{fin_th_n} and Theorem \ref{bay_TV}}

{ \noindent\underline{\it Proof of Theorem \ref{fin_th_n}}. As already observed, the proof of this result  follows from Theorem \ref{th_gen_tv_2w}, that one has to specialize to the case 
$$
A = \left\{ V^{J^{(i)}}_{x^{(i)}}V^{J^{(i)}}_{x^{(i)}}A_{i,j}^{(L+1)} : i,j= 1,...,d\right\}, \quad K = \left\{ V^{J^{(i)}}_{x^{(i)}}V^{J^{(i)}}_{x^{(i)}}K_{i,j}^{(L+1)} : i,j= 1,...,d\right\},
$$
and combine with Proposition \ref{EAmK}, as well as with the content of Remarks \ref{oss_FHMNP}, \ref{gaus_cond_der_z}, \ref{norm_8}, and \ref{trA_fin}. 
}

\bigskip 

\noindent\underline{\it Proof of Theorem \ref{bay_TV}}. Denote by $\mu_{Z|\mathcal{D}}$, $\mu_{G|\mathcal{D}}$, $\mu_Z$ and $\mu_G$ the laws of $Z_{|\mathcal{D}}$, $G_{|\mathcal{D}}$, $Z$ and $G$ respectively. One has that
    \[
    d{\mu_{Z|\mathcal{D}}}(x)=\frac{\mathcal{L}(x)}{\mathbb{E}[\mathcal{L}(Z)]}d\mu_Z(x)\quad\text{and}\quad d\mu_{G|\mathcal{D}}(x)=\frac{\mathcal{L}(x)}{\mathbb{E}[\mathcal{L}(G)]}d\mu_G(x),
    \]
    and hence
    \[
    d_{TV}({Z}_{|\mathcal{D}},{G}_{|\mathcal{D}})
    =\sup_{B\in\mathcal{B}(\mathbb{R}^{d\times n_{L+1}})}\Bigg|\int_B\frac{\mathcal{L}(x)}{\mathbb{E}[\mathcal{L}(Z)]}d\mu_Z(x)-\int_B\frac{\mathcal{L}(x)}{\mathbb{E}[\mathcal{L}(G)]}d\mu_G(x)\Bigg|
    \]
    \begin{multline*}
    \le \sup_{B\in\mathcal{B}(\mathbb{R}^{d\times n_{L+1}})}\Bigg|\int_B\frac{\mathcal{L}(x)}{\mathbb{E}[\mathcal{L}(Z)]}d\mu_Z(x)-\int_B\frac{\mathcal{L}(x)}{\mathbb{E}[\mathcal{L}(G)]}d\mu_Z(x)\Bigg|\\
    +\sup_{B\in\mathcal{B}(\mathbb{R}^{d\times n_{L+1}})}\Bigg|\int_B\frac{\mathcal{L}(x)}{\mathbb{E}[\mathcal{L}(G)]}d\mu_Z(x)-\int_B\frac{\mathcal{L}(x)}{\mathbb{E}[\mathcal{L}(G)]}d\mu_G(x)\Bigg|
    \end{multline*}
   \begin{multline*}
       \le \|\mathcal{L}\|_{\infty}\sup_{B\in\mathcal{B}(\mathbb{R}^{d\times n_{L+1}})}\mathbb{P}(Z\in B)\Bigg|\frac{1}{\mathbb{E}[\mathcal{L}(Z)]}-\frac{1}{\mathbb{E}[\mathcal{L}(G)]}\Bigg|\\
       +\frac{1}{\mathbb{E}[\mathcal{L}(G)]}\sup_{B\in\mathcal{B}(\mathbb{R}^{d\times n_{L+1}})}\Bigg|\int_B{\mathcal{L}(x)}d\mu_Z(x)-\int_B{\mathcal{L}(x)}d\mu_G(x)\Bigg|
   \end{multline*}
   \begin{multline*}
       \le \frac{\|\mathcal{L}\|^2_{\infty}}{\mathbb{E}[\mathcal{L}(Z)]\mathbb{E}[\mathcal{L}(G)]}\Bigg|{\mathbb{E}\Big[\frac{\mathcal{L}(Z)}{\|\mathcal{L}\|_{\infty}}\Big]}-{\mathbb{E}\Big[\frac{\mathcal{L}(G)}{\|\mathcal{L}\|_{\infty}}\Big]}\Bigg|\\
        +\frac{\|\mathcal{L}\|_{\infty}}{\mathbb{E}[\mathcal{L}(G)]}\sup_{B\in\mathcal{B}(\mathbb{R}^{d\times n_{L+1}})}\Bigg|\mathbb{E}\Big[1_B(Z)\frac{\mathcal{L}(Z)}{\|\mathcal{L}\|_{\infty}}\Big]-\mathbb{E}\Big[1_B(G)\frac{\mathcal{L}(G)}{\|\mathcal{L}\|_{\infty}}\Big]\Bigg|
   \end{multline*}
   \[
   \le \frac{\|\mathcal{L}\|_{\infty}}{\mathbb{E}[\mathcal{L}(G)]}\Bigg(\frac{\|\mathcal{L}\|_{\infty}}{\mathbb{E}[\mathcal{L}(Z)]}+1\Bigg)d_{TV}(Z,G),
   \]
   thanks to the second identity in the definition of the Total Variation distance provided in \eqref{dTV_def}. The final bound in the statement directly follows from Theorem \ref{fin_th_n}.

\subsection{Proofs of the results in Section \ref{gen_res}}\label{Sec_pr1}

\subsubsection{Proof of Lemma \ref{inEC}}\label{ss:plinec}
Using that $x\mapsto x\log x $ is a non-negative convex function, that the symbol $\frac{1_{E^C}}{\mathbb{P}(E^C)} {\mathbb{P}}$ defines a probability measure and
  exploiting Jensen's inequality, one infers that
\begin{multline*}
\int_{\mathbb{R}^d}{\mathbb{E}[\phi_A(x)1_{E^C}]}\log \Bigg(\frac{\mathbb{E}[\phi_A(x)1_{E^C}]}{\mathbb{P}(E^C)\phi_K(x)}\Bigg)dx\le \int_{\mathbb{R}^d}\mathbb{E}\Bigg[\phi_A(x)1_{E^C}\log \Bigg(\frac{\phi_A(x)}{\phi_K(x)}\Bigg)\Bigg]dx\\
\le \int_{\mathbb{R}^d}\mathbb{E}\Bigg[\phi_A(x)1_{E^C}\log \Bigg(\frac{\phi_A(x)}{\phi_K(x)}\Bigg)\Bigg]dx+ \int_{\mathbb{R}^d}\mathbb{E}\Bigg[\phi_K(x)1_{E^C}\log \Bigg(\frac{\phi_K(x)}{\phi_A(x)}\Bigg)\Bigg]dx\\
=\frac{1}{2}\mathbb{E}\Big[tr\Big(\sqrt{A}K^{-1}\sqrt{A}+\sqrt{K}A^{-1}\sqrt{K}-2I_d\Big)1_{E^C}\Big]\\
\le \frac{1}{2}\mathbb{E}\Big[\big(\|A\|_{HS}\|K^{-1}\|_{HS}+\|K\|_{HS}\|A^{-1}\|_{HS}\big)1_{E^C}\Big],
\end{multline*}

\noindent { where the first inequality trivially follows from the addition of a positive term, and the subsequent identity is a direct consequence of classical formulae for the relative entropy between absolutely continuous Gaussian elements}. To conclude, we observe that, by definition, on the event $E^C$ one has that $\frac{1}{2}< \frac{\|A-K\|_{op}}{\lambda(K)}$: as a consequence, using H$\ddot{o}$lder's inequality and the bound $\|A-K\|_{op}\le \|A-K\|_{HS}$, one deduces that
\begin{multline*}
\int_{\mathbb{R}^d}{\mathbb{E}[\phi_A(x)1_{E^C}]}\log \Bigg(\frac{\mathbb{E}[\phi_A(x)1_{E^C}]}{\mathbb{P}(E^C)\phi_K(x)}\Bigg)dx\\
\le \frac{8}{\lambda(K)^4}\Big(\|K\|_{HS}\mathbb{E}[\|A^{-1}\|_{HS}^2]^{1/2}+\|K^{-1}\|_{HS}\mathbb{E}[\|A\|_{HS}^2]^{1/2}\Big)\mathbb{E}[\|A-K\|_{HS}^8]^{1/2},
\end{multline*}
from which the desired bound follows at once.

\subsubsection{Proof of Lemma \ref{inE}: preliminary results}\label{ss:ppline}
Without loss of generality for the rest of this section, we will assume that $G$ is independent { of the pair $(A, F)$.} We introduce the matrices 
\begin{equation}\label{def_Gt}
\Gamma_t:=tA+(1-t)K, \quad \text{ $t\in [0,1]$} 
\end{equation}
and observe that on the event $E:=\{\|A-K\|_{op}\le \frac{\lambda(K)}{2}\}$ the matrix $\Gamma_t$ is strictly positive definite for every $t\in[0,1]$ if $K$ is strictly positive definite. In fact, for every $x\in\mathbb{R}^d$ with $\|x\|=1$ one has that 
\begin{equation}\label{min_eig_G}
x^T\Gamma_tx=tx^T(A-K)x+x^TKx\ge \lambda(K)-\|A-K\|_{op}\ge \frac{\lambda(K)}{2}>0,
\end{equation}
yielding that, for every $\omega\in E$, the function $\phi_{\Gamma_t(\omega)}$ ( see \eqref{phi_M}), is well-defined. To study the first term in \eqref{spiezzo in due}, we define the following class of interpolating functions:
\begin{equation}\label{tilde_g}
\tilde{g}(A,t,x):=\frac{\mathbb{E}[\phi_{\Gamma_t}(x)1_E]}{\mathbb{P}(E)\phi_K(x)},\quad t\in[0,1],\,x\in\mathbb{R}^d,
\end{equation}
and
\begin{equation}\label{fun_psi}
\psi(A,t,x):=\mathbb{P}(E)\tilde{g}(A,t,x)\log(\tilde{g}(A,t,x)),\quad t\in[0,1],x\in\mathbb{R}^d.
\end{equation}
Then, observing that $\psi(0,x)=0$ for every $x\in\mathbb{R}^d$, one deduces that
\[
\int_{\mathbb{R}^d}{\mathbb{E}[\phi_A(x)1_E]}\log\Bigg(\frac{\mathbb{E}[\phi_A(x)1_E]}{\mathbb{P}(E)\phi_K(x)}\Bigg)dx=\mathbb{E}[\psi(1,G)-\psi(0,G)];
\] 
the strategy of Proof of Lemma \ref{inE} is then {to use the Taylor expansion to the order three of $\psi$ around $t=0$, and to obtain an appropriate control of the remainder.}

\,

The derivability of $\psi$ in $t$ is a consequence of the next Lemma (proved in Section \ref{proof_lemma_pol_bound_hk}) and of the Remark immediately after.
Let us now define
\begin{equation}\label{fun_g}
g(A,t,x):=\frac{\phi_{\Gamma_t}(x)}{\phi_K(x)}
\end{equation}
and
\begin{equation}\label{def_h_k}
h_k(A,t,x):=\frac{1}{g(A,t,x)}\frac{\partial^k g}{\partial t^k}(A,t,x),
\end{equation}
noticing that $\tilde{g}(A,t,x)=\mathbb{E}[g(A,t,x)\frac{1_E}{\mathbb{P}(E)}]$.
\begin{lemma}\label{pol_bound_hk}
{For every integer $k\ge 1$, there exists a positive polynomial $p_k:\mathbb{R}\to\mathbb{R}$ such that, on the event $E:=\Big\{\|A-K\|_{op}\le \frac{\lambda(K)}{2}\Big\}$, one has the bound
\[
|h_k(A,t,x)|\le p_k(\|x\|)\|A-K\|_{HS}^k.
\]
}
\end{lemma}
\begin{remark}\label{pass_sotto_deriv}
For $k\ge 1$ integer, $t\in [0,1]$ and $x\in\mathbb{R}^d$, recalling definitions \eqref{tilde_g} and \eqref{fun_g}, respectively, for $\tilde{g}$ and $g$, one has that
\[
\frac{\partial^k \tilde{g}}{\partial t^k}(A,t,x)=\frac{\partial^k}{\partial t^k}\mathbb{E}\Big[g(A,t,x)\frac{1_E}{\mathbb{P}(E)}\Big]=\mathbb{E}\Big[\frac{\partial^k g}{\partial t^k}(A,t,x)\frac{1_E}{\mathbb{P}(E)}\Big].
\]
To see this, one can use the fact that (by virtue of \eqref{min_eig_G} and denoting by $\lambda(\Gamma_t)$ the minimum eigenvalue of $\Gamma_t$) on the event $E$ one has the bound $\lambda(\Gamma_t)\ge \frac{\lambda(K)}{2}$ and therefore, using Lemma \ref{pol_bound_hk},
\begin{equation}\label{passo sotto A}
|h_k(A,t,x)|g(A,t,x)\le p_k(\|x\|) \Big(\frac{\lambda(K)\sqrt{d}}{2}\Big)^k\frac{1}{(\pi\lambda(K))^{d/2}}\frac{1}{\phi_K(x)}.
\end{equation}
 We observe that the quantity on the right-hand side of \eqref{passo sotto A} does not depend on $t$ and it is integrable with respect {to the law of $A$}, in such a way that it is possible to pass the derivative under the sign of integral.
\end{remark}

\begin{lemma}\label{lemm_der_psi}
If $K$ is invertible, then $\psi\in C^{\infty}(\mathbb{R}^d)$. In particular, if $k\ge 4$,
\[
\frac{\partial \psi}{\partial t}(A,t,x)=\mathbb{P}(E)\Bigg(\frac{\partial\tilde{g}}{\partial t}(A,t,x)\log(\tilde{g}(A,t,x))+\frac{\partial \tilde{g}}{\partial t}(A,t,x)\Bigg),
\]
\[
\frac{\partial^2 \psi}{\partial t^2}(A,t,x)=\mathbb{P}(E)\Bigg(\frac{\partial^2\tilde{g}}{\partial t^2}(A,t,x)\log(\tilde{g}(A,t,x))+\frac{1}{\tilde{g}(A,t,x)}\Big(\frac{\partial \tilde{g}}{\partial t}(A,t,x)\Big)^2+\frac{\partial^2 \tilde{g}}{\partial t^2}(A,t,x)\Bigg),
\]
\begin{multline*}
\frac{\partial^3 \psi}{\partial t^3}(A,t,x)=\mathbb{P}(E)\Bigg(\frac{\partial^3\tilde{g}}{\partial t^3}(A,t,x)\log(\tilde{g}(A,t,x))+\frac{3}{\tilde{g}(A,t,x)}\frac{\partial^2 \tilde{g}}{\partial t^2}(A,t,x)\frac{\partial\tilde{g}}{\partial t}(A,t,x)\\
-\frac{1}{(\tilde{g}(A,t,x))^2}\Big(\frac{\partial \tilde{g}}{\partial t}(A,t,x)\Big)^3
+\frac{\partial^3 \tilde{g}}{\partial t^3}(A,t,x)\Bigg),
\end{multline*}
\begin{multline*}
\frac{\partial^4 \psi}{\partial t^4}(A,t,x)=\mathbb{P}(E)\Bigg(\frac{\partial^4\tilde{g}}{\partial t^4}(A,t,x)\log(\tilde{g}(A,t,x))+\frac{4}{\tilde{g}(A,t,x)}\frac{\partial^3 \tilde{g}}{\partial t^3}(A,t,x)\frac{\partial\tilde{g}}{\partial t}(A,t,x)\\
+\frac{3}{\tilde{g}(A,t,x)}\Big(\frac{\partial^2 \tilde{g}}{\partial t^2}(A,t,x)\Big)^2
-\frac{6}{(\tilde{g}(A,t,x))^2}\frac{\partial^2 \tilde{g}}{\partial t^2}(A,t,x)\Big(\frac{\partial\tilde{g}}{\partial t}(A,t,x)\Big)^2+\frac{\partial^4 \tilde{g}}{\partial t^4}(A,t,x)\\
+\frac{2}{(\tilde{g}(A,t,x))^3}\Big(\frac{\partial\tilde{g}}{\partial t}(A,t,x)\Big)^4\Bigg).
\end{multline*}
\end{lemma}

\medskip

As a consequence of the previous statement, using the Taylor expansion in $t=0$ of $\psi(A,t,x)$ one deduces the identity
\begin{multline*}
\int_{\mathbb{R}^d}{\mathbb{E}[\phi_A(x)1_E]}\log \Bigg(\frac{\mathbb{E}[\phi_A(x)1_E]}{\mathbb{P}(E)\phi_K(x)}\Bigg)dx\\
=\mathbb{E}\Bigg[\psi(A,0,G)+\frac{\partial \psi}{\partial t}(A,0,G)+\frac{1}{2}\frac{\partial^2 \psi}{\partial t^2}(A,0,G)
+\frac{1}{6}\frac{\partial^3 \psi}{\partial t^3}(A,0,G)
+\frac{1}{24}\frac{\partial \psi}{\partial t^4}(A,\eta,G)\Bigg],
\end{multline*}
with $\eta\in (0,1)$. In this expression, certain terms exhibit a general structure and will therefore be studied in the next section, using the tools introduced in the following remarks and proposition.

The first remark illustrates an application of a known relation between the derivatives of the Gaussian density with respect to the covariance matrix and those with respect to the argument $x$, as found in \cite{john}. For completeness, we provide a full proof here.

\begin{remark}\label{der_t_hess}
Recalling notations  \eqref{phi_M} and \eqref{def_Gt}, one has that
\begin{equation}\label{der_t_x_dens}
\frac{\partial \phi_{\Gamma_t}(x)}{\partial t}=\frac{1}{2}tr\Big((A-K)\nabla^2 \phi_{\Gamma_t}(x)\Big),
\end{equation}
where $\nabla^2 \phi_{\Gamma_t}(x)$ is the Hessian matrix of $\phi_{\Gamma_t}$ in $x$. To see this, one can use the relation
\[
\frac{\partial \phi_{\Gamma_t}(x)}{\partial t}=\frac{1}{2}\Big(\langle x,\Gamma_t^{-1}(A-K)\Gamma_t^{-1}x\rangle-tr(\Gamma_t^{-1}(A-K))\Big)\phi_{\Gamma_t}(x)
\]
and
\[
tr\Big((A-K)\nabla^2 \phi_{\Gamma_t}(x)\Big)=\sum_{i=1}^{d}\sum_{j=1}^{d}(A-K)_{i,j}\frac{\partial^2\phi_{\Gamma_t}}{\partial x_i\partial x_j}(x)
\]
\[
=\sum_{i=1}^{d}\sum_{j=1}^{d}(A-K)_{i,j}\Big(\sum_{r=1}^{d}\sum_{s=1}^{d}\phi_{\Gamma_t}(x)(\Gamma_t^{-1})_{j,r}(\Gamma_t^{-1})_{i,s}x_rx_s-\phi_{\Gamma_t}(x)(\Gamma_t^{-1})_{i,j}\Big)
\]
\[
=\sum_{r=1}^{d}\sum_{s=1}^{d}\phi_{\Gamma_t}(x)(\Gamma_t^{-1}(A-K)\Gamma_t^{-1})_{r,s}x_rx_s-\phi_{\Gamma_t}(x)tr(\Gamma_t^{-1}(A-K))=2\frac{\partial \phi_{\Gamma_t}(x)}{\partial t},
\]
which yields the desired identity.
\end{remark}

\begin{remark}\label{media_deriv_nulla}
Recalling that $A$ is assumed to be independent of $G$, for $k\ge 1$ integer and $t\in [0,1]$, on the event $E$ defined in \eqref{def_ev_E}, one has that
\[
\mathbb{E}\Big[\frac{\partial^k g}{\partial t^k}(A,t,G)\big|A\Big]=\mathbb{E}\Big[\frac{\partial^k g}{\partial t^k}(M,t,G)\Big]_{\big|_{M=A}}=0,
\]
where $g$ is defined in \eqref{fun_g}.
To see this, we start by proving that $|\frac{\partial^k g}{\partial t^k}(A,t,G)|$ is bounded by a quantity which is independent of $t$ and that it is integrable with respect to the law of $G$. Using Lemma \ref{pol_bound_hk} and proceeding as in \eqref{passo sotto A}, we infer the bound
\[
\Big|\frac{\partial^k g}{\partial t^k}(A,t,G)\Big|=|h_k(A,t,G)|\frac{\phi_{\Gamma_t}(G)}{\phi_K(G)}\le  p_k(\|G\|)\|A-K\|_{HS}^k\frac{e^{-\frac{1}{2}\langle G, \Gamma_t^{-1}G\rangle}}{(\pi\lambda(K))^{k/2}\phi_K(G)}.
\]
Moreover, $\langle G,\Gamma_t^{-1}G\rangle\ge \lambda(\Gamma_t^{-1})\|G\|^2\ge\frac{2}{2\|K\|_{op}+\lambda(K)}\|G\|^2$, where we have used the fact that, in this case, 
\[
\lambda(\Gamma_t)\le \|\Gamma_t\|_{op}\le\|A-K\|_{op}+\|K\|_{op}\le \frac{\lambda(K)}{2}+\|K\|_{op}.
\]
As a consequence,
\[
\Big|\frac{\partial^k g}{\partial t^k}(A,t,G)\Big|\le  p_k(\|G\|)\|A-K\|_{HS}^k\frac{e^{-\frac{1}{2\|K\|_{op}+\lambda(K)}\|G\|^2}}{(\pi\lambda(K))^{k/2}\phi_K(G)},
\]
which is integrable {with respect to the law of $G$} and does not depend on $t$, as desired. We now switch derivative and integral to obtain the chain of equalities
\[
\mathbb{E}\Big[\frac{\partial^k g}{\partial t^k}(A,t,G)\big|A\Big]=\frac{\partial^k }{\partial t^k}\mathbb{E}\Big[g(A,t,G)\big|A\Big]=\frac{\partial^k }{\partial t^k}\int_{\mathbb{R}^d}\phi_{tA+(1-t)K}(x)dx=0,
\]
thus concluding the argument.
\end{remark}

\begin{prop}\label{var_hk}
Fix $k\ge 1$ as well as a random vectors $N\sim\mathcal{N}_d(0,I_{d})$ independent of $A$. Then,
\[
\mathbb{E}\Big[\Big(h_k(A,t,F_t)\Big)^21_\mathbb{E}\Big]=\frac{k!}{2^k}\mathbb{E}\Big[\Big(\langle N,\sqrt{\Gamma_t}^{-1}(A-K)\Gamma_t^{-1}(A-K)\sqrt{\Gamma_t}^{-1}N\rangle\Big)^k 1_\mathbb{E}\Big]
\]
\end{prop}
\smallskip 

The proof uses the following classical result.
\begin{theorem}[\bf Isserlis' Theorem, see \cite{Isserlis}]\label{iss_th}
If $X:=(X_1,\dots,X_{2k})\sim\mathcal{N}_{2k}(0,M)$ in $\mathbb{R}^{2k}$ with $k\ge 1$  integer, then
\[
\mathbb{E}\Big[\Pi_{i=1}^{2k} X_i\Big]=\frac{1}{2^{k}k!}\sum_{\sigma\in\Sigma_{2k}}\Pi_{j=1}^{k}M_{\sigma(2j-1),\sigma(2j)},
\]
where $\Sigma_{2k}$ is the set of all the permutations of $2k$ elements.
\end{theorem}

\smallskip

\begin{proof}[Proof of Proposition \ref{var_hk}]
Using the identity \eqref{der_t_x_dens}, exchanging the derivative with respect to $t$ with the {derivatives} with respect to $x$, whose { associated gradient} is indicated with $\nabla$, and using an { recursive} argument on $k$, it is easy to prove that on the event $E$
\begin{equation}\label{der_g_k}
\frac{\partial^k g}{\partial t^k}(A,t,x)=\frac{1}{2^k}\Big\langle (A-K)^{\otimes k},\frac{\nabla^{2k}\phi_{\Gamma_t}(x)}{\phi_{\Gamma_t}(x)}\Big\rangle g(A,t,x),
\end{equation}
where, in general, if $M\in\mathbb{R}^{d\times d}$ is a matrix, then $M^{\otimes k}\in\mathbb{R}^{(d\times d)^k}$ is defined as 
\[
(M^{\otimes k})_{i_1,i_2,\dots, i_{2k-1}i_{2k}}:=M_{i_1,i_2}\dots M_{i_{2k-1},i_{2k}}
\]
 for arbitrary indices $i_1,\dots,i_{2k}\in\{1,\dots,d\}$ and the scalar product $\langle \cdot,\cdot\rangle$ is the scalar product in $\mathbb{R}^{(d\times d)^k}$. To see this, consider first the case $k=1$: using \eqref{der_t_x_dens}, one infers that
\[
\frac{\partial g}{\partial t}(A,t,x)=\frac{1}{\phi_K(x)}\frac{\partial}{\partial t}\phi_{\Gamma_t}(x)=\frac{1}{2}tr\Bigg((A-K)\frac{\nabla^2\phi_{\Gamma_t}(x)}{\phi_K(x)}\Bigg).\\
\]
Assuming now that identity \eqref{der_g_k} is true for $k-1$, using once more \eqref{der_t_x_dens}, one deduces that
\begin{multline*}
\frac{\partial^k g}{\partial t^k}(A,t,x)=\frac{1}{2^{k-1}}\Big\langle (A-K)^{\otimes (k-1)},\frac{\partial}{\partial t}\frac{\nabla^{2(k-1)}\phi_{\Gamma_t}(x)}{\phi_{K}(x)}\Big\rangle\\
=\frac{1}{2^{k}}\Big\langle (A-K)^{\otimes (k-1)},{\nabla^{2(k-1)}tr\Big((A-K)\nabla^{2}\phi_{\Gamma_t}(x)}\Big)\Big\rangle\frac{1}{\phi_K(x)}
\end{multline*}
and identity \eqref{der_g_k} easily follows. Exploiting the fact that $\phi_{\Gamma_t}(x)=\phi_{I_d}(\sqrt{\Gamma_t}^{-1}x)\frac{1}{\sqrt{det(\Gamma_t)}}$, again by induction on $k$, it is easy to see that
\[
\Big\langle (A-K)^{\otimes k},\frac{\nabla^{2k}\phi_{\Gamma_t}(x)}{\phi_{\Gamma_t}(x)}\Big\rangle =\Big\langle \Big(\sqrt{\Gamma_t}^{-1}(A-K)\sqrt{\Gamma_t}^{-1}\Big)^{\otimes k},\frac{\nabla^{2k}\phi_{I_d}}{\phi_{I_d}}(\sqrt{\Gamma_t}^{-1}x)\Big\rangle .
\]
This yields that, for $N\sim\mathcal{N}_d(0,I_d)$ independent of $A$, using definition \eqref{def_h_k},
\begin{multline}\label{second_ind_hk}
\mathbb{E}\Big[\Big(h_k(A,t,F_t)\Big)^21_\mathbb{E}\Big]=\mathbb{E}\Big[\Big(\frac{1}{g(A,t,F_t)}\frac{\partial^k g}{\partial t^k}(A,t,F_t)\Big)^21_\mathbb{E}\Big]\\
=\mathbb{E}\Big[\Big(\frac{1}{2^k}\Big\langle \Big(\sqrt{\Gamma_t}^{-1}(A-K)\sqrt{\Gamma_t}^{-1}\Big)^{\otimes k},\frac{\nabla^{2k}\phi_{I_d}}{\phi_{I_d}}(N)\Big\rangle\Big)^21_\mathbb{E}\Big].
\end{multline}
Let us now define
\[
M:=\sqrt{\Gamma_t}^{-1}(A-K)\sqrt{\Gamma_t}^{-1}
\]
and for every multi-index, 
\[
J\in S^{(2k)}:= \Big\{J:=(j_1,\dots,j_d)\in \mathbb{N}_0^d: j_1+\dots +j_d=2k\Big\},
\]
let us define
\[
A_J:=\Big\{\alpha:=(\alpha_1,\dots,\alpha_{2k})\in \{1,\dots,d\}^{2k}: \sum_{r=1}^{2k}1_{\{\alpha_r=s\}}=j_s \quad \forall s=1,\dots,d\Big\}.
\]
Then, from \eqref{second_ind_hk} it follows that
\[
\mathbb{E}\Big[\Big(h_k(A,t,F_t)\Big)^21_\mathbb{E}\Big]\]
\[=\frac{1}{2^{2k}}\mathbb{E}\Bigg[\Bigg(\sum_{i_1=1}^d\dots\sum_{i_{2k}=1}^d M_{i_1,i_2}\dots M_{i_{2k-1},i_{2k}}\frac{1}{\phi_{I_d}(N)}\frac{\partial^{2k} \phi_{I_d}}{\partial x_{i_1}\dots \partial x_{i_{2k}}}(N)\Bigg)^21_\mathbb{E}\Bigg]
\]
\[
=\frac{1}{2^{2k}}\mathbb{E}\Bigg[\Bigg(\sum_{J\in S^{(2k)}}\sum_{\alpha\in A_J} M_{\alpha_1,\alpha_2}\dots M_{\alpha_{2k-1},\alpha_{2k}}\frac{1}{\phi_{I_d}(N)}\frac{\partial^{2k} \phi_{I_d}}{\partial x_{1}^{j_1}\dots \partial x_{d}^{j_d}}(N)\Bigg)^21_\mathbb{E}\Bigg]
\]
\[
=\frac{1}{2^{2k}}\mathbb{E}\Bigg[\Bigg(\sum_{J\in S^{(2k)}}\sum_{\alpha\in A_J} M_{\alpha_1,\alpha_2}\dots M_{\alpha_{2k-1},\alpha_{2k}}H_{j_1}(N_1)\dots H_{j_d}(N_d)\Bigg)^21_\mathbb{E}\Bigg],
\]
where, in the last equality, we used property \eqref{her_mult_uni} of the multivariate Hermite polynomials. Using now Proposition \ref{var_her} and the independence between $A$ and $N$, we infer that
\begin{multline*}
\mathbb{E}\Big[\Big(h_k(A,t,F_t)\Big)^21_\mathbb{E}\Big]\\
=\mathbb{E}\Bigg[\frac{1}{2^{2k}}\sum_{J\in S^{(2k)}}\sum_{\alpha\in A_J}\sum_{\beta\in A_J} M_{\alpha_1,\alpha_2}\dots M_{\alpha_{2k-1},\alpha_{2k}}M_{\beta_1,\beta_2}\dots M_{\beta_{2k-1},\beta_{2k}}{j_1!}\dots {j_d!}1_\mathbb{E}\Bigg].
\end{multline*}
Let us now observe that once $\alpha\in A_J$ is fixed, every element in $A_J$ {is uniquely characterized by a} permutation of $\alpha=(\alpha_1,\dots,\alpha_{2k})$ and hence the sum over $\beta\in A_J$ can be replaced with the sum over all permutations of $2k$ elements, $\Sigma_{2k}$, divided by $j_1!\dots j_d!$ which is the number of permutations of $\alpha$ that exchange identical elements. As a consequence,
\begin{multline*}
\mathbb{E}\Big[\Big(h_k(A,t,F_t)\Big)^21_\mathbb{E}\Big]\\
=\mathbb{E}\Bigg[\frac{1}{2^{2k}}\sum_{J\in S^{(2k)}}\sum_{\alpha\in A_J}\sum_{\sigma\in \Sigma_{2k}} M_{\alpha_1,\alpha_2}\dots M_{\alpha_{2k-1},\alpha_{2k}}M_{\alpha_{\sigma(1)},\alpha_{\sigma(2)}}\dots M_{\alpha_{\sigma(2k-1)},\alpha_{\sigma(2k)}}1_\mathbb{E}\Bigg]
\end{multline*}
\[
=\mathbb{E}\Bigg[\frac{1}{2^{2k}}\sum_{i_1=1}^d\dots\sum_{i_{2k}=1}^d\sum_{\sigma\in \Sigma_{2k}} M_{i_1,i_2}\dots M_{i_{2k-1},i_{2k}}M_{i_{\sigma(1)},i_{\sigma(2)}}\dots M_{i_{\sigma(2k-1)},i_{\sigma(2k)}}1_\mathbb{E}\Bigg]
\]
\begin{equation}\label{eq_con_is}
=\frac{k!}{2^k}\sum_{i_1=1}^d\dots\sum_{i_{2k}=1}^d \mathbb{E}\Big[M_{i_1,i_2}\dots M_{i_{2k-1},i_{2k}}(\sqrt{M}N)_{i_1}\dots(\sqrt{M}N)_{i_{2k}}\Big]
\end{equation}
\[
=\frac{k!}{2^k} \mathbb{E}\Big[\Big(\langle \sqrt{M}N,M\sqrt{M}N\rangle\Big)^k1_\mathbb{E}\Big]=\frac{k!}{2^k}\mathbb{E}\Big[\Big(\langle N,M^2 N\rangle\Big)^k1_\mathbb{E}\Big],
\]
where we used Theorem \ref{iss_th} to derive equation \eqref{eq_con_is}, with $N\sim\mathcal{N}_d(0,I_{d})$ independent of $A$.

\end{proof}

The following Lemma is a consequence of Proposition \ref{var_hk}, and is proved in Section \ref{proof_lemma_mom_sec_h}.

\begin{lemma}\label{mom_sec_h}
For every $t\in [0,1]$ it holds that
\[
\mathbb{E}[(h_1(A,t,F_t))^21_E]=\frac{1}{2}\mathbb{E}[tr((\Gamma_t^{-1}(A-K))^2)1_E],
\]
\[
\mathbb{E}[(h_1(A,t,F_t))^41_E]=
3\mathbb{E}[tr((\Gamma_t^{-1}(A-K))^4)1_E]+\frac{3}{4}\mathbb{E}[(tr((\Gamma_t^{-1}(A-K))^2))^21_E]\\,
\]
\begin{multline*}
\mathbb{E}[(h_1(A,t,F_t))^61_E]=60 \mathbb{E}\Big[tr((\Gamma_t^{-1}(A-K))^6)1_\mathbb{E}\Big]+\frac{15}{8}\mathbb{E}\Big[(tr((\Gamma_t^{-1}(A-K))^2))^31_\mathbb{E}\Big]\\
+10\mathbb{E}\Big[(tr((\Gamma_t^{-1}(A-K))^3))^21_\mathbb{E}\Big]
+\frac{45}{2}\mathbb{E}\Big[tr((\Gamma_t^{-1}(A-K))^2)tr((\Gamma_t^{-1}(A-K))^4)1_\mathbb{E}\Big],
\end{multline*}
\begin{equation}\label{id_h22}
\mathbb{E}[(h_2(A,t,F_t))^21_E]=\mathbb{E}[tr((\Gamma_t^{-1}(A-K))^4)1_E]+\frac{1}{2}\mathbb{E}[(tr((\Gamma_t^{-1}(A-K))^2))^21_E],
\end{equation}
\begin{multline*}
\mathbb{E}[(h_3(A,t,F_t))^21_E]=\frac{3}{4}\mathbb{E}\Big[\Big(tr\big((\Gamma^{-1}_t(A-K))^2\big)\Big)^31_\mathbb{E}\Big]+6\mathbb{E}\Big[tr\big((\Gamma^{-1}_t(A-K))^6\big)1_\mathbb{E}\Big]\\
+\frac{9}{2}\mathbb{E}\Big[tr\big((\Gamma^{-1}_t(A-K))^2\big)tr\big((\Gamma^{-1}_t(A-K))^4\big)1_\mathbb{E}\Big],\\
\end{multline*}
\begin{multline}\label{h4_lemma}
\mathbb{E}[(h_4(A,t,F_t))^21_E]=\frac{3}{2}\mathbb{E}\Big[\Big(tr\big((\Gamma^{-1}_t(A-K))^2\big)\Big)^41_\mathbb{E}\Big]\\
+18\mathbb{E}\Big[\Big(tr\big((\Gamma^{-1}_t(A-K))^2\big)\Big)^2tr\big((\Gamma^{-1}_t(A-K))^4\big)1_\mathbb{E}\Big]\\
+18\mathbb{E}\Big[\Big(tr\big((\Gamma^{-1}_t(A-K))^4\big)\Big)^21_\mathbb{E}\Big]
+72\mathbb{E}\Big[tr\big((\Gamma^{-1}_t(A-K))^8\big)1_\mathbb{E}\Big]\\
+48\mathbb{E}\Big[tr\big((\Gamma^{-1}_t(A-K))^2\big)tr\big((\Gamma^{-1}_t(A-K))^6\big)1_\mathbb{E}\Big].
\end{multline}
\end{lemma}

\subsubsection{Proof of Lemma \ref{inE}}\label{ss:pline}
Performing a Taylor expansion in $t=0$ of the function $\psi(A,t,x)$ defined in \eqref{fun_psi}, one obtains that
\begin{multline*}
\int_{\mathbb{R}^d}{\mathbb{E}[\phi_A(x)1_E]}\log \Bigg(\frac{\mathbb{E}[\phi_A(x)1_E]}{\mathbb{P}(E)\phi_K(x)}\Bigg)dx\\
=\mathbb{E}\Big[\psi(A,0,G)+\frac{\partial \psi}{\partial t}(A,0,G)+\frac{1}{2}\frac{\partial^2 \psi}{\partial t^2}(A,0,G)
+\frac{1}{6}\frac{\partial^3 \psi}{\partial t^3}(A,0,G)+\frac{1}{24}\frac{\partial \psi}{\partial t^4}(A,\eta,G)\Big]
\end{multline*}
where $\eta\in [0,1]$. Then, using Lemma \ref{lemm_der_psi}, Remark \eqref{pass_sotto_deriv} and Remark \eqref{media_deriv_nulla} one has that
\begin{multline}\label{taylor_dec_psi}
\int_{\mathbb{R}^d}{\mathbb{E}[\phi_A(x)1_E]}\log \Bigg(\frac{\mathbb{E}[\phi_A(x)1_E]}{\mathbb{P}(E)\phi_K(x)}\Bigg)dx=\mathbb{P}(E)\Bigg(\frac{1}{2}\mathbb{E}\Big[\Big(\frac{\partial \tilde{g}}{\partial t}(A,0,G)\Big)^2\Big]\\
+\frac{1}{2}\mathbb{E}\Big[\frac{\partial^2 \tilde{g}}{\partial t^2}(A,0,G)\frac{\partial\tilde{g}}{\partial t}(A,0,G)\Big]
-\frac{1}{6}\mathbb{E}\Big[\Big(\frac{\partial \tilde{g}}{\partial t}(A,0,G)\Big)^3\Big]+
\frac{1}{24}\mathbb{E}\Big[\frac{\partial^4\tilde{g}}{\partial t^4}(A,\eta,G)\log(\tilde{g}(A,\eta,G))\Big]\\
+\frac{1}{6}\mathbb{E}\Big[\frac{1}{\tilde{g}(A,\eta,G)}\frac{\partial^3 \tilde{g}}{\partial t^3}(A,\eta,G)\frac{\partial\tilde{g}}{\partial t}(A,\eta,G)\Big]+\frac{1}{8}\mathbb{E}\Big[\frac{1}{\tilde{g}(A,\eta,G)}\Big(\frac{\partial^2 \tilde{g}}{\partial t^2}(A,\eta,G)\Big)^2\Big]\\
-\frac{1}{4}\mathbb{E}\Big[\frac{1}{(\tilde{g}(A,\eta,G))^2}\frac{\partial^2 \tilde{g}}{\partial t^2}(A,\eta,G)\Big(\frac{\partial\tilde{g}}{\partial t}(A,\eta,G)\Big)^2\Big]+\frac{1}{12}\mathbb{E}\Big[\frac{1}{(\tilde{g}(A,\eta,G))^3}\Big(\frac{\partial\tilde{g}}{\partial t}(A,\eta,G)\Big)^4\Big]\Bigg).
\end{multline}
Note that, in the previous computation, we assumed that all the summands are integrable: in the subsequent lemmas, it is proved that this is the case, as soon as $\mathbb{E}[\|A\|_{HS}^8]<\infty$. The following technical statements focus on the terms appearing on the right-hand side of \eqref{taylor_dec_psi}; they will be proved in Section \ref{proof_technical}.

\begin{lemma}\label{term_log}
\begin{multline*}
\mathbb{P}(E)\mathbb{E}\Big[\frac{\partial^4\tilde{g}}{\partial t^4}(A,\eta,G)\log(\tilde{g}(A,\eta,G))\Big]\\
\le \frac{3\sqrt{70}}{2}\mathbb{E}\Big[\|\Gamma^{-1}_\eta(A-K)\|_{HS}^81_\mathbb{E}\Big]^{1/2}\Bigg(\frac{1}{2}\max{\Big\{\Big|\log\frac{2^d det K}{(2\|K\|_{op}+\lambda(K))^d}\Big|,\log\frac{2^d detK}{\lambda(K)^d}\Big\}}\\
+\frac{\sqrt{d}(\sqrt{2}+1)}{4}\|K^{-1}\|_{HS}\lambda(K)\Bigg).
\end{multline*}
\end{lemma}

\begin{lemma}\label{stime_per_th_fin}
For $i,j\in\{1,2,3\}$, $k\ge 1$ integer and $\eta\in [0,1]$, recalling definition \eqref{def_h_k},
\begin{multline}\label{gen_for_prod}
\mathbb{P}(E)\mathbb{E}\Big[\frac{1}{(\tilde{g}(A,\eta,G))^k}\frac{\partial^i \tilde{g}}{\partial t^i}(A,\eta,G)\Big(\frac{\partial^j\tilde{g}}{\partial t^j}(A,\eta,G)\Big)^k\Big]\\
\le \mathbb{E}\Bigg[(h_i(A,\eta,G))^2{g(A,\eta,G){1_E}}\Bigg]^{1/2}\mathbb{E}\Bigg[(h_j(A,\eta,G))^{2k}{g(A,\eta,G){1_E}}\Bigg]^{1/2}
\end{multline}
and in particular
\begin{equation}\label{h_quad}
\mathbb{P}(E)\mathbb{E}\Big[\frac{1}{\tilde{g}(A,\eta,G)}\Big(\frac{\partial^2 \tilde{g}}{\partial t^2}(A,\eta,G)\Big)^2\Big]\le  \frac{3}{2}\mathbb{E}\Big[\|\Gamma_\eta^{-1}(A-K)\|_{HS}^41_\mathbb{E}\Big],
\end{equation}
\begin{multline*}
\mathbb{P}(E)\mathbb{E}\Big[\frac{1}{\tilde{g}(A,\eta,G)}\frac{\partial^3 \tilde{g}}{\partial t^3}(A,\eta,G)\frac{\partial\tilde{g}}{\partial t}(A,\eta,G)\Big] \\
\le\frac{3 \sqrt{5}}{2\sqrt{2}}\mathbb{E}\Big[\|\Gamma^{-1}_\eta(A-K)\|_{HS}^61_\mathbb{E}\Big]^{1/2}\mathbb{E}\Big[\|\Gamma^{-1}_\eta(A-K)\|_{HS}^21_\mathbb{E}\Big]^{1/2},
\end{multline*}
\[
\mathbb{P}(E)\mathbb{E}\Big[\frac{1}{(\tilde{g}(A,\eta,G))^2}\frac{\partial^2 \tilde{g}}{\partial t^2}(A,\eta,G)\Big(\frac{\partial\tilde{g}}{\partial t}(A,\eta,G)\Big)^2\Big]\le  \frac{3\sqrt{5}}{2\sqrt{2}}\mathbb{E}\Big[\|\Gamma_\eta^{-1}(A-K)\|_{HS}^41_\mathbb{E}\Big],
\]
\begin{multline*}
\mathbb{P}(E)\mathbb{E}\Big[\frac{1}{(\tilde{g}(A,\eta,G))^3}\Big(\frac{\partial\tilde{g}}{\partial t}(A,\eta,G)\Big)^4\Big]\\
\le \frac{\sqrt{5}}{4}(10+\sqrt{3}+4\sqrt{6})\mathbb{E}\Big[\|\Gamma^{-1}_\eta(A-K)\|_{HS}^61_\mathbb{E}\Big]^{1/2}\mathbb{E}\Big[\|\Gamma^{-1}_\eta(A-K)\|_{HS}^21_\mathbb{E}\Big]^{1/2}.
\end{multline*}
\end{lemma}

\begin{lemma}\label{term_con_zero}
\begin{multline*}
{\mathbb{P}(E)}\Bigg(\frac{1}{2}\mathbb{E}\Big[\Big(\frac{\partial \tilde{g}}{\partial t}(A,0,G)\Big)^2\Big]-\frac{1}{6}\mathbb{E}\Big[\Big(\frac{\partial \tilde{g}}{\partial t}(A,0,G)\Big)^3\Big]+\frac{1}{2}\mathbb{E}\Big[\frac{\partial^2 \tilde{g}}{\partial t^2}(A,0,G)\frac{\partial\tilde{g}}{\partial t}(A,0,G)\Big]\Bigg)\\
\le \frac{\sqrt{3}}{12\sqrt{2}}\big(2\sqrt{6}+3\sqrt{2}+2+\sqrt{d}\big)\|K^{-1}\|_{HS}^2\Bigg( \|\mathbb{E}[A]-K\|_{HS}^2+\frac{8}{\lambda(K)^2}\mathbb{E}\Big[\|A-K\|_{HS}^2\Big]^2\\
+\frac{32}{\lambda(K)^4}\|K\|_{HS}^2\mathbb{E}\Big[\|A-K\|_{HS}^4\Big]\Bigg)+\frac{\sqrt{3}}{8}\|K^{-1}\|_{HS}^4\mathbb{E}\Big[\|A-K\|_{HS}^4\Big].
\end{multline*}
\end{lemma}

Applying Lemmas \ref{term_log}, \ref{stime_per_th_fin} and \ref{term_con_zero} to inequality \eqref{taylor_dec_psi}, one deduces that
\begin{multline*}
\int_{\mathbb{R}^d}{\mathbb{E}[\phi_A(x)1_E]}\log \Bigg(\frac{\mathbb{E}[\phi_A(x)1_E]}{\mathbb{P}(E)\phi_K(x)}\Bigg)dx\le
\frac{\sqrt{3}}{12\sqrt{2}}\big(2\sqrt{6}+3\sqrt{2}+2+\sqrt{d}\big)\|K^{-1}\|_{HS}^2\cdot\\
\cdot \Bigg( \|\mathbb{E}[A]-K\|_{HS}^2
+\frac{8}{\lambda(K)^2}\mathbb{E}\Big[\|A-K\|_{HS}^2\Big]^2
+\frac{32}{\lambda(K)^4}\|K\|_{HS}^2\mathbb{E}\Big[\|A-K\|_{HS}^4\Big]\Bigg)\\
+\frac{\sqrt{3}}{8}\|K^{-1}\|_{HS}^4\mathbb{E}\Big[\|A-K\|_{HS}^4\Big]
+ \frac{\sqrt{70}}{16}\mathbb{E}\Big[\|\Gamma^{-1}_\eta(A-K)\|_{HS}^81_\mathbb{E}\Big]^{1/2}\cdot\\
\cdot\Bigg(\frac{1}{2}\max{\Big\{\Big|\log\frac{2^d det K}{(2\|K\|_{op}+\lambda(K))^d}\Big|,\log\frac{2^d detK}{\lambda(K)^d}\Big\}}
+\frac{(\sqrt{2}+1){d}}{4}
\Bigg)\\
+\frac{\sqrt{5}}{24}\Big({5}+\frac{\sqrt{3}}{2}+2\sqrt{6}+3\sqrt{2}\Big)\mathbb{E}\Big[\|\Gamma^{-1}_\eta(A-K)\|_{HS}^61_\mathbb{E}\Big]^{1/2}\mathbb{E}\Big[\|\Gamma^{-1}_\eta(A-K)\|_{HS}^21_\mathbb{E}\Big]^{1/2}\\
+\frac{3}{16}(1
+ \sqrt{10})\mathbb{E}\Big[\|\Gamma_\eta^{-1}(A-K)\|_{HS}^41_\mathbb{E}\Big]
\end{multline*}
\begin{multline*}
\le\frac{\sqrt{3}}{12\sqrt{2}}\big(2\sqrt{6}+3\sqrt{2}+2+\sqrt{d}\big)\|K^{-1}\|_{HS}^2\Bigg( \|\mathbb{E}[A]-K\|_{HS}^2
+\frac{8}{\lambda(K)^2}\mathbb{E}\Big[\|A-K\|_{HS}^2\Big]^2\\
+\frac{32}{\lambda(K)^4}\|K\|_{HS}^2\mathbb{E}\Big[\|A-K\|_{HS}^4\Big]\Bigg)
+\frac{\sqrt{3}d^2}{8\lambda(K)^4}\mathbb{E}\Big[\|A-K\|_{HS}^4\Big]
+ \frac{\sqrt{70}d^2}{\lambda(K)^4}\mathbb{E}\Big[\|A-K\|_{HS}^8\Big]^{1/2}\cdot\\
\cdot\Bigg(\frac{1}{2}\max{\Big\{\Big|\log\frac{2^d det K}{(2\|K\|_{op}+\lambda(K))^d}\Big|,\log\frac{2^d detK}{\lambda(K)^d}\Big\}}
+\frac{(\sqrt{2}+1){d}}{4}
\Bigg)\\
+\frac{2d^2\sqrt{5}}{3\lambda(K)^4}\Big({5}+\frac{\sqrt{3}}{2}+2\sqrt{6}+3\sqrt{2}\Big)\mathbb{E}\Big[\|A-K\|_{HS}^6\Big]^{1/2}\mathbb{E}\Big[\|A-K\|_{HS}^2\Big]^{1/2}\\
+\frac{3d^2}{\lambda(K)^4}(1
+ \sqrt{10})\mathbb{E}\Big[\|A-K\|_{HS}^4\Big],
\end{multline*}
where we have used that, on the event $E$, one has that $\lambda(\Gamma_\eta)\ge \frac{\lambda(K)}{2}$, as proved in \eqref{min_eig_G}. Therefore 
\begin{multline*}
\int_{\mathbb{R}^d}{\mathbb{E}[\phi_A(x)1_E]}\log \Bigg(\frac{\mathbb{E}[\phi_A(x)1_E]}{\mathbb{P}(E)\phi_K(x)}\Bigg)dx\\
\le\frac{\sqrt{3}}{12\sqrt{2}}\big(2\sqrt{6}+3\sqrt{2}+2+\sqrt{d}\big)\|K^{-1}\|_{HS}^2\|\mathbb{E}[A]-K\|_{HS}^2\\
+\Bigg\{\frac{\sqrt{2}}{\lambda(K)^4\sqrt{3}}\big(2\sqrt{6}+3\sqrt{2}+2+\sqrt{d}\big)\|K^{-1}\|_{HS}^2(\lambda(K)^2+4\|K\|_{HS}^2)\\
+ \frac{\sqrt{70}d^2}{\lambda(K)^4}\Bigg(\frac{1}{2}\max{\Big\{\Big|\log\frac{2^d det K}{(2\|K\|_{op}+\lambda(K))^d}\Big|,\log\frac{2^d detK}{\lambda(K)^d}\Big\}}
+\frac{(\sqrt{2}+1){d}}{4}
\Bigg)\\
+\Big(3+\frac{\sqrt{3}}{8}+ 5\sqrt{10}+\frac{4\sqrt{30}}{3}+\frac{\sqrt{15}}{3}+\frac{10\sqrt{5}}{3}\Big)\frac{d^2}{\lambda(K)^4}\Bigg\}\mathbb{E}\Big[\|A-K\|_{HS}^8\Big]^{1/2}.
\end{multline*}

The subsequent section focuses on the proofs of some crucial ancillary results.

\subsection{Proofs of technical results}\label{proof_technical}
\subsubsection{Proof of Lemma \ref{pol_bound_hk}}\label{proof_lemma_pol_bound_hk}
We will show by induction on $k\ge 1$ the stronger result that there exist positive polynomials $p_k$, $\{r^{(i)}_k\}_{i\ge 1}:\mathbb{R}\to\mathbb{R}$ such that 
\begin{equation}\label{ind_pol}
|h_k(A,t,x)|\le p_k(\|x\|)\|A-K\|_{HS}^k\quad\text{and}\quad \Big|\frac{\partial^i h_k}{\partial t^i}(A,t,x)\Big|\le r^{(i)}_k(\|x\|)\|A-K\|_{HS}^{k+i}.
\end{equation}
To see this, recall definition \eqref{def_h_k} and observe that, when $k=1$,
\begin{multline*}
|h_1(A,t,x)|=\Big|\frac{1}{g(A,t,x)}\frac{\partial g}{\partial t}(A,t,x)\Big|=\frac{1}{2}\Big|\langle x, \Gamma_t^{-1}(A-K)\Gamma_t^{-1}x\rangle-tr(\Gamma_t^{-1}(A-K))\Big|\\
\le \frac{1}{2}\Big(\|x\|^2\|\Gamma_t^{-1}\|_{op}^2\|A-K\|_{op}+\|\Gamma_t^{-1}\|_{HS}\|A-K\|_{HS}\Big)\le \frac{1}{2}\Big(\frac{1}{\lambda(\Gamma_t)^2}\|x\|^2+\frac{\sqrt{d}}{\lambda(\Gamma_t)}\Big)\|A-K\|_{HS}?
\end{multline*}
Similarly, an induction argument shows that, for every integer $i\ge 1$, there exist constants $C^{(i)}_1, C^{(i)}_2\in\mathbb{R}$ such that
\[
\frac{\partial^i h_1}{\partial t^i}(A,t,x)= C^{(i)}_1\langle x, \Big(\Gamma_t^{-1}(A-K)\Big)^{i+1}\Gamma_t^{-1}x\rangle+C^{(i)}_2tr\Big(\big(\Gamma_t^{-1}(A-K)\big)^{i+1}\Big),
\]
yielding in turn that
\begin{multline*}
\Big|\frac{\partial^i h_1}{\partial t^i}(A,t,x)\Big|\le \Big(|C_1^{(i)}|\|x\|^2\|\Gamma_t^{-1}\|_{op}^{i+2}+|C_2^{(i)}|\|\Gamma_t^{-1}\|_{HS}^{i+1}\Big)\|A-K\|_{HS}^{i+1}\\
\le \Big(|C_1^{(i)}| \|x\|^2\frac{1}{\lambda(\Gamma_t)^{i+2}}+|C_2^{(i)}|\frac{d^{\frac{i+1}{2}}}{\lambda(\Gamma_t)^{i+1}}\Big)\|A-K\|_{HS}^{i+1}.
\end{multline*}
Hence, using that on the event $E$ one has that $\lambda(\Gamma_t)\ge \frac{\lambda(K)}{2}$ (as proved in \eqref{min_eig_G}) the property \eqref{ind_pol} is verified with 
\[
p_1(y):=\frac{2}{\lambda(K)^2}y^2+\frac{\sqrt{d}}{\lambda(K)} \quad\text{and}\quad r^{(i)}_1(y):=|C_1^{(i)}|y^2\frac{2^{i+2}}{\lambda(K)^{i+2}}+|C_2^{(i)}|\frac{2^{i+1}d^{\frac{i+1}{2}}}{\lambda(K)^{i+1}}
\]
 for any $i\ge 1$ integer. Now we observe that, for $k\ge 1$,
\begin{equation}\label{iter_pol_mie}
h_{k+1}(A,t,x)=\frac{\partial h_k}{\partial t}(A,t,x)+h_k(A,t,x)h_1(A,t,x),
\end{equation}
which is a consequence of the fact that, by definition,
\[
h_{k+1}(A,t,x)=\frac{1}{g(A,t,x)}\frac{\partial^{k+1}g}{\partial t^{k+1}}(A,t,x)=\frac{1}{g(A,t,x)}\frac{\partial}{\partial t}(h_k(A,t,x)g(A,t,x)).
\]
Consequently, assuming the inductive hypothesis \eqref{ind_pol} holds for some $k \geq 1$, and using \eqref{iter_pol_mie}, the property \eqref{ind_pol} is also verified for $k+1$.

\subsubsection{Proof of Lemma \ref{mom_sec_h}}\label{proof_lemma_mom_sec_h}
Using Proposition \ref{var_hk} and selecting a random element $N\sim\mathcal{N}_d(0,I_d)$ independent of $A$, one deduces that
\begin{multline*}
\mathbb{E}\Big[(h_1(A,t,F_t))^21_\mathbb{E}\Big]=\frac{1}{2}\mathbb{E}\Big[\langle N,\Big(\sqrt{\Gamma_t}^{-1}(A-K)\sqrt{\Gamma_t}^{-1}\Big)^2N\rangle 1_\mathbb{E}\Big]\\
=\frac{1}{2}\mathbb{E}\Big[tr\Big(\Big(\sqrt{\Gamma_t}^{-1}(A-K)\sqrt{\Gamma_t}^{-1}\Big)^2\Big)1_\mathbb{E}\Big],
\end{multline*}
\[
\mathbb{E}\Big[(h_2(A,t,F_t))^21_\mathbb{E}\Big]=\frac{1}{2}\mathbb{E}\Big[\Big(\langle N,\Big(\sqrt{\Gamma_t}^{-1}(A-K)\sqrt{\Gamma_t}^{-1}\Big)^2N\rangle\Big)^2 1_\mathbb{E}\Big]=\frac{1}{2}\mathbb{E}\Bigg[\Bigg(\sum_{i=1}^d \lambda_i N_i^2\Bigg)^21_\mathbb{E}\Bigg],
\]
where $\{\lambda_i\}_{i=1,\dots,d}$ are the eigenvalues of $M^2:=(\sqrt{\Gamma_t}^{-1}(A-K)\sqrt{\Gamma_t}^{-1})^2$, {where we have used the fact that a real symmetrical matrix can be diagonalized by an orthonormal matrix, as well as that the law of a standard Gaussian vector is invariant by orthogonal transformations}. As a consequence,
\begin{multline*}
\mathbb{E}\Big[(h_2(A,t,F_t))^21_\mathbb{E}\Big]=\frac{1}{2}\sum_{i=1}^d \mathbb{E}\Big[\lambda_i^2 N_i^41_\mathbb{E}\Big]+\frac{1}{2}\sum_{i=1}^d\sum_{j=1,j\ne i}^d \mathbb{E}\Big[\lambda_i\lambda_j N_i^2N_j^21_\mathbb{E}\Big]\\
=\mathbb{E}\Big[tr(M^4)1_\mathbb{E}\Big]+\frac{1}{2}\mathbb{E}\Big[\big(tr(M^2)\big)^21_\mathbb{E}\Big].
\end{multline*}
On the other hand,
\begin{multline*}
\mathbb{E}\Big[(h_3(A,t,F_t))^21_\mathbb{E}\Big]=\frac{3}{4}\mathbb{E}\Bigg[\Bigg(\sum_{i=1}^d \lambda_i (N_i^2-1)+tr(M^2)\Bigg)^31_\mathbb{E}\Bigg]
=\frac{3}{4}\mathbb{E}\Big[\Big(\sum_{i=1}^d \lambda_i (N_i^2-1)\Big)^31_\mathbb{E}\Big]\\
+\frac{3}{4}\mathbb{E}\Big[\big(tr(M^2)\big)^31_\mathbb{E}\Big]+\frac{9}{4}\mathbb{E}\Big[\Big(\sum_{i=1}^d \lambda_i (N_i^2-1)\Big)^2tr(M^2)1_\mathbb{E}\Big]+\frac{9}{4}\mathbb{E}\Big[\big(tr(M^2)\big)^2\sum_{i=1}^d \lambda_i (N_i^2-1)1_\mathbb{E}\Big]
\end{multline*}
\[
=\frac{3}{4}\sum_{i=1}^d\mathbb{E}\Big[\lambda_i^3(N_i^2-1)^31_\mathbb{E}\Big]+\frac{3}{4}\mathbb{E}\Big[\big(tr(M^2)\big)^31_\mathbb{E}\Big]+\frac{9}{4}\sum_{i=1}^d \mathbb{E}\Big[\lambda_i ^2(N_i^2-1)^2tr(M^2)1_\mathbb{E}\Big]
\]
\[
=6\mathbb{E}\Big[tr(M^6)1_\mathbb{E}\Big]+\frac{3}{4}\mathbb{E}\Big[\big(tr(M^2)\big)^31_\mathbb{E}\Big]+\frac{9}{2}\mathbb{E}\Big[tr(M^4)tr(M^2)1_\mathbb{E}\Big]
\]
and 
\begin{multline*}
\mathbb{E}\Big[(h_4(A,t,F_t))^21_\mathbb{E}\Big]=\frac{3}{2}\mathbb{E}\Bigg[\Bigg(\sum_{i=1}^d \lambda_i (N_i^2-1)+tr(M^2)\Bigg)^41_\mathbb{E}\Bigg]=\frac{3}{2}\mathbb{E}\Big[\Big(\sum_{i=1}^d \lambda_i (N_i^2-1)\Big)^41_\mathbb{E}\Big]\\
+\frac{3}{2}\mathbb{E}\Big[\big(tr(M^2)\big)^41_\mathbb{E}\Big]+6\mathbb{E}\Big[\Big(\sum_{i=1}^d \lambda_i (N_i^2-1)\Big)^3tr(M^2)1_\mathbb{E}\Big]+6\mathbb{E}\Big[\big(tr(M^2)\big)^3\sum_{i=1}^d \lambda_i (N_i^2-1)1_\mathbb{E}\Big]\\
+9\mathbb{E}\Big[\Big(\sum_{i=1}^d \lambda_i (N_i^2-1)\Big)^2\big(tr(M^2)\big)^21_\mathbb{E}\Big]
\end{multline*}
\begin{multline*}
=\frac{3}{2}\sum_{i=1}^d\mathbb{E}\Big[ \lambda_i^4 (N_i^2-1)^41_\mathbb{E}\Big]+\frac{9}{2}\sum_{i=1}^d\sum_{j=1,j\ne i}^d \mathbb{E}\Big[\lambda_i^2\lambda_j^2 (N_i^2-1)^2(N_j^2-1)^21_\mathbb{E}\Big]\\
+\frac{3}{2}\mathbb{E}\Big[\big(tr(M^2)\big)^41_\mathbb{E}\Big]
+6\sum_{i=1}^d\mathbb{E}\Big[\lambda_i^3 (N_i^2-1)^3tr(M^2)1_\mathbb{E}\Big]
+9\sum_{i=1}^d\mathbb{E}\Big[ \lambda_i^2 (N_i^2-1)^2\big(tr(M^2)\big)^21_\mathbb{E}\Big]
\end{multline*}
\begin{multline*}
=72 \mathbb{E}\Big[tr(M^8)1_\mathbb{E}\Big]+18\mathbb{E}\Big[\big(tr(M^4)\big)^21_\mathbb{E}\Big]+\frac{3}{2}\mathbb{E}\Big[\big(tr(M^2)\big)^41_\mathbb{E}\Big]+48\mathbb{E}\Big[tr(M^6)tr(M^2)1_\mathbb{E}\Big]\\
+18\mathbb{E}\Big[tr(M^4)\big(tr(M^2)\big)^21_\mathbb{E}\Big].
\end{multline*}
Finally, applying Proposition 2.7.13 and Corollary A.2.4 from \cite{NP12} it holds that 
\[
\mathbb{E}[(h_1(A,t,F_t))^41_E]
=3\mathbb{E}\Big[tr(M^4)1_\mathbb{E}\Big]+\frac{3}{4}\mathbb{E}\Big[(tr(M^2))^21_\mathbb{E}\Big],
\]
and 
\begin{multline*}
\mathbb{E}[(h_1(A,t,F_t))^61_E]=60 \mathbb{E}\Big[tr(M^6)1_\mathbb{E}\Big]+\frac{15}{8}\mathbb{E}\Big[(tr(M^2))^31_\mathbb{E}\Big]+10\mathbb{E}\Big[(tr(M^3))^21_\mathbb{E}\Big]\\
+\frac{45}{2}\mathbb{E}\Big[tr(M^2)tr(M^4)1_\mathbb{E}\Big].
\end{multline*}

\subsubsection{Proof of Lemma \ref{term_log}}
Using Remark \ref{pass_sotto_deriv}, definition \eqref{def_h_k} and assuming without loss of generality that $A$ and $G$, one obtains that
\begin{multline*}
\mathbb{P}(E)\mathbb{E}\Big[\frac{\partial^4\tilde{g}}{\partial t^4}(A,\eta,G)\log(\tilde{g}(A,\eta,G))\Big]\\
=\mathbb{P}(E)\mathbb{E}\Bigg[\mathbb{E}\Big[h_4(A,\eta,G)g(A,\eta,G)\frac{1_E}{\mathbb{P}(E)}\big|G\Big]\log\Big(\mathbb{E}\Big[g(A,\eta,G)\frac{1_E}{\mathbb{P}(E)}\big|G\Big]\Big)\Bigg]\\
=\mathbb{E}\Bigg[\frac{\mathbb{E}\Big[h_4(A,\eta,G)g(A,\eta,G)1_\mathbb{E}\big|G\Big]}{\mathbb{E}\Big[g(A,\eta,G)\frac{1_E}{\mathbb{P}(E)}\big|G\Big]}\mathbb{E}\Big[g(A,\eta,G)\frac{1_E}{\mathbb{P}(E)}\big|G\Big]\log\Big(\mathbb{E}\Big[g(A,\eta,G)\frac{1_E}{\mathbb{P}(E)}\big|G\Big]\Big)\Bigg]
\end{multline*}
\[
\le \mathbb{E}\Bigg[\frac{\mathbb{E}\Big[h_4(A,\eta,G)g(A,\eta,G)1_\mathbb{E}\big|G\Big]}{\mathbb{E}\Big[g(A,\eta,G)\frac{1_E}{\mathbb{P}(E)}\big|G\Big]}\mathbb{E}\Big[g(A,\eta,G)\frac{1_E}{\mathbb{P}(E)}\log\Big(g(A,\eta,G)\Big)\big|G\Big]\Bigg],
\]
thanks to the convexity of the function $x\mapsto x\log x$ and Jensen inequality. Hence, considering the explicit expression of ${g}$ given by \eqref{fun_g},
\begin{multline*}
\mathbb{P}(E)\mathbb{E}\Big[\frac{\partial^4\tilde{g}}{\partial t^4}(A,\eta,G)\log(\tilde{g}(A,\eta,G))\Big]
\le \mathbb{E}\Bigg[\frac{\mathbb{E}\Big[h_4(A,\eta,G)g(A,\eta,G)1_\mathbb{E}\big|G\Big]}{\mathbb{E}\Big[g(A,\eta,G)\frac{1_E}{\mathbb{P}(E)}\big|G\Big]}\cdot\\
\cdot \mathbb{E}\Big[g(A,\eta,G)\frac{1_E}{\mathbb{P}(E)}\Big(\frac{1}{2}\log \frac{det(K)}{det(\Gamma_\eta)}-\frac{1}{2}\langle G, (\Gamma_\eta^{-1}-K^{-1})G\rangle\Big)\big|G\Big]\Bigg]
\end{multline*}
\begin{multline*}
\le \frac{1}{2}\mathbb{E}\Bigg[\frac{\Big|\mathbb{E}\Big[h_4(A,\eta,G)g(\eta,G)1_\mathbb{E}\big|G\Big]\Big|}{\mathbb{E}\Big[g(A,\eta,G)\frac{1_E}{\mathbb{P}(E)}\big|G\Big]}\mathbb{E}\Big[g(A,\eta,G)\frac{1_E}{\mathbb{P}(E)}\Big|\log \frac{det(K)}{det(\Gamma_\eta)}\Big|\big|G\Big]\Bigg]\\
-\frac{1}{2}\mathbb{E}\Bigg[\frac{\mathbb{E}\Big[h_4(A,\eta,G)g(A,\eta,G)1_\mathbb{E}\big|G\Big]}{\mathbb{E}\Big[g(A,\eta,G)\frac{1_E}{\mathbb{P}(E)}\big|G\Big]}\mathbb{E}\Big[g(A,\eta,G)\frac{1_E}{\mathbb{P}(E)}\langle G, (\Gamma_\eta^{-1}-K^{-1})G\rangle\big|G\Big]\Bigg]
\end{multline*}

\begin{multline*}
\le \frac{1}{2}{\mathbb{E}\Big[\Big|h_4(A,\eta,G)\Big|g(A,\eta,G)1_\mathbb{E}\Big]}\max{\Big\{\Big|\log\frac{2^d det K}{(2\|K\|_{op}+\lambda(K))^d}\Big|,\log\frac{2^d detK}{\lambda(K)^d}\Big\}}\\
-\frac{1}{2}\mathbb{E}\Bigg[\frac{\mathbb{E}\Big[h_4(A,\eta,G)g(A,\eta,G)1_\mathbb{E}\big|G\Big]}{\mathbb{E}\Big[g(A,\eta,G)\frac{1_E}{\mathbb{P}(E)}\big|G\Big]}\mathbb{E}\Big[g(A,\eta,G)\frac{1_E}{\mathbb{P}(E)}\big|G\Big]\cdot\\
\cdot \mathbb{E}\Bigg[\frac{g(A,\eta,G)\frac{1_E}{\mathbb{P}(E)}}{\mathbb{E}\Big[g(A,\eta,G)\frac{1_E}{\mathbb{P}(E)}\big|G\Big]}\langle G, (\Gamma_\eta^{-1}-K^{-1})G\rangle\Big|G\Bigg]\Bigg],
\end{multline*}
where we used the fact that, on the event $E$ defined in \eqref{def_ev_E}, $\Big(\frac{\lambda(K)}{2}\Big)^d\le det(\Gamma_\eta)\le \Big(\frac{\lambda(K)}{2}+\|K\|_{op}\Big)^d$ for every $\eta\in [0,1]$. Applying now the Cauchy-Schwarz inequality in the second summand and then applying Jensen inequality with respect to the probability measure {whose density with respect to $\mathbb{P}[\bullet\, | \, G]$ is given by} $\frac{g(A,\eta,G)\frac{1_E}{\mathbb{P}(E)}}{\mathbb{E}\Big[g(A,\eta,G)\frac{1_E}{\mathbb{P}(E)}\big|G\Big]}$, we deduce that
\begin{multline*}
\mathbb{P}(E)\mathbb{E}\Big[\frac{\partial^4\tilde{g}}{\partial t^4}(A,\eta,G)\log(\tilde{g}(A,\eta,G))\Big]\\
\le \frac{1}{2}\max{\Big\{\Big|\log\frac{2^d det K}{(2\|K\|_{op}+\lambda(K))^d}\Big|,\log\frac{2^d detK}{\lambda(K)^d}\Big\}}\mathbb{E}\Big[{|h_4(A,\eta,G)|^2g(A,\eta,G){1_E}}\Big]^{1/2}\\
+\frac{\mathbb{P}(E)}{2}\mathbb{E}\Bigg[\frac{\mathbb{E}\Big[|h_4(A,\eta,G)|^2g(A,\eta,G)\frac{1_E}{\mathbb{P}(E)}\big|G\Big]}{\mathbb{E}\Big[g(A,\eta,G)\frac{1_E}{\mathbb{P}(E)}\big|G\Big]}\mathbb{E}\Big[g(A,\eta,G)\frac{1_E}{\mathbb{P}(E)}\big|G\Big]\Bigg]^{1/2}\cdot\\
\cdot \mathbb{E}\Bigg[\mathbb{E}\Big[g(A,\eta,G)\frac{1_E}{\mathbb{P}(E)}\big|G\Big]\mathbb{E}\Bigg[\frac{g(A,\eta,G)\frac{1_E}{\mathbb{P}(E)}}{\mathbb{E}\Big[g(A,\eta,G)\frac{1_E}{\mathbb{P}(E)}\big|G\Big]}\Big(\langle G, (\Gamma_\eta^{-1}-K^{-1})G\rangle\Big)^2\Big|G\Bigg]\Bigg]^{1/2}
\end{multline*}

\begin{multline*}
\le \frac{1}{2}\max{\Big\{\Big|\log\frac{2^d det K}{(2\|K\|_{op}+\lambda(K))^d}\Big|,\log\frac{2^d detK}{\lambda(K)^d}\Big\}}\mathbb{E}\Big[|h_4(A,\eta,F_\eta)|^2{1_E}\Big]^{1/2}\\
+\frac{1}{2}\mathbb{E}\Big[|h_4(A,\eta,F_\eta)|^2{1_E}\Big]^{1/2} \mathbb{E}\Bigg[{1_E}\Big(\langle F_\eta, (\Gamma_\eta^{-1}-K^{-1})F_\eta\rangle\Big)^2\Bigg]^{1/2}
\end{multline*}

\begin{multline*}
\le \Bigg(\frac{3}{2}\mathbb{E}\Big[\Big(tr\big((\Gamma^{-1}_\eta(A-K))^2\big)\Big)^41_\mathbb{E}\Big]+18\mathbb{E}\Big[\Big(tr\big((\Gamma^{-1}_\eta(A-K))^2\big)\Big)^2tr\big((\Gamma^{-1}_\eta(A-K))^4\big)1_\mathbb{E}\Big]\\
+18\mathbb{E}\Big[\Big(tr\big((\Gamma^{-1}_\eta(A-K))^4\big)\Big)^21_\mathbb{E}\Big]
+72\mathbb{E}\Big[tr\big((\Gamma^{-1}_\eta(A-K))^8\big)1_\mathbb{E}\Big]\\
+48\mathbb{E}\Big[tr\big((\Gamma^{-1}_\eta(A-K))^2\big)tr\big((\Gamma^{-1}_\eta(A-K))^6\big)1_\mathbb{E}\Big]\Bigg)^{1/2}\cdot\\
	\cdot\Bigg(\frac{1}{2}\max{\Big\{\Big|\log\frac{2^d det K}{(2\|K\|_{op}+\lambda(K))^d}\Big|,\log\frac{2^d detK}{\lambda(K)^d}\Big\}}+\frac{\sqrt{2}}{2}\mathbb{E}\Big[\|I_d-\sqrt{\Gamma_\eta}K^{-1}\sqrt{\Gamma_\eta}\|_{HS}^21_\mathbb{E}\Big]^{1/2}\\
+\frac{1}{2}\mathbb{E}\Big[\Big(tr(I_d-\sqrt{\Gamma_\eta}K^{-1}\sqrt{\Gamma_\eta})\Big)^21_\mathbb{E}\Big]^{1/2}\Bigg),
\end{multline*}
where we have exploited identity \eqref{h4_lemma}. Observe that, thanks to the fact that on the event $E$ one has that $\|A-K\|_{op}\le\frac{\lambda(K)}{2}$, for every $\eta\in [0,1]$ 
\begin{multline*}
\mathbb{E}\Big[\|I_d-\sqrt{\Gamma_\eta}K^{-1}\sqrt{\Gamma_\eta}\|_{HS}^21_\mathbb{E}\Big]=\mathbb{E}\Big[\|I_d-{\Gamma_\eta}K^{-1}\|_{HS}^21_\mathbb{E}\Big]=\mathbb{E}\Big[\|(K-{\Gamma_\eta})K^{-1}\|_{HS}^21_\mathbb{E}\Big]\\
\le \|K^{-1}\|_{HS}^2\mathbb{E}\Big[\|A-K\|_{HS}^21_\mathbb{E}\Big]\le  \|K^{-1}\|_{HS}^2\frac{d\lambda(K)^2}{4}
\end{multline*}
and 
\begin{multline*}
\mathbb{E}\Big[\Big(tr(I_d-\sqrt{\Gamma_\eta}K^{-1}\sqrt{\Gamma_\eta})\Big)^21_\mathbb{E}\Big]=\mathbb{E}\Big[\Big(tr\big((K-\Gamma_\eta)K^{-1}\big)\Big)^21_\mathbb{E}\Big]\\
\le \|K^{-1}\|_{HS}^2\mathbb{E}\Big[\|A-K\|_{HS}^21_\mathbb{E}\Big]
\le \|K^{-1}\|_{HS}^2\frac{d\lambda(K)^2}{4}.
\end{multline*}
As a consequence,
\begin{multline*}
\mathbb{P}(E)\mathbb{E}\Big[\frac{\partial^4\tilde{g}}{\partial t^4}(A,\eta,G)\log(\tilde{g}(A,\eta,G))\Big]\\
\le \frac{3\sqrt{70}}{2}\mathbb{E}\Big[\|\Gamma^{-1}_\eta(A-K)\|_{HS}^81_\mathbb{E}\Big]^{1/2}\Bigg(\frac{1}{2}\max{\Big\{\Big|\log\frac{2^d det K}{(2\|K\|_{op}+\lambda(K))^d}\Big|,\log\frac{2^d detK}{\lambda(K)^d}\Big\}}\\
+\frac{\sqrt{d}(\sqrt{2}+1)}{4}\|K^{-1}\|_{HS}\lambda(K)\Bigg).
\end{multline*}
yielding the desired conclusion.

\subsubsection{Proof of Lemma \ref{stime_per_th_fin}}
By Remark \ref{pass_sotto_deriv}, definitions \eqref{def_h_k} and \eqref{tilde_g} and recalling that $A$ is assumed to be independent of $G$,
\begin{multline*}
\mathbb{P}(E)\mathbb{E}\Big[\frac{1}{(\tilde{g}(A,\eta,G))^k}\frac{\partial^i \tilde{g}}{\partial t^i}(A,\eta,G)\Big(\frac{\partial^j\tilde{g}}{\partial t^j}(A,\eta,G)\Big)^k\Big]\\
=\mathbb{P}(E)\mathbb{E}\Bigg[\frac{1}{\mathbb{E}\Big[{g}(A,\eta,G)\frac{1_E}{\mathbb{P}(E)}\big| G\Big]^k}\mathbb{E}\Big[\frac{\partial^i {g}}{\partial t^i}(A,\eta,G)\frac{1_E}{\mathbb{P}(E)}\big| G\Big]\mathbb{E}\Big[\frac{\partial^j {g}}{\partial t^j}(A,\eta,G)\frac{1_E}{\mathbb{P}(E)}\big| G\Big]^k\Bigg]
\end{multline*}
\begin{multline*}
=\mathbb{P}(E)\mathbb{E}\Bigg[{\mathbb{E}\Big[{g}(A,\eta,G)\frac{1_E}{\mathbb{P}(E)}\big| G\Big]}\mathbb{E}\Bigg[h_i(A,\eta,G)\frac{g(A,\eta,G)\frac{1_E}{\mathbb{P}(E)}}{\mathbb{E}\Big[{g}(A,\eta,G)\frac{1_E}{\mathbb{P}(E)}\big| G\Big]}\Big| G\Bigg]\cdot\\
\cdot \mathbb{E}\Bigg[h_j(A,\eta,G)\frac{g(A,\eta,G)\frac{1_E}{\mathbb{P}(E)}}{\mathbb{E}\Big[{g}(A,\eta,G)\frac{1_E}{\mathbb{P}(E)}\big| G\Big]}\Big| G\Bigg]^k\Bigg]
\end{multline*}
\begin{multline}\label{CSrad}
\le \mathbb{P}(E)\mathbb{E}\Bigg[{\mathbb{E}\Big[{g}(A,\eta,G)\frac{1_E}{\mathbb{P}(E)}\big| G\Big]}\mathbb{E}\Bigg[h_i(A,\eta,G)\frac{g(A,\eta,G)\frac{1_E}{\mathbb{P}(E)}}{\mathbb{E}\Big[{g}(A,\eta,G)\frac{1_E}{\mathbb{P}(E)}\big| G\Big]}\Big| G\Bigg]^2\Bigg]^{1/2}\cdot\\
\cdot \mathbb{E}\Bigg[\mathbb{E}\Big[{g}(A,\eta,G)\frac{1_E}{\mathbb{P}(E)}\big| G\Big]\mathbb{E}\Bigg[h_j(A,\eta,G)\frac{g(A,\eta,G)\frac{1_E}{\mathbb{P}(E)}}{\mathbb{E}\Big[{g}(A,\eta,G)\frac{1_E}{\mathbb{P}(E)}\big|G\Big]}\Big| G\Bigg]^{2k}\Bigg]^{1/2}
\end{multline}
\begin{equation}\label{Jen_hij}
\le \mathbb{E}\Bigg[(h_i(A,\eta,G))^2{g(A,\eta,G){1_E}}\Bigg]^{1/2} \mathbb{E}\Bigg[(h_j(A,\eta,G))^{2k}{g(A,\eta,G){1_E}}\Bigg]^{1/2},
\end{equation}
using Cauchy-Schwarz inequality in \eqref{CSrad} and Jensen inequality with respect to the probability { whose density with respect to $\mathbb{P} [\bullet \, |\, G]$ is given by} $\frac{g(A,\eta,G)\frac{1_E}{\mathbb{P}(E)}}{\mathbb{E}\Big[{g}(A,\eta,G)\frac{1_E}{\mathbb{P}(E)}\big|G\Big]}$ to obtain (\ref{Jen_hij}). From inequality \eqref{gen_for_prod} and Lemma \ref{mom_sec_h} it follows that 
\begin{multline*}
\mathbb{P}(E)\mathbb{E}\Big[\frac{1}{\tilde{g}(A,\eta,G)}\Big(\frac{\partial^2 \tilde{g}}{\partial t^2}(A,\eta,G)\Big)^2\Big]\le  \mathbb{E}\Big[(h_2(A,\eta,G))^2{g(A,\eta,G){1_E}}\Big]\\
=\mathbb{E}[tr((\Gamma_\eta^{-1}(A-K))^4)1_E]+\frac{1}{2}\mathbb{E}[(tr((\Gamma_\eta^{-1}(A-K))^2))^21_E]\le \frac{3}{2}\mathbb{E}\Big[\|\Gamma_\eta^{-1}(A-K)\|_{HS}^41_\mathbb{E}\Big],
\end{multline*}

\begin{multline*}
\mathbb{P}(E)\mathbb{E}\Big[\frac{1}{\tilde{g}(A,\eta,G)}\frac{\partial^3 \tilde{g}}{\partial t^3}(A,\eta,G)\frac{\partial\tilde{g}}{\partial t}(A,\eta,G)\Big]\\
\le \mathbb{E}\Big[(h_3(A,\eta,G))^2{g(A,\eta,G){1_E}}\Big]^{1/2}\mathbb{E}\Big[(h_1(A,\eta,G))^2{g(A,\eta,G){1_E}}\Big]^{1/2}
\end{multline*}

\begin{multline*}
=\Bigg(\frac{3}{4}\mathbb{E}\Big[\Big(tr\big((\Gamma^{-1}_\eta(A-K))^2\big)\Big)^31_\mathbb{E}\Big]+\frac{9}{2}\mathbb{E}\Big[tr\big((\Gamma^{-1}_\eta(A-K))^2\big)tr\big((\Gamma^{-1}_\eta(A-K))^4\big)1_\mathbb{E}\Big]\\
+6\mathbb{E}\Big[tr\big((\Gamma^{-1}_\eta(A-K))^6\big)1_\mathbb{E}\Big]\Bigg)^{1/2}\Bigg(\frac{1}{2}\mathbb{E}[tr((\Gamma_\eta^{-1}(A-K))^2)1_E]\Bigg)^{1/2}\\
\le\frac{3 \sqrt{5}}{2\sqrt{2}}\mathbb{E}\Big[\|\Gamma^{-1}_\eta(A-K)\|_{HS}^61_\mathbb{E}\Big]^{1/2}\mathbb{E}\Big[\|\Gamma^{-1}_\eta(A-K)\|_{HS}^21_\mathbb{E}\Big]^{1/2},
\end{multline*}

\begin{multline*}
\mathbb{P}(E)\mathbb{E}\Big[\frac{1}{(\tilde{g}(A,\eta,G))^2}\frac{\partial^2 \tilde{g}}{\partial t^2}(A,\eta,G)\Big(\frac{\partial\tilde{g}}{\partial t}(A,\eta,G)\Big)^2\Big]\\
 \le \mathbb{E}\Bigg[(h_2(A,\eta,G))^2{g(A,\eta,G){1_E}}\Bigg]^{1/2} \mathbb{E}\Bigg[(h_1(A,\eta,G))^4{g(A,\eta,G){1_E}}\Bigg]^{1/2}\\
=\Bigg(\mathbb{E}[tr((\Gamma_\eta^{-1}(A-K))^4)1_E]+\frac{1}{2}\mathbb{E}[(tr((\Gamma_\eta^{-1}(A-K))^2))^21_E]\Bigg)^{1/2}\cdot\\
\cdot \Bigg(3\mathbb{E}[tr((\Gamma_\eta^{-1}(A-K))^4)1_E]+\frac{3}{4}\mathbb{E}[(tr((\Gamma_\eta^{-1}(A-K))^2))^21_E]\Bigg)^{1/2}
\le \frac{3\sqrt{5}}{2\sqrt{2}}\mathbb{E}\Big[\|\Gamma_\eta^{-1}(A-K)\|_{HS}^41_\mathbb{E}\Big],
\end{multline*}

and
\begin{multline*}
\mathbb{P}(E)\mathbb{E}\Big[\frac{1}{(\tilde{g}(A,\eta,G))^3}\Big(\frac{\partial\tilde{g}}{\partial t}(A,\eta,G)\Big)^4\Big]\\
\le \mathbb{E}\Bigg[(h_1(A,\eta,G))^2{g(A,\eta,G){1_E}}\Bigg]^{1/2} \mathbb{E}\Bigg[(h_1(A,\eta,G))^6{g(A,\eta,G){1_E}}\Bigg]^{1/2}\\
\le \Bigg(\frac{1}{2}\mathbb{E}[tr((\Gamma_\eta^{-1}(A-K))^2)1_E]\Bigg)^{1/2}\Bigg(60 \mathbb{E}\Big[tr((\Gamma_\eta^{-1}(A-K))^6)1_\mathbb{E}\Big]+\frac{15}{8}\mathbb{E}\Big[(tr((\Gamma_\eta^{-1}(A-K))^2))^31_\mathbb{E}\Big]\\
+10\mathbb{E}\Big[(tr((\Gamma_\eta^{-1}(A-K))^3))^21_\mathbb{E}\Big]
+\frac{45}{2}\mathbb{E}\Big[tr((\Gamma_\eta^{-1}(A-K))^2)tr((\Gamma_\eta^{-1}(A-K))^4)1_\mathbb{E}\Big]\Bigg)^{1/2}\\
\le \frac{\sqrt{5}}{4}(10+\sqrt{3}+4\sqrt{6})\mathbb{E}\Big[\|\Gamma^{-1}_\eta(A-K)\|_{HS}^61_\mathbb{E}\Big]^{1/2}\mathbb{E}\Big[\|\Gamma^{-1}_\eta(A-K)\|_{HS}^21_\mathbb{E}\Big]^{1/2}.
\end{multline*}

\subsection{Proof of Lemma \ref{term_con_zero}}
Let us observe that
\begin{multline*}
\mathbb{E}\Big[h_1(A,0,x)\frac{1_E}{\mathbb{P}(E)}\Big]=\frac{1}{2}\langle x,K^{-1}\Big(\mathbb{E}\Big[A\frac{1_E}{\mathbb{P}(E)}\Big]-K\Big)K^{-1}x\rangle-\frac{1}{2}tr\Big(K^{-1}\Big(\mathbb{E}\Big[A\frac{1_E}{\mathbb{P}(E)}\Big]-K\Big)\Big)\\
=\frac{1}{\phi_K(x)}\frac{\partial}{\partial t}\Big( \phi_{\mathbb{E}[\Gamma_t\frac{1_E}{\mathbb{P}(E)}]}(x)\Big)_{\Big|_{t=0}}
\end{multline*}
and therefore it is possible to easily adapt the results from Lemma \ref{mom_sec_h}, replacing the matrix $A$ with $\mathbb{E}\Big[A\frac{1_E}{\mathbb{P}(E)}\Big]$, which is invertible when $\mathbb{P}(E)\ne 0$ because of inequality \eqref{aut_A} applied inside the expectation. More precisely,
\begin{equation}\label{h1_gen2}
\mathbb{E}\Big[\Big(\frac{\partial \tilde{g}}{\partial t}(A,0,G)\Big)^2\Big]=\frac{1}{2}\mathbb{E}\Bigg[tr\Bigg(\Big(K^{-1}\Big(\mathbb{E}\Big[A\frac{1_E}{\mathbb{P}(E)}\Big]-K\Big)\Big)^2\Bigg)\Bigg]=\frac{1}{2}\Big\|\Big(\mathbb{E}\Big[A\frac{1_E}{\mathbb{P}(E)}\Big]-K\Big)K^{-1}\Big\|_{HS}^2
\end{equation}
and 
\begin{multline}\label{h4_gen2}
\mathbb{E}\Big[\Big(\frac{\partial \tilde{g}}{\partial t}(A,0,G)\Big)^4\Big]=
3\mathbb{E}\Bigg[tr\Bigg(\Big(K^{-1}\Big(\mathbb{E}\Big[A\frac{1_E}{\mathbb{P}(E)}\Big]-K\Big)\Big)^4\Bigg)\Bigg]\\
+\frac{3}{4}\mathbb{E}\Bigg[\Bigg(tr\Bigg(\Big(K^{-1}\Big(\mathbb{E}\Big[A\frac{1_E}{\mathbb{P}(E)}\Big]-K\Big)\Big)^2\Bigg)\Bigg)^2\Bigg],\\
\end{multline}
Hence, using Cauchy-Schwarz,
\[
\mathbb{P}(E)\mathbb{E}\Big[\Big(\frac{\partial \tilde{g}}{\partial t}(A,0,G)\Big)^3\Big]\le\mathbb{P}(E)\mathbb{E}\Big[\Big(\frac{\partial \tilde{g}}{\partial t}(A,0,G)\Big)^2\Big]^{1/2}\mathbb{E}\Big[\Big(\frac{\partial \tilde{g}}{\partial t}(A,0,G)\Big)^4\Big]^{1/2} 
\]
and therefore, using \eqref{h1_gen2}, \eqref{h4_gen2} and the fact that for any symmetrical matrix $M$ one has that $tr(M^4)\le\|M\|_{op}^2\|M\|_{HS}^2$ and $\Big(tr(M^2)\Big)^2\le d\|M\|^2_{op}\|M\|_{HS}^2$,
\begin{multline*}
\mathbb{P}(E)\mathbb{E}\Big[\Big(\frac{\partial \tilde{g}}{\partial t}(A,0,G)\Big)^3\Big]\le\frac{\sqrt{3}\mathbb{P}(E)}{\sqrt{2}}\Big(1+\frac{\sqrt{d}}{{2}}\Big)\Big\|\Big(\mathbb{E}\Big[A\frac{1_E}{\mathbb{P}(E)}\Big]-K\Big)K^{-1}\Big\|_{HS}^2\cdot\\
\cdot\Big\|K^{-1}\Big(\mathbb{E}\Big[A\frac{1_E}{\mathbb{P}(E)}\Big]-K\Big)\Big]\Big\|_{op},
\end{multline*}
where 
\[
\Big\|K^{-1}\Big(\mathbb{E}\Big[A\frac{1_E}{\mathbb{P}(E)}\Big]-K\Big)\Big]\Big\|_{op}\le \|K^{-1}\|_{op}\mathbb{E}\Big[\|A-K\|_{op}\frac{1_E}{\mathbb{P}(E)}\Big]\le \frac{\lambda(K)}{2}\|K^{-1}\|_{op}= \frac{1}{2}
\]
and we have used that $\|A-K\|_{op}\le\frac{\lambda(K)}{2}$ on the event $E$. It follows that
\[
\mathbb{P}(E)\mathbb{E}\Big[\Big(\frac{\partial \tilde{g}}{\partial t}(A,0,G)\Big)^3\Big]\le \frac{\sqrt{3}\mathbb{P}(E)}{2\sqrt{2}}\Big(1+\frac{\sqrt{d}}{{2}}\Big)\Big\|\Big(\mathbb{E}\Big[A\frac{1_E}{\mathbb{P}(E)}\Big]-K\Big)K^{-1}\Big\|_{HS}^2.
\]
Finally, in order to deal with the last term, we apply again Cauchy-Schwarz to infer that
\[
\mathbb{P}(E)\mathbb{E}\Big[\frac{\partial^2 \tilde{g}}{\partial t^2}(A,0,G)\frac{\partial\tilde{g}}{\partial t}(A,0,G)\Big]\le \mathbb{P}(E)\mathbb{E}\Big[\Big(\frac{\partial^2 \tilde{g}}{\partial t^2}(A,0,G)\Big)^2\Big]^{1/2}\mathbb{E}\Big[\Big(\frac{\partial\tilde{g}}{\partial t}(A,0,G)\Big)^2\Big]^{1/2}
\]

\[
\le\frac{\sqrt{3}\mathbb{P}(E)}{{2}}\Big\|K^{-1}\Big(\mathbb{E}\Big[A\frac{1_E}{\mathbb{P}(E)}\Big]-K\Big)\Big\|_{HS}\mathbb{E}\Big[\|K^{-1}(A-K)\|_{HS}^4\Big]^{1/2}
\]
\[
\le \frac{\sqrt{3}\mathbb{P}(E)}{{4}}\Big\|K^{-1}\Big(\mathbb{E}\Big[A\frac{1_E}{\mathbb{P}(E)}\Big]-K\Big)\Big\|_{HS}^2+\frac{\sqrt{3}\mathbb{P}(E)}{{4}} \mathbb{E}\Big[\|K^{-1}(A-K)\|_{HS}^4\Big]
\]
thanks to the bound \eqref{h1_gen2}, to the identity \eqref{h_quad} with $\eta=0$ and to the fact that $2ab\le a^2+b^2$ for every $a,b\in\mathbb{R}$. To conclude, it is now sufficient to study the following quantity:
\begin{multline*}
\frac{\mathbb{P}(E)}{2}\Big\|\Big(\mathbb{E}\Big[A\frac{1_E}{\mathbb{P}(E)}\Big]-K\Big)K^{-1}\Big\|_{HS}^2\\
\le  {\mathbb{P}(E)}\|(\mathbb{E}[A]-K)K^{-1}\|_{HS}^2+{\mathbb{P}(E)}\Big\|\mathbb{E}\Big[A\Big(\frac{1_E}{\mathbb{P}(E)}-1\Big)\Big]K^{-1}\Big\|_{HS}^2
\end{multline*}

\[
\le  {\mathbb{P}(E)}\|(\mathbb{E}[A]-K)K^{-1}\|_{HS}^2+{\mathbb{P}(E)}\mathbb{E}\Big[\|A\|_{HS}\Big|\frac{1_E}{\mathbb{P}(E)}-1\Big|\Big]^2\|K^{-1}\|_{HS}^2
\]
\[
\le {\mathbb{P}(E)}\|(\mathbb{E}[A]-K)K^{-1}\|_{HS}^2+{\mathbb{P}(E)}\mathbb{E}\Big[\|A\|_{HS}^2\Big]\mathbb{E}\Big[\Big|\frac{1_E}{\mathbb{P}(E)}-1\Big|^2\Big]\|K^{-1}\|_{HS}^2
\]
\[
\le {\mathbb{P}(E)}\|(\mathbb{E}[A]-K)K^{-1}\|_{HS}^2+2\mathbb{E}\Big[\|A-K\|_{HS}^2\Big]\|K^{-1}\|_{HS}^2{\mathbb{P}(E^C)}+2\|K\|_{HS}^2\|K^{-1}\|_{HS}^2{\mathbb{P}(E^C)}
\]
\begin{multline}\label{uso_markov}
\le {\mathbb{P}(E)}\|(\mathbb{E}[A]-K)K^{-1}\|_{HS}^2+\frac{8}{\lambda(K)^2}\mathbb{E}\Big[\|A-K\|_{HS}^2\Big]\|K^{-1}\|_{HS}^2\mathbb{E}\Big[\|A-K\|_{op}^2\Big]\\
+\frac{32}{\lambda(K)^4}\|K\|_{HS}^2\|K^{-1}\|_{HS}^2\mathbb{E}\Big[\|A-K\|_{op}^4\Big]
\end{multline}
where we have used Markov's inequality to bound $\mathbb{P}(E^C)=\mathbb{P}\Big(\|A-K\|_{op}>\frac{\lambda(K)}{2}\Big)$ to obtain \eqref{uso_markov}, recalling that $\lambda(K)$ is defined as the minimum eigenvalue of $K$. As a consequence, 
\begin{multline*}
\frac{\mathbb{P}(E)}{2}\Big\|\Big(\mathbb{E}\Big[A\frac{1_E}{\mathbb{P}(E)}\Big]-K\Big)K^{-1}\Big\|_{HS}^2
\le \|\mathbb{E}[A]-K\|_{HS}^2\|K^{-1}\|_{HS}^2\\+\frac{8}{\lambda(K)^2}\mathbb{E}\Big[\|A-K\|_{HS}^2\Big]^2\|K^{-1}\|_{HS}^2
+\frac{32}{\lambda(K)^4}\|K\|_{HS}^2\|K^{-1}\|_{HS}^2\mathbb{E}\Big[\|A-K\|_{HS}^4\Big],
\end{multline*}
and the proof is concluded.
\end{document}